%% file: intcurves.tex
\documentclass[11pt]{amsart}
\usepackage[titletoc]{appendix}
\usepackage[foot]{amsaddr}

\calclayout

\input{header-en}

\title{Algebraic independence for values of integral curves}
\author{Tiago J. Fonseca}


\address{Max-Planck-Institut für Mathematik,
        Vivatsgasse 7, 53111 Bonn, Germany.}
\email{fonseca@mpim-bonn.mpg.de}

\date{\today}

\begin{document}
\maketitle

\begin{abstract}
  We prove a transcendence theorem concerning values of holomorphic maps from a disk to a quasi-projective variety over $\overline{\mathbf{Q}}$ that are integral curves of some algebraic vector field (defined over $\overline{\mathbf{Q}}$).  These maps are required to satisfy some integrality property, besides a growth condition and a strong form of Zariski-density that are natural for integral curves of algebraic vector fields.


  This result generalizes a theorem of Nesterenko concerning algebraic independence of values of the Eisenstein series $E_2,E_4,E_6$. The main technical improvement in our approach is the replacement of a rather restrictive hypothesis of polynomial growth on Taylor coefficients by a geometric notion of \emph{moderate growth} formulated in terms of Value Distribution Theory.  
\end{abstract}

\tableofcontents

\section{Introduction}

\subsection{A theorem of Nesterenko}

This work was motivated by questions related to the following algebraic independence result.

Let $E_2$, $E_4$, and $E_6$ be the classical Eisenstein series, seen as holomorphic functions on the complex unit disk $D \defeq \{q \in \CC \mid |q|<1\}$, explicitly defined by
\begin{align*}
E_{2}(q) = 1 - 24 \sum_{j=1}^{\infty}\sigma_1(j)q^j\text{, }\ \ E_{4}(q) = 1+240 \sum_{j=1}^{\infty}\sigma_3(j)q^j\text{, } \ \ 
E_6(q) = 1-504 \sum_{j=1}^{\infty}\sigma_5(j)q^j
\end{align*}
for every $q \in D$, where $\sigma_k(j) \defeq \sum_{d\mid j}d^k \in \ZZ$. Let us also consider the $q$-expansion of the $j$-invariant
\begin{align*}
J(q) = 1728\frac{E_4(q)^3}{E_4(q)^3 -E_6(q)^2} = \frac{1}{q} + 744 + \sum_{j=1}^{\infty}c(j)q^j\text{.}
\end{align*}

\begin{theorem}[Nesterenko \cite{nesterenko96}]\label{thmnes}
  For every $z \in D\minus\{0\}$, we have
  \begin{align*}
   \trdeg_{\QQ}\QQ(z,E_2(z),E_4(z),E_6(z))\ge 3\text{.}
  \end{align*}
\end{theorem}

This result is an improvement of Barré-Sirieix's, Diaz's, Gramain's, and Philibert's breakthrough \cite{stephanois96} concerning the solution of a conjecture of Mahler: for every algebraic $z \in D\minus\{0\}$, $J(z)$ is transcendental. 

In order to fully motivate our contributions, we next sketch the main steps of Nesterenko's original proof.

In view of an algebraic independence criterion due to Philippon (\cite{philippon86} Théorème 2.11; see also \cite{nesterenko96} Lemma 2.5), it suffices to construct a sequence of polynomials with integral coefficients $Q_n \in \ZZ[X_0,X_1,X_2,X_3]$, for $n \gg 0$, such that $\deg Q_n = O(n\log n)$, $\log \|Q_n\|_{\infty} = O(n\log^2n)$ --- here, $\|Q_n\|_{\infty}$ denotes the maximum of the absolute values of all the coefficients of $Q_n$ ---, and
\begin{align*}
   -an^4 \le \log |Q_n(z, E_2(z),E_4(z),E_6(z))|  \le - b n^4
\end{align*}
for some real constants $a>b>0$.

For this, Nesterenko implemented a \emph{method} benefiting from the fact that $E_2$, $E_4$, and $E_6$ have integral Taylor coefficients in their $q$-expansion and satisfy the so-called \emph{Ramanujan equations}:
\begin{align*}
q\frac{d E_2}{dq} = \frac{E_2^2 - E_4}{12}\text{, }\ \ q\frac{d E_4}{dq} = \frac{E_2E_4 - E_6}{3}\text{, } \ \ q\frac{d E_6}{dq} = \frac{E_2E_6 - E_4^2}{2}
\text{.}
\end{align*}
It is also essential in his construction that
\begin{enumerate}
  \item[(i)] [\emph{Growth condition}] for each $k \in \{1,2,3\}$, the sequence of Taylor coefficients $(E_{2k}^{(j)}(0)/j!)_{j\ge 0}$ grows polynomialy in $j$, and
  \item[(ii)] [\emph{Zero Lemma}]\footnote{In Diophantine Approximation and Transcendental Number Theory, ``Zero Lemma'' is an umbrella term covering several auxiliary results involving estimates of number of zeros in a certain region, or zeros multiplicities, of polynomials composed with analytic functions.} there exists a constant $C>0$ such that
    \begin{align*}
     \ord_{q=0}P(q,E_2(q),E_4(q),E_6(q))\le C(\deg P)^4
    \end{align*}
    for every non-zero polynomial $P \in \CC[X_0,X_1,X_2,X_3] \minus\{0\}$.
\end{enumerate}
The first condition can be easily deduced from the explicit description of the Taylor coefficients of $E_{2k}$ given above. The second, which may be regarded as a strong form of algebraic independence between the functions $q$, $E_2(q)$, $E_4(q)$, and $E_6(q)$, is a non-trivial consequence of Nesterenko's \emph{$D$-property} (Definition \ref{dprop} below; cf. \cite{nesterenko96} Paragraph 6), an algebraic property concerning the global behavior of the foliation in $\CC^4$ induced by the vector field
\begin{align}\label{vecfield}
v \defeq x_0\frac{\partial}{\partial x_0} + \frac{(x_1^2 - x_2)}{12}\frac{\partial}{\partial x_1} + \frac{(x_1x_2-x_3)}{3}\frac{\partial}{\partial x_2} + \frac{(x_1x_3-x_2^2)}{2}\frac{\partial}{\partial x_3}\text{.}
\end{align}
A considerable part of \cite{nesterenko96} is devoted to a proof of a stronger form of the estimate in (ii).

Nesterenko's method goes as follows.
\begin{enumerate}
  \item Using that the Taylor coefficients of $E_2$, $E_4$, and $E_6$ are integers of polynomial growth (property (i) above), we may apply Siegel's Lemma (\cite{lang66} I.1 Lemma 1) to obtain \emph{auxiliary polynomials} with integral coefficients $P_n \in \ZZ[X_0,X_1,X_2,X_3]\minus \{0\}$ such that $\deg P_n =n$, $\log \|P_n\|_{\infty} =O(n\log n)$, and
\begin{align*}
\ord_{q=0} P_n(q,E_2(q),E_4(q),E_6(q)) \ge c n^4
\end{align*}
for some constant $c>0$.

\item For a fixed $z \in D \minus \{0\}$, the next step consists in proving the existence of a sequence $j_n = O(n\log n)$ and of constants $\alpha >\beta >0$ such that the composed function $f_n(q) \defeq P_n(q,E_2(q),E_4(q),E_6(q))$ satisfies
  \begin{align*}
   -\alpha n^4 \le \log |f_n^{(j_n)}(z) | \le -\beta n^4
  \end{align*}
  for $n \gg 0$. The main point for obtaining the above lower bound is that, if all the Taylor coefficients of $f_n$ at $q=z$ up to a sufficiently large order are too small, then its first non-zero Taylor coefficient at $q=0$ will have absolute value $<1$, thereby contradicting its integrality. Here, we also make essential use of property (ii) above. This is the most delicate part of the argument.

\item Finally, for $n\gg 0$, if we consider the differential operator
  \begin{align*}
    v^{[j_n]} \defeq 12^{j_n}v \circ (v-1)\circ \cdots \circ (v-(j_n-1))\text{,}
  \end{align*}
  then the Ramanujan equations imply that $Q_n \defeq v^{[j_d]}(P_n) \in \ZZ[X_0,X_1,X_2,X_3]$ satisfies
  \begin{align*}
   (12q)^{j_n}f_n^{(j_n)}(q)  = Q_n(q,E_2(q),E_4(q),E_6(q))
  \end{align*}
  for every $q \in D$. The required properties for $Q_n$ are now easily deducible from (1) and (2).
\end{enumerate}  

\subsection{A puzzling remark}

One of the most striking features of the above method is its generality.

Indeed, a close inspection of the previous arguments suggests that, if $f_1,\ldots,f_m$ are holomorphic functions on the unit disk $D$ with integral Taylor coefficients at $q=0$, satisfying some algebraic differential equations with rational coefficients, and verifying conditions akin to (i) and (ii) above, then, \emph{mutatis mutandis}, the above method applied to the system $(f_1,\ldots,f_m)$ in place of $(E_2,E_4,E_6)$ would produce another transcendence result.

This was certainly known to specialists; see, for instance, \cite{nesterenko01} Section 3, where the pertinent properties satisfied by $E_2$, $E_4$, and $E_6$ were axiomatized as above --- more generally, see Philippon's notion of \emph{$K$-functions}, introduced in \cite{philippon98}. Clearly, one may produce examples of such $f_i$ \emph{ad libitum} by algebraically manipulating Eisenstein series, but this procedure does not lead to \emph{new} transcendence results. The problem on the existence of functions $f_1,\ldots,f_m$ satisfying the above properties, but not ``related''  to classical modular forms (in some imprecise sense), was explicitly stated by Zudilin in \cite{zudilin03}.

Since the publication of \cite{COGP91} and \cite{morrison93}, it became apparent that the phenomenon of Mirror Symmetry provides a large class of functions with integral Taylor coefficients with respect to some \emph{canonical coordinate} and which satisfy natural algebraic differential equations --- see \cite{LY98}, \cite{zudilin02}, \cite{KVS06}, \cite{KR10}, \cite{DRR17} for integrality issues. In \cite{zudilin00}, Zudilin studies some candidates within Mirror Symmetry for playing the role of $(f_1,\ldots,f_m)$, but the few cases where he is able to prove all the required properties of Nesterenko's method --- those linked to elliptic curves and $K3$ surfaces --- are all of modular nature.

It becomes clear in Zudilin's work that one of the main obstructions in applying this method for such functions is condition (i) (and, in particular, that the radius of convergence is equal to 1), which is not verified in general. In a more basic level, computing radii of convergence or getting global information on the domain of definition of such functions pertaining to Mirror Symmetry is a current research problem; see \cite{KR12} for results on certain families of mirror maps. 

The following phenomenon provides further evidence that condition (i) is overly restrictive. Let $f$ be the holomorphic function on $D$ given by $f(q) = qJ(q)$ and set $\theta \defeq q\frac{d}{dq}$. Since $\QQ(J,\theta J, \theta^2J) = \QQ(E_2,E_4,E_6)$ (see, for instance, the explicit formulas in \cite{nesterenko96} Paragraph 1), it follows from Theorem \ref{thmnes} that
\begin{align*}
\trdeg_{\QQ}\QQ(z,f(z), \theta f(z), \theta^2f(z))\ge 3
\end{align*}
for any $z \in D\minus \{0\}$. However, Nesterenko's method \emph{cannot} be \emph{directly} applied to the system $(f,\theta f, \theta^2f)$ since the sequence $c(j)$ does not grows polynomialy in $j$.\footnote{Actually, $c(j) \sim \frac{e^{4\pi \sqrt{j}}}{\sqrt{2}j^{3/4}}$; see \cite{petersson32} or \cite{rademacher38}.} All the other good properties are nevertheless satisfied: $f$, $\theta f$, and $\theta^2f$ have integral Taylor coefficients, the Ramanujan equations imply that $f$ satisfies a third order algebraic differential equation with rational coefficients, and a condition similar to (ii) also holds.

This paper grew from an observation of J.-B. Bost and H. Randriambololona that the growth condition (i) in Nesterenko's method could be replaced by a geometric notion of \emph{moderate growth} formulated in terms of characteristic functions \emph{à la} Nevanlinna Theory. Besides being weaker than the growth condition in (i), which in principle enlarges the domain of application of Nesterenko's method, this geometric growth condition is preserved under some algebraic manipulations on the input functions, thereby eliminating the odd phenomenon explained in last paragraph.

We next explain our main results. Further directions and open problems are indicated below.

\subsection{Our main results: a geometric approach}

Our main theorem is a general geometric formulation of Nesterenko's method valid for arbitrary rings of algebraic integers and more general quasi-projective ambient spaces.

Let us first informally introduce the geometric notions which will replace conditions (i) and (ii) above.

\subsubsection{Moderate growth}

Let $X$ be a smooth projective variety over $\CC$, and $h$ be a $C^{\infty}$ Hermitian metric on the complex manifold $X(\CC)$. Let $\omega \defeq -\Im h$ be the positive real (1,1)-form on $X(\CC)$ associated to $h$. To fix ideas, the reader may consider the example $X = \PP_{\CC}^n$ endowed with the Fubini-Study metric, for which $\omega$ is given in homogeneous coordinates $\mathbf{z} = (z_0:\cdots: z_n)$ by
\begin{align*}
\omega = \frac{i}{2\pi}\partial \overline{\partial}\log  |\mathbf{z}|^2 = \frac{i}{2\pi} \left(\frac{\sum_{j=0}^ndz_j\wedge d\overline{z}_j}{|\mathbf{z}|^2} - \frac{\sum_{j,k=0}^n\overline{z}_jz_kdz_j\wedge d\overline{z}_k}{|\mathbf{z}|^4} \right)\text{,}
\end{align*}
where $|\mathbf{z}|^2 = \sum_{j=0}^n |z_j|^2$.

Let $R>0$ be a real number and denote by $D_R \defeq \{z \in \CC \mid |z|<R\}$ the complex disk of radius $R$ centered at the origin. One may measure the growth of an analytic map $\varphi: D_R \to X(\CC)$  as follows. For each $t \in (0,R)$, the area of the ``disk'' $\varphi(D_t)$ in $X(\CC)$ with respect to the metric $h$ is given by
\begin{align*}
  A_{\varphi}(t) \defeq \int_{D_t}\varphi^*\omega\text{.}
\end{align*}
We may then form the \emph{characteristic function}
\begin{align*}
  T_{\varphi} : (0,R) &\to \RR_{\ge 0}\\
       r &\mapsto T_{\varphi}(r) \defeq \int_{0}^r A_{\varphi}(t)d\log t\text{,}
\end{align*}
and we say that $\varphi$ has \emph{moderate growth} in $X$ if
\begin{align*}
\limsup_{r \to R-} \frac{T_{\varphi}(r)}{\log \frac{1}{1-\frac{r}{R}}} < +\infty
\end{align*}
By the compactness of $X(\CC)$, moderate growth does not depend on the choice of Hermitian metric.

When $R=1$, any analytic map $\varphi: D \to \CC^n \subset \PP^n(\CC)$ whose coordinates have Taylor coefficients of polynomial growth has moderate growth in $\PP_{\CC}^n$ (see Example \ref{polgrowth} below). Therefore, moderate growth generalizes the growth condition (i) in Nesterenko's method.

Moderate growth is nonetheless more flexible than polynomial growth on Taylor coefficients. For instance, as long as the image of $\varphi : D_R \to X(\CC)$ is Zariski-dense in $X$, moderate growth is a birational invariant in the following sense: if $f:X \to Y$ is a birational morphism between smooth projective varieties over $\CC$, then $\varphi$ has moderate growth in $X$ if and only if $f\circ \varphi$ has moderate growth in $Y$ (see Theorem \ref{funcfield} and Corollary \ref{corobirinv} below). In particular, this allows us to define, via compactifications, an unambiguous notion of moderate growth in smooth \emph{quasi-projective} varieties. 


\subsubsection{ZL-density}

Let $k$ be a field, $X$ a projective variety over $k$ of dimension $n$, and $L$ an ample line bundle on $X$. Consider a parametrized \emph{formal curve} $\hat{\varphi} : \Spf k[\![q]\!] \to X$ in $X$, i.e., $\hat{\varphi}$ is a morphism of formal $k$-schemes.

We say that $\hat{\varphi}$ is \emph{ZL-dense} in $X$ if there exists a constant $C>0$ such that, for every integer $d\ge 1$ and every non-zero global section $s \in \Gamma(X,L^{\tensor d})\minus\{0\}$, we have
\begin{align*}
  \ord_{q=0}\hat{\varphi}^*s \le C d^n\text{.}
\end{align*}
The exponent $n=\dim X$ in the above polynomial bound is the smallest possible (see Proposition \ref{propsmallest} below). Moreover, since $L$ is ample, a ZL-dense formal curve has Zariski-dense image. Thus, ZL-density may be regarded as a strong form of Zariski-density.

For a formal curve $\hat{\varphi}: \Spf k[\![q]\!] \to \AA^n_k \subset \PP^n_k$ the above notion boils down to a classical Zero Lemma property --- here, $L=\mathcal{O}_{\PP^n_k}(1)$. For instance, taking $k=\CC$, $n=4$, and $\hat{\varphi}$ defined by the system of formal series $(q,E_2(q),E_4(q),E_6(q))$, the Zero Lemma in condition (ii) above amounts to asserting that $\hat{\varphi}$ is ZL-dense in $\PP^4_{\CC}$. 

We shall prove that ZL-density does not depend on the choice of $L$ (see Proposition \ref{amplediv} below). Actually, if $X$ is only quasi-projective and the closed point $\hat{\varphi}(0)$ in the image of $\hat{\varphi}$ is a regular point of $X$, then we shall prove that ZL-density does not depend on the choice of a projective compactification of $X$ (see Corollary \ref{corostandard} below).

\subsubsection{Statement of our main theorem and proof method}

Let $K$ be a number field and $\mathcal{O}_K$ be its ring of integers. By an \emph{arithmetic scheme} over $\mathcal{O}_K$ we mean an integral scheme $\mathcal{X}$ endowed with a separated and flat morphism of finite type $\mathcal{X} \to \Spec \mathcal{O}_K$ .

The following theorem formalizes and generalizes Nesterenko's method.

\begin{theorem}\label{lethm}
  Let $\mathcal{X}$ be a quasi-projective arithmetic scheme over $\mathcal{O}_K$ of relative dimension $n\ge 2$, with smooth generic fiber $\mathcal{X}_K$, and let $\hat{\varphi}: \Spf \mathcal{O}_K[\![q]\!] \to \mathcal{X}$ be a morphism of formal $\mathcal{O}_K$-schemes such that, for every field embedding $\sigma : K \hookrightarrow \CC$, the formal curve $\hat{\varphi}_{\sigma} : \Spf \CC[\![q]\!] \to \mathcal{X}_{\sigma}$, obtained from $\hat{\varphi}$ by base change, lifts to an analytic curve $\varphi_{\sigma}: D_{R_{\sigma}}\subset \CC \to \mathcal{X}_{\sigma}^{\an}$ defined on a disk of radius $R_{\sigma}>0$ centered at the origin.

  Assume that
  \begin{align*}
\prod_{\sigma: K \into \CC} R_{\sigma}=1
  \end{align*}
  and that there exists a vector field $v \in \Gamma(\mathcal{X}_K,T_{\mathcal{X}_K/K})\minus\{0\}$ on the generic fiber of $\mathcal{X}$ such that $\hat{\varphi}_K : \Spf K[\![q]\!] \to \mathcal{X}_K$ satisfies the differential equation
  \begin{align*}
    q\frac{d\hat{\varphi}_K}{dq} = v\circ \hat{\varphi}_K\text{.}
  \end{align*}
  If, moreover,
  \begin{enumerate}
  \item the formal curve $\hat{\varphi}_K$ is ZL-dense in $\mathcal{X}_K$, and
  \item for each field embedding $\sigma: K \hookrightarrow \CC$, the analytic curve $\varphi_{\sigma}: D_{R_{\sigma}} \to \mathcal{X}_{\sigma}^{\an}$ has moderate growth,
  \end{enumerate}
  then, for every $\sigma : K \hookrightarrow \CC$, and every $z \in D_{R_{\sigma}}\smallsetminus\{0\}$, the field of definition $K(\varphi_{\sigma}(z))$ of the complex point $\varphi_{\sigma}(z)$ in $\mathcal{X}_K$ satisfies
  \begin{align*}
\trdeg_{\QQ} K(\varphi_{\sigma}(z))\ge n-1\text{.}
  \end{align*}
\end{theorem}

Let us remark that the conditions of \emph{ZL-density} and of \emph{moderate growth}, corresponding to conditions (i) and (ii) in Nesterenko's method, are actually very mild hypotheses.

For instance, ZL-density is automatic whenever $\hat{\varphi}$ is a smooth integral curve of some vector field satisfying Nesterenko's $D$-property. When the ambient space is an affine space, this is also a theorem of Nesterenko (\cite{nesterenko96} Theorem 6), which was recently extended to a geometric framework by Binyamini \cite{binyamini14}. In Appendix \ref{appzl} we explain how to slightly modify Binyamini's arguments to prove a similar statement for any smooth quasi-projective variety.

Moderate growth, in turn, is satisfied for curves having uniformly bounded derivative on the disk (endowed with the Poincaré metric; see Example \ref{bd} for a precise statement). In particular, a theorem of Brunella (\cite{FS08} Theorem 16; see also \cite{FS08} Theorem 15) implies that, for a generic one dimensional holomorphic foliation (with singularities) $\mathcal{F}$ on $\PP^n(\CC)$, any integral curve to $\mathcal{F}$ parametrized by a disk has moderate growth.

In the broader context of Transcendental Number Theory, our result may be regarded as complementary to the Siegel-Shidlovsky and Schneider-Lang theories, which also deal with algebraic independence or transcendence of values of integral curves of algebraic vector fields (see \cite{gasbarri13}, \cite{gasbarri10}, and \cite{herblot12} for general geometric formulations). Indeed, while the Siegel-Shidlovsky and Schneider-Lang criteria handle curves parametrized by \emph{parabolic} Riemann surfaces, our theorem deals with the \emph{hyperbolic} case.

Our proof of Theorem \ref{lethm} bears the same general structure of Nesterenko's method. We also start by reducing it to a diophantine approximation statement: Theorem \ref{lethm'} below. This is done via the same algebraic independence criterion of Philippon; we explain in Appendix \ref{genphil} how to generalize it to arbitrary quasi-projective varieties. The first step in the method, concerning the construction of ``auxiliary polynomials'', is replaced in our geometric framework by a construction of ``auxiliary sections'' given by Theorem \ref{auxsec} below, the proof of which makes essential use of Bost's \emph{method of slopes} in Arakelov Theory (\cite{bost95}; cf. \cite{bost01}, \cite{bost06}). The second step also involves estimating some higher order derivative --- here, our main tool is a general result comparing, for a section of a Hermitian line bundle on a disk, norms of jets at two distinct points; see Proposition \ref{proptaylor} below and its corollaries. The third and last step is essentially the same trick using the differential equation as explained above.

\subsection{Further directions and open problems}

With Theorem \ref{lethm} in hand, we may turn the puzzling remark explained above into a precise mathematical question: \emph{is there any example of application of Theorem \ref{lethm} whose resulting transcendence statement is not contained in Theorem \ref{thmnes}?} 

As promising as the potential candidates from the theory of Mirror Symmetry may seem, one must face, given our current state of knowledge, the logical possibility of a negative answer.  However, let us remark that a \emph{proof} of this fact would be as \emph{remarkable} as the discovery of a new example, since it would imply that modular functions and their derivatives are the only ones satisfying the (quite general) hypotheses of Theorem \ref{lethm}, which make no explict reference to their geometric nature in terms of moduli of elliptic curves.

This paper, specially our geometric formulation of Theorem \ref{lethm}, should also be seen as a first step in a larger program aiming to adapt and apply Nesterenko's method to the study of algebraic independence of abelian periods of genus $g>1$.

Since values of quasimodular forms can be expressed in terms of periods of elliptic curves, the theorem of Nesterenko recovers in particular a famous result of Chudnovsky \cite{chudnovsky80} establishing Grothendieck's Period Conjecture for complex multiplication elliptic curves. A family of analogs of $(E_2,E_4,E_6)$ related to the Siegel moduli spaces $\mathcal{A}_g$ was constructed in \cite{fonseca16} and \cite{fonseca17}; in this setting involving analytic functions of several variables, in addition to being easier to read the relevant properties required by Nesterenko's method from geometry, the appearance of truly geometric phenomena --- such as positive dimensional ``special subvarieties" of $\mathcal{A}_g$ --- strongly suggests a geometric approach. 

\subsection{Organization of this article}

A great effort has been done to isolate all the different techniques intervening in Nesterenko's method and to place them in their natural generality. This distillation process is aimed not only at improving the readability of our paper, but also at making these techniques suitable for other applications in Diophantine Approximation. 

Section \ref{strongzd} contains the definition of ZL-density and some of its basic properties; here we use elementary Intersection Theory. Section \ref{mgje} defines characteristic functions and moderate growth for 2-forms on a disk and contains basic versions of the jet estimates we shall need later; our main result here is Proposition \ref{proptaylor} (see also Corollary \ref{corotaylor2}). Both Sections \ref{strongzd} and \ref{mgje} are self-contained and are of independent interest. 

Section \ref{mgqp} treats the special case of moderate growth for analytic curves (as explained in this introduction) and it depends only on the beginning of Section \ref{mgje}. Its main objective is to prove that, under a non-degeneracy hypothesis, this concept is a birational invariant of the target space (Theorem \ref{funcfield}); this is essentially classical material on Nevanlinna Theory.

Section \ref{secauxsec} is devoted to the construction of ``auxiliary sections'' in a geometric context. Here, we combine the concepts of moderate growth developed in Sections \ref{mgje} and \ref{mgqp} with Bost's \emph{slope inequality} to obtain Theorem \ref{auxsec}. This section contains a review of the prerequisites in Arakelov Theory.

In Section \ref{derivesections} we explain how vector fields induce derivations on global sections of line bundles and we provide some $L^{\infty}$ estimates. This section is also self-contained and of independent interest.

Section \ref{finalproof} contains a proof of Theorem \ref{lethm}. The reader will recognize, in Lemmas \ref{lemma1}, \ref{lemma2}, and \ref{lemma3}, natural generalizations of the three steps of Nesterenko's method explained above.

Finally, Appendices \ref{genphil} and \ref{appzl} concern geometric generalizations of results of Philippon and Binyamini originally stated only for affine (or projective) spaces, as explained above. In Appendix \ref{appzl} we make use of the basic constructions of Section \ref{derivesections}.

\subsection{Acknowledgments}

This work was supported by a public grant as part of the FMJH project, and is part of my PhD thesis at Universit\'{e} Paris-Sud, Orsay, under the supervision of Jean-Benoît Bost. I thank him and Hugues Randriam for allowing me to use their preliminary non-published notes on moderate growth as a starting point for this paper. I am also grateful to Dinh Tuan Huynh for a fruitful discussion on Nevanlinna Theory and to Daniel Bertrand for pointing me out a missing reference. Finally, I thank the Max-Planck-Institut f\"{u}r Mathematik, Bonn, for the excellent working conditions in which this paper was finished.

\subsection{Terminology and notation}

\subsubsection{}

By an (algebraic) \emph{variety} over a field $k$ we mean a separated integral scheme of finite type over $k$. 

\subsubsection{}\label{defsemiample}

Recall that a line bundle $L$ on a scheme $X$ is \emph{semiample} if there exists an integer $m\ge 1$ such that $L^{\tensor m}$ is generated by its global sections. Observe that ample line bundles are semiample, and that semiampleness is preserved under pullbacks.

\subsubsection{}

A real (1,1)-form $\omega$ on a complex manifold $M$ can always be written, in local coordinates $(z_1,\ldots,z_n)$ on $M$, as
\begin{align*}
  \omega = \frac{i}{2}\sum_{k,l=1}^nh_{kl}dz_k\wedge d\overline{z}_l
\end{align*}
where $H\defeq (h_{kl})_{1\le k,l\le n}$ is a Hermitian matrix. We say that $\omega$ is \emph{positive} (resp. \emph{semipositive}) if the matrix $H$ is positive-definite (resp. positive-semidefinite). Note that semipositive (1,1)-forms are stable under pullbacks. 

\subsubsection{}

By a Hermitian line bundle $\overline{L}=(L,\| \ \|)$ on a complex manifold $M$, we mean a holomorphic line bundle $L$ on $M$ endowed with a $C^{\infty}$ Hermitian metric $\| \ \|$. If $\Theta$ denotes the curvature of the Chern connection on $L$ associated to $\| \ \|$ (locally, $\Theta = - \partial \overline{\partial}\log \|e\|^2$ where $e$ is some trivialization of $L$), then we define the \emph{Chern curvature} of $\overline{L}$ by
\begin{align*}
c_1(\overline{L}) = \frac{i}{2\pi} \Theta\text{.}
\end{align*}
This is a closed real $C^{\infty}$ $(1,1)$-form on $M$ whose class in $H^2(M,\RR)$ coincides with the first Chern class $c_1(L)$. We say that $\overline{L}$ is \emph{positive} (resp. \emph{semipositive}) if $c_1(\overline{L})$ is positive (resp. semipositive).

\subsubsection{}\label{dc}

We use the standard notation
$$
d^c = \frac{i}{4\pi}(\overline{\partial} - \partial)= \frac{1}{4\pi} \left(r\frac{\partial}{\partial r}\tensor d\theta -\frac{1}{r}\frac{\partial}{\partial \theta}\tensor dr \right)\text{,}
$$
so that $dd^c = \frac{i}{2\pi}\partial \overline{\partial}$.

\subsubsection{}
The continuous function $\log^+: \RR \to \RR$ is defined by
\begin{align*}
  \log^+x = \begin{cases}
    \log x & \text{ if }x\ge 1\\
    0 & \text{ otherwhise}.
            \end{cases}
\end{align*}

\section{ZL-dense formal curves in quasi-projective varieties}\label{strongzd}

In this section we introduce the purely algebraic concept of \emph{ZL-dense} formal curves, and we prove some of its basic properties. This notion refines the property of being Zariski-dense and isolates the content of the Zero Lemma necessary in Nesterenko's method; that is, a formal curve satisfies the Zero Lemma if and only if it is ZL-dense.

\subsection{Degree of a divisor with respect to a line bundle}

Let $k$ be a field and $X$ be a variety over $k$. 

Recall from \cite{fulton84} 2.5 that (the isomorphism class of) a line bundle $L$ on $X$ defines an additive operator
\begin{align*}
\alpha \mapsto c_1(L)\cap \alpha
\end{align*}
on the abelian group of algebraic cycles in $X$ modulo rational equivalence; if $\alpha$ is the class of a subvariety $V$ of $X$, then $c_1(L)\cap \alpha$ is by definition the class of the cycle in $V$ associated to any Cartier divisor $D$ of $V$ for which $L|_V\cong \mathcal{O}_V(D)$. The $r$-fold composition of this operator with itself is denoted by $\alpha \mapsto c_1(L)^r\cap \alpha$. 

We say that a cycle class $\alpha$ in $X$ is \emph{semipositive} if there exists an integer $m\ge 1$ such that $m\alpha$ can be represented by a non-negative cycle in $X$ (i.e., a cycle of the form $\sum_i m_i[V_i]$ with each $m_i\ge 0$). For instance, the cycle class of a Cartier divisor $D$ is semipositive if and only if some positive multiple of $D$ is linearly equivalent to an effective divisor.

\begin{lemma}\label{lemmasemiample}
Let $L$ be a semiample line bundle on $X$ (see \ref{defsemiample} for a definition). Then, for any semipositive cycle class $\alpha$ in $X$, $c_1(L)\cap \alpha$ is semipositive.
\end{lemma}  

\begin{proof}
  Let $m\ge 1$ be an integer such that $m\alpha$ is represented by the cycle $\sum_{i}m_i[V_i]$, with each $m_i\ge 0$.

  As $L$ is semiample, there exists an integer $n\ge 1$ such that $L^{\tensor n}$ is generated by global sections. In particular, for any subvariety $V$ of $X$, the line bundle $L^{\tensor n}|_V$ on $V$ admits a non-zero global section $s_V$.

 For every $i$, $c_1(L^{\tensor n})\cap [V_i]$ is the cycle class induced by the effective Cartier divisor $\div(s_{V_i})$ on $V_i$, so that  $nm\,  c_1(L)\cap \alpha = c_1(L^{\tensor n})\cap m\alpha$ is represented by the non-negative cycle $\sum_i m_i [\div(s_{V_i})]$.
\end{proof}

Still following the terminology of \cite{fulton84}, for any line bundle $L$ on $X$, and any $r$-cycle class $\alpha$ in $X$, the \emph{$L$-degree} of $\alpha$ is defined by
\begin{align*}
  \deg_L\alpha = \deg( c_1(L)^{r}\cap \alpha)\text{,}
\end{align*}
where $\deg$ denotes the degree function on zero-cycle classes. If $D$ is a Cartier divisor on $X$, then we denote by
\begin{align*}
\deg_LD = \deg(c_1(L)^{\dim X - 1}\cap [D])
\end{align*}
the $L$-degree of the cycle class $[D]$ induced by $D$.

Observe that the degree of a semipositive zero-cycle class is non-negative. In the next result, we use the following easy consequence of Lemma \ref{lemmasemiample}: if $L$ is semiample and $\alpha$ is semipositive, then $\deg_L\alpha \ge 0$ (cf. \cite{fulton84} Lemma 12.1).

\begin{prop}\label{propsemiample}
  Let $X$ be a projective variety over a field $k$ and $L$ (resp. $M$) be an ample (resp. semiample) line bundle on $X$. Then, there exists an integer $m\ge 1$ such that, for any semipositive $r$-cycle class $\alpha$ in $X$, we have
  \begin{align*}
    0 \le \deg_M\alpha \le m^r\deg_L\alpha\text{.}
  \end{align*}
\end{prop}

\begin{proof}
  Let $m\ge 1$ be an integer such that $N\defeq L^{\tensor m}\tensor M^{\vee}$ is semiample. For any $r$-cycle class $\alpha$ in $X$, we have
  \begin{align*}
   m^r\deg_L\alpha = \deg_{L^{\tensor m}}\alpha = \deg_{N\tensor M}\alpha = \sum_{s=0}^r \binom{r}{s}\deg_M(c_1(N)^{s}\cap \alpha)\text{.}
  \end{align*}
 Since $N$ is semiample and $\alpha$ is semipositive, it follows from Lemma \ref{lemmasemiample} that each $c_1(N)^s\cap\alpha$ is semipositive. As $M$ is also semiample, we conclude that each term in the right-hand side of the above equation is non-negative, so that $m^r\deg_L\alpha \ge \deg_M\alpha\ge 0$.
\end{proof}

\begin{obs}
By combining the above proposition with an induction argument in $r$, one can actually prove the following generalization. Let $X$ be a projective variety over a field $k$, $L$ be an ample line bundle on $X$, and $M$ be \emph{any} line bundle on $X$. Then there exists a constant $C>0$ such that $|\deg_M\alpha|\le C \deg_L \alpha$ for every semipositive cycle class $\alpha$ in $X$.
\end{obs}

\begin{coro}\label{comparable}
  Let $X$ be a projective variety over a field $k$. If $L$ and $M$ are ample line bundles on $X$, then there exist constants $C_1,C_2 > 0$ such that
  \begin{align*}
   C_1\deg_MD \le \deg_LD \le C_2\deg_MD
  \end{align*}
  for any effective Cartier divisor $D$ in $X$. \hfill $\blacksquare$
\end{coro}

\subsection{ZL-dense formal curves in projective varieties}

Let $X$ be an algebraic variety over a field $k$ and consider the ring of formal power series $k[\![q]\!]$. By a \emph{formal curve} in $X$ we mean a morphism of $k$-schemes $\hat{\varphi} : \Spec k[\![q]\!] \to X$, or, equivalently, a morphism of formal $k$-schemes $\hat{\varphi} : \Spf k[\![q]\!] \to X$. The $k$-point of $X$ obtained by composing the $k$-point of $\Spec k[\![q]\!]$ given by the ideal $(q)\subset k[\![q]\!]$ with $\hat{\varphi}$ is denoted by $\hat{\varphi}(0)$.

Let $\hat{\varphi} : \Spec k[\![q]\!] \to X$ be a formal curve in $X$, and $D$ be an effective Cartier divisor in $X$. We define the \emph{intersection multiplicity} of $D$ with $\hat{\varphi}$ (at $\hat{\varphi}(0)$) by
\begin{align*}
\mult_{\hat{\varphi}}D \defeq \ord_{0} \hat{\varphi}^*f\text{,}
\end{align*}
where $f \in \mathcal{O}_{X, \hat{\varphi}(0)}$ is any local equation for $D$ around $\hat{\varphi}(0)$. This clearly does not depend on the choice of $f$. The multiplicity function $\mult_{\hat{\varphi}}$  is additive and takes values in $\NN \cup\{+\infty\}$.

\begin{defi}
  Let $X$ be a projective variety of dimension $n$ over a field $k$ and let $L$ be any ample line bundle on $X$. We say that a formal curve $\hat{\varphi}: \Spec k[\![q]\!] \to X$ in $X$ is  \emph{ZL-dense} if  there exists a constant $C>0$ such that
  \begin{align}\label{boundZLdense}
   \mult_{\hat{\varphi}}D \le C (\deg_LD)^n
  \end{align}
  for every effective Cartier divisor $D$ in $X$.
\end{defi}

Observe that the choice of $L$ in the above definition is irrelevant by Corollary \ref{comparable}.

Let us remark that the exponent $n = \dim X$ intervening in the polynomial bound (\ref{boundZLdense}) is the smallest possible one:

\begin{prop}\label{propsmallest}
Let $X$ be a projective variety of dimension $n$ over $k$ endowed with an ample line bundle $L$, and $\hat{\varphi}$ be a formal curve in $X$. Then there exists $\varepsilon>0$ and sequence of effective Cartier divisors $(D_i)_{i\ge 1}$ on $X$ satisfying
  \begin{align*}
    \lim_{i\rightarrow  +\infty}\deg_L D_i  = +\infty \ \ \text{ and }\ \ \mult_{\hat{\varphi}}D_i \ge \varepsilon(\deg_LD_i)^n\text{ for every }i\ge 1\text{.} 
  \end{align*}
\end{prop}

In the above statement we allow the possibility that $\mult_{\hat{\varphi}}D_i = +\infty$ (i.e. the divisor $D_i$ vanishes identically along the formal curve $\hat{\varphi}$) by adopting the standard convention that $+\infty\ge t$ for every $t\in\RR$.

\begin{proof}
  Assume first that $X=\PP^n_k$ and $L=\mathcal{O}(1)$. Consider the natural projection $\AA^n_k\minus\{0\} \to \PP^n_k$ and lift $\hat{\varphi}$ to some $\hat{\psi}: \Spec k[\![q]\!] \to \AA^n_k\minus\{0\}$. Let $i\ge 1$ be an integer. If $P = \sum_{|I|=i}a_IX^I \in k[X_0,\ldots,X_n]$ is a homogeneous polynomial of degree $i$, seen as a regular function on $\AA^n_k\minus\{0\}$, then we may write
  $$
  \hat{\psi}^*P  = \sum_{j=0}^{\infty}\left(\sum_{|I|=i}a_Ib_{I,j}\right)q^j \in k[\![q]\!]\text{,}
  $$
  for some $b_{I,j} \in k$ depending on the coefficients of the $n+1$ formal series defining $\hat{\psi}$. Since
  \begin{align*}
   \text{card}\{I \in \NN^{n+1}\mid |I|=i\} = \binom{i+n}{n}> \frac{1}{n!}i^n\text{,}
  \end{align*}
  it follows from elementary linear algebra that there exists a non-zero homogeneous polynomial $P_i$ of degree $i$ such that  $\ord_0 \hat{\psi}^*P_i\ge \frac{1}{n!}i^n$. By considering the Cartier divisors $D_i$ on $\PP^n_k$ induced by $P_i$, we see that we may take $\varepsilon = 1/n!$ in this case.

  The general case follows from the above one by considering a finite surjective morphism $f:X \to \PP^n_k$ satisfying $f^*\mathcal{O}(1)\cong L^{m}$ for some $m\ge 1$.
\end{proof}

\begin{obs}
It follows from the above proof that the Cartier divisors $D_i$ can actually be taken in the linear system $|L^{\tensor m i}|$ for some fixed integer $m\ge 1$.
\end{obs}

Any ZL-dense formal curve $\hat{\varphi}$ in a projective variety $X$ has a dense image in the Zariski topology. Indeed, since $X$ is projective, any Zariski-closed subset of $X$ is contained in the support of some effective Cartier divisor of $X$; then, one simply remarks that ZL-density implies that $\mult_{\hat{\varphi}}D < + \infty$ for any effective Cartier divisor $D$ in $X$, so that the image of $\hat{\varphi}$ is not contained in the \mbox{support of $D$}. 

The following example shows that the converse is not true in general.

\begin{ex}[Lacunary series]
  Let $k$ be a field and $(n_i)_{i\ge 0}$ be an increasing sequence of natural numbers satisfying $\lim_{i\rightarrow  +\infty}\frac{n_{i+1}}{n_i^2} =   +\infty$.  If $h \in k[\![q]\!]$ is any formal series of the form
  \begin{align*}
    h(q) = \sum_{i\ge 0}a_i q^{n_i}\text{,} \ \ \ \ a_i\neq 0
  \end{align*}
  then the formal curve $\hat{\varphi} : \Spec k[\![q]\!] \to \PP^2_k$, given in homogeneous coordinates by \mbox{$\hat{\varphi}(q)=(1:q:h(q))$}, is not ZL-dense. Indeed, for any integer $d\ge 0$, we may consider the homogeneous polynomial of degree $n_d$
  \begin{align*}
P_d = X_0^{n_d-1}X_2 - \sum_{i=0}^d a_i X_0^{n_d-n_i}X_1^{n_i}
  \end{align*}
  so that
  \begin{align*}
   \frac{\mult_{\hat{\varphi}} \div(P_d)}{n_d^2} = \frac{n_{d+1}}{n_d^2} \to +\infty
  \end{align*}
  as $d\rightarrow  +\infty$.

  Observe that the image of $\hat{\varphi}$ is indeed Zariski-dense. By contradiction, if $C\subset \PP^2_k$ is an irreducible curve containing the image of $\hat{\varphi}$, then, for any effective Cartier divisor $D$ in $\PP^2_k$ whose support does not contains $C$, we have $\mult_{\hat{\varphi}}D = i(\hat{\varphi}(0),C\cdot D) \le \deg C \cdot \deg D$. By construction, this is absurd for $D=\div (P_d)$ and $d$ sufficiently large.
\end{ex}

For natural examples of ZL-dense formal curves we refer to Appendix \ref{appzl}.

\subsection{Reformulation in terms of sections of an ample line bundle}

Let $X$ be an algebraic variety over a field $k$ and $\hat{\varphi}: \Spec k[\![q]\!] \to X$ be a formal curve. If $L$ is any line bundle on $X$, and $s$ is a section of $L$ on a neighborhood of $\hat{\varphi}(0)$,  we may consider the vanishing order of $\hat{\varphi}^*s \in \Gamma(\Spec k[\![q]\!], \hat{\varphi}^*L)$ at $q=0$, which coincides with the intersection multiplicity of the effective Cartier divisor $\div(s)$ with $\hat{\varphi}$: 
\begin{align*}
   \ord_{0}\hat{\varphi}^*s = \mult_{\hat{\varphi}} \div(s)\text{.} 
\end{align*}
The next proposition shows that ZL-density is a condition that has to be checked only for Cartier divisors arising from sections of powers of some fixed ample line bundle.

\begin{prop}\label{amplediv}
  Let $X$ be a projective variety of dimension $n$ over a field $k$ and $L$ be an ample line bundle on $X$. A formal curve $\hat{\varphi}: \Spec k[\![q]\!] \to X$ is ZL-dense if and only if there exists a constant $C>0$ such that
  \begin{align*}
   \ord_{0}\hat{\varphi}^*s \le C d^n
  \end{align*}
  for any integer $d\ge 1$, and any $s \in \Gamma(X, L^{\tensor d}) \minus \{0\}$.
\end{prop}

\begin{proof}
  The necessity follows from the fact that, for any $s \in \Gamma(X,L^{\tensor d})\minus \{0\}$,  $\deg_L \div (s) = (\deg_LX)d$.

  To prove the sufficiency, fix any finite surjective morphism $f: X \to \PP^n_k$ such that $f^*\mathcal{O}(1)$ is isomorphic to $L^{\tensor m}$ for some $m\ge 1$. If $E$ is an effective Cartier divisor in $\PP_k^n$, then there exists a section $s \in \Gamma(\PP^n_k, \mathcal{O}(\deg_{\mathcal{O}(1)} E))$ satisfying $E = \div (s)$, so that
  \begin{align}\label{aux1}
   \mult_{f\circ \hat{\varphi}} E = \ord_{0}(f\circ \hat{\varphi})^*s = \ord_{0}\hat{\varphi}^*(f^*s) \le C (m\deg_{\mathcal{O}(1)}E)^n = Cm^n(\deg_{\mathcal{O}(1)}E)^n\text{.}
  \end{align}
  
  Let $D$ be an effective Cartier divisor in $X$. Since $f$ is finite and $\PP^n_k$ is normal, we may define the pushforward $f_*D$ by taking norms: there is an open affine covering $(U_i)_i$ of $\PP^n_k$ such that $D$ admits a local equation $h_i$ on each $f^{-1}(U_i)$, and we define $f_*D= [(\text{Norm}_f(h_i),U_i)_i]$. As $f^*f_*D-D$ is effective, we obtain
  \begin{align*}
   \mult_{\hat{\varphi}}D \le \mult_{\hat{\varphi}}f^*f_*D = \mult_{f\circ \hat{\varphi}}f_*D\text{.}
  \end{align*}
  Note that the Weil divisor associated to $f_*D$ coincides with the pushforward (of cycles) of the Weil divisor associated to $D$ (cf. \cite{fulton84} Proposition 1.4). In particular, the projection formula gives
  \begin{align*}
   \deg_{\mathcal{O}(1)}f_*D = \deg_{L^{\tensor m}}D = m^{n-1}\deg_{L}D
  \end{align*}
  so that, by (\ref{aux1}),
  \begin{align*}
    \mult_{\hat{\varphi}}D \le Cm^n (\deg_{\mathcal{O}(1)}f_*D)^n = Cm^{n^2}(\deg_LD)^n\text{.}
  \end{align*}
\end{proof}

An advantage of considering the above equivalent form of ZL-density stems from the vector space structure of the sets $\Gamma(X,L^{\tensor d})$, $d\ge 1$. In general, a formal curve $\hat{\varphi}: \Spec k[\![q]\!] \to X$ induces, for every integer $d\ge 1$, a decreasing filtration by linear subspaces $(E_d^i)_{i\ge 0}$ on the $k$-vector space $E_d \defeq \Gamma(X,L^{\tensor d})$ defined by $E_d^i\defeq \{s \in E_d \mid \ord_0\hat{\varphi}^*s\ge i\}$.

\begin{obs}
Since $E_d$ is finite dimensional, there exists $i_d\ge 1$ such that $E_d^{i_d}= \bigcap_{i\ge 0}E_d^i = \{s \in E_d \mid \hat{\varphi}^*s=0\}$.  In other words, for every $s \in E_d$ such that $\hat{\varphi}^*s \neq 0$, we have $\ord_0\hat{\varphi}^*s <i_d$. In particular, this shows that one may replace in Proposition \ref{amplediv} the condition ``for any integer $d\ge 1$" by the weaker ``for any sufficiently large integer $d$".
\end{obs}

As a first application of Proposition \ref{amplediv}, we use the filtration $(E_d^i)_{i\ge 0}$ to show that ZL-density is a geometric property.

\begin{prop}
Let $X$ be a geometrically integral projective variety over a field $k$ and \linebreak$\hat{\varphi} : \Spf k[\![q]\!] \to X$ be a formal curve. Then, for any field extension $K$ of $k$, the formal curve $\hat{\varphi}_K : \Spf K[\![q]\!] \to X_K$, obtained from $\hat{\varphi}$ by base change, is ZL-dense in $X_K$ if and only if $\hat{\varphi}$ is ZL-dense in $X$.
\end{prop}

\begin{proof}
  Let $d\ge 1$ and $i\ge 0$ be integers. Note that $E_d\tensor_kK$ may be canonically identified with $\Gamma(X_K, L_K^{\tensor d})$. Moreover, since $E_d^i$ is the kernel of the $k$-linear map
  \begin{align*}
    E_d &\to \Gamma(\Spec k[\![q]\!],\hat{\varphi}^*L^{\tensor d})\tensor_k k[\![q]\!]/(q^{i})\\
    s &\mapsto \hat{\varphi}^*s \mod q^{i}\text{,}
  \end{align*}
  we conclude that $E_d^{i}\tensor_k K = \{t \in E_d\tensor_kK \mid \ord_0\hat{\varphi}_K^*t \ge i\}$.

  In particular, for every integer $d\ge 1$, and any real number $\kappa >0$, $\ord_0\hat{\varphi}^*s\le \kappa$ for every $s \in \Gamma(X,L^{\tensor d})\minus\{0\}$ if and only if $\ord_0\hat{\varphi}^*_Kt\le \kappa$ for every $t \in \Gamma(X_K,L_K^{\tensor d})\minus\{0\}$.
\end{proof}

As another application of Proposition \ref{amplediv}, we prove the following result which will be used in our proof of Proposition \ref{propmod}.

\begin{prop}\label{surj}
Let $f:X \to Y$ be a surjective morphism between projective varieties of dimension $n$ over a field $k$ and let $\hat{\varphi}:\Spec k[\![q]\!] \to X$ be a formal curve in $X$. If $\hat{\varphi}$ is ZL-dense in $X$, then $f\circ \hat{\varphi}$ is ZL-dense in $Y$.
\end{prop}

\begin{proof}
  Let $L$ be an ample line bundle on $X$ admitting a global section $s \in \Gamma(X,L)\minus \{0\}$, and $M$ be any ample line bundle on $Y$. Since $f^*M$ is semiample, $N\defeq L \tensor f^{*}M$ is ample (\cite{EGAII} Proposition 4.5.6 (ii)).

  Let $d\ge 1$ be an integer, and $t\in \Gamma(Y,M^{\tensor d})\minus\{0\}$. Since $f$ is surjective, $f^*t\neq 0$. Thus $s^{\tensor d}\tensor f^*t$ is a non-zero global section of $N^{\tensor d}$ and, since $\hat{\varphi}$ is ZL-dense in $X$, there exists a constant $C>0$ independent of $d$ such that
  \begin{align*}
    \ord_{0}\hat{\varphi}^*(s^{\tensor d}\tensor f^*t) \le Cd^n\text{.}
  \end{align*}
To complete the proof, it is sufficient to remark that
  \begin{align*}
\ord_{0}(f\circ \hat{\varphi})^*(t) = \ord_{0}\hat{\varphi}^*(f^*t) \le \ord_{0}\hat{\varphi}^*(s^{\tensor d}\tensor f^*t)\text{.}
  \end{align*}
\end{proof}

\begin{obs}
The above proposition, combined with the arguments in the proof of Proposition \ref{amplediv}, actually shows that for any finite surjective morphism $f:X \to Y$ between projective varieties over a field $k$, with $Y$ normal, a formal curve $\hat{\varphi}: \Spec k[\![q]\!] \to X$ is ZL-dense in $X$ if and only if $f\circ \hat{\varphi}$ is ZL-dense in $Y$. 
\end{obs}

\subsection{ZL-density in quasi-projective varieties}

 We defined a notion of ZL-density for formal curves in \emph{projective} varieties. In this paragraph, under a mild technical condition, we extend this notion, via compactification, to formal curves in \emph{quasi-projective} varieties. To assure that we obtain a well defined notion, we must show that this does not depend on the choice of compactification.

We start by reformulating ZL-density in terms of \emph{Weil divisors} in place of Cartier divisors. Let $X$ be an algebraic variety over a field $k$, and $\hat{\varphi}: \Spec k[\![q]\!] \to X$ be a formal curve. If $\hat{\varphi}(0)$ is a regular point of $X$, then we may define the intersection multiplicity of a Weil divisor with $\hat{\varphi}$: if $U$ is a regular open neighborhood of $\hat{\varphi}(0)$ and $Z$ is a Weil divisor on $X$, then $Z\cap U$ is induced by some Cartier divisor $D$ on $U$, and we define
\begin{align*}
  \mult_{\hat{\varphi}}Z = \mult_{\hat{\varphi}}D\text{.}
\end{align*}
We may thus mimic the proof of Proposition \ref{amplediv} to obtain the following result.

\begin{prop}
  Let $X$ be a projective variety of dimension $n$ over a field $k$, $L$ be an ample line bundle on $X$, and $\hat{\varphi} : \Spec k[\![q]\!] \to X$ be a formal curve such that $\hat{\varphi}(0)$ is a regular point of $X$. Then, $\hat{\varphi}$ is ZL-dense in $X$ if and only if there exists a constant $C>0$ such that, for every effective Weil divisor $Z$ on $X$,
  \begin{align*}
    \mult_{\hat{\varphi}}Z \le C (\deg_LZ)^n
  \end{align*}
\end{prop}

We are now in position to prove that a modification away from $\hat{\varphi}$ does not affect ZL-density.

\begin{prop}\label{propmod}
Let $f:X \to Y$ be a proper morphism between projective varieties over a field $k$, and $U$ be an open subset of $Y$ such that $f$ induces an isomorphism $f^{-1}(U) \stackrel{\sim}{\to} U$. If $\hat{\varphi}: \Spec k[\![q]\!] \to f^{-1}(U)\subset X$ is a formal curve such that $\hat{\varphi}(0)$ is a regular point of $X$, then $\hat{\varphi}$ is ZL-dense in $X$ if and only if $f\circ \hat{\varphi}$ is ZL-dense in $Y$.
\end{prop}

\begin{proof}
  Since $f:X \to Y$ is a proper birational morphism, and $Y$ is irreducible, $f$ is surjective. By Proposition \ref{surj}, if $\hat{\varphi}$ is ZL-dense in $X$, then $f\circ\hat{\varphi}$ is ZL-dense in $Y$.

  Conversely, suppose that $f\circ\hat{\varphi}$ is ZL-dense in $Y$. Fix an ample line bundle $L$ (resp. $M$) on $X$ (resp. $Y$), and let $Z$ be an effective Weil divisor on $X$. Since $f$ is an isomorphism over $U$ and $\hat{\varphi}$ factors through $f^{-1}(U)$, we have
  \begin{align*}
    \mult_{\hat{\varphi}}Z = \mult_{f\circ \hat{\varphi}}f_*Z\text{.}
  \end{align*}
  As $f\circ\hat{\varphi}$ is ZL-dense in $Y$, there is a constant $C_1>0$ (not depending on $Z$) such that
  \begin{align*}
    \mult_{f\circ\hat{\varphi}}f_*Z \le C_1 (\deg_M f_*Z)^n\text{,}
  \end{align*}
  where $n=\dim Y = \dim X$.   By the projection formula, $\deg_M f_*Z = \deg_{f^*M}Z$. Since $f^*M$ is semiample, it follows from Proposition \ref{propsemiample} that there exists a constant $C_2>0$ such that $\deg_{f^*M}Z\le C_2 \deg_L Z$. We conclude that
  \begin{align*}
   \mult_{\hat{\varphi}}Z \le C_1C^n_2 (\deg_LZ)^n\text{.}
  \end{align*}
\end{proof}

Let us remark that it is essential in our proof above to consider the pushforward of a Weil divisor on $X$ under the proper morphism $f$. This explains why we reformulated ZL-density in terms of Weil divisors as the pushforward of a Cartier divisor under a proper morphism is not well defined in general.



\begin{coro}\label{corostandard}
Let $X$ be a quasi-projective variety over a field $k$, and $\hat{\varphi}: \Spec k[\![q]\!] \to X$ be a formal curve such that $\hat{\varphi}(0)$ is a regular point of $X$. If $j_i:X \hookrightarrow \overline{X}_i$, $i=1,2$, are two projective compactifications of $X$, then $j_1\circ \hat{\varphi}$ is ZL-dense in $\overline{X}_1$ if and only if $j_2\circ \hat{\varphi}$ is ZL-dense in $\overline{X}_2$. 
\end{coro}

\begin{proof}
Consider the scheme theoretic image $\overline{X}$ of $(j_1,j_2): X \to \overline{X}_1\times_k\overline{X}_2$ and apply Proposition \ref{propmod} to the natural projections $\overline{X} \to \overline{X}_i$, $i=1,2$.
\end{proof}


This enables us to define a good notion of ZL-density in a quasi-projective variety.

\begin{defi}
Let $X$ be a quasi-projective variety over a field $k$, and $\hat{\varphi}: \Spec k[\![q]\!] \to X$ be a formal curve such that $\hat{\varphi}(0)$ is a regular point of $X$. We say that $\hat{\varphi}$ is \emph{ZL-dense} in $X$ if there exists a projective compactification $j: X \hookrightarrow \overline{X}$ of $X$ such that $j\circ \hat{\varphi}$ is ZL-dense in $\overline{X}$.
\end{defi}

\section{Moderate growth and jet estimates on complex disks}\label{mgje}

In this section we introduce \emph{characteristic functions} and \emph{moderate growth} of certain 2-forms on a complex disk; these are purely analytic notions. We then proceed to establishing natural estimates on jets of sections of holomorphic line bundles on disks. In a sense, our exposition is more basic than the usual accounts on Value Distribution Theory, since characteristic functions of analytic curves will be a special case of our construction.

The kind of jet estimates we consider here play a central role in Diophantine Approximation and transcendence proofs. They notably appear in such proofs using the formalism of Arakelov Geometry, to estimate the height of evaluation maps, when applying Bost's method of slopes (see, for instance, \cite{bost95}, \cite{bost01}, \cite{graftieaux05}, \cite{gasbarri10}, \cite{gasbarri13}, \cite{herblot12}). 

\subsection{Characteristic functions}\label{charfunc}

Let $r>0$ be a real number, and $p\in D_r \defeq \{z \in \CC \mid |z|<r\}$. Recall that the \emph{Green's function} of $D_r$ at $p$ is defined by
\begin{align*}
g_{D_r,p}: \CC &\to (-\infty,+\infty]\\
                   z&\mapsto \log^+ \left|\frac{r^2 -\overline{p}z}{r(z-p)}\right|\text{.}
\end{align*}
This is a superharmonic (thus locally integrable) function on $\CC$, real-valued and continuous on $\CC\minus\{p\}$, strictly positive and harmonic (thus $C^{\infty}$) on $D_r\minus \{p\}$, and vanishing identically on $\CC \minus D_r$.

For any locally bounded $2$-form $\alpha$ defined on an open neighborhood of $D_r$, we denote
\begin{align*}
T_{\alpha, p}(r) \defeq \int_{\CC}g_{D_r,p}\alpha\text{.}
\end{align*}

\begin{obs}\label{intparts}
An integration by parts with $u(t) = \int_{D_t}\alpha$ and $v(t)=\log t$ shows that
\begin{align*}
T_{\alpha,0}(r) = \int_0^r \left( \int_{D_t}\alpha\right)\frac{dt}{t}\text{.}
\end{align*}
Since $g_{D_r,p}$ is the composition of $g_{D_r,0}$ with the automorphism of $D_r$ given by
\begin{align*}
  \sigma_{r,p}(z) = \frac{r^2(z-p)}{r^2-\overline{p}z}\text{,}
\end{align*}
we obtain
\begin{align*}
T_{\alpha,p}(r) = \int_0^r \left( \int_{D_t}(\sigma_{r,p}^{-1})^*\alpha\right)\frac{dt}{t}\text{.}
\end{align*}
\end{obs}

Let $R>0$ be a real number, and $\alpha$ be a locally bounded semipositive $(1,1)$-form on the disk $D_{R}$. The non-decreasing function
\begin{align*}
  T_{\alpha} : (0,R) &\to \RR_{\ge 0}\\
                 r &\mapsto T_{\alpha,0}(r)
\end{align*}
is the \emph{characteristic function} of $\alpha$ in $D_R$.

We shall be particularly interested in the following special case. Let $\overline{L} = (L, \| \ \|)$ be a semipositive Hermitian line bundle on $D_R$. The \emph{characteristic function} of $\overline{L}$ in $D_R$ is defined by
\begin{align*}
T_{\overline{L}} \defeq T_{c_1(\overline{L})}\text{.}
\end{align*}
We also denote $T_{\overline{L},p}(r) = T_{c_1(\overline{L}),p}(r)$ for $p \in D_r\subset D_R$. 

\begin{obs}
  Let $d \ge 1$ be an integer. As $c_1(\overline{L}^{\tensor d}) = d\cdot c_1(\overline{L})$, we have $T_{\overline{L}^{\tensor d},p}(r) = d\cdot T_{\overline{L},p}(r)$.
\end{obs}

\subsection{Forms of moderate growth}

Let $R>0$ be a real number, and $\alpha$ be a locally bounded semipositive (1,1)-form on $D_R$.

\begin{defi}\label{defmodgrowth}
  We say that $\alpha$ has \emph{moderate growth} if there exist constants $a,b>0$ such that
  \begin{align*}
    T_{\alpha}(r) \le a + b \log \frac{1}{1-\frac{r}{R}}
  \end{align*}
  for any $r\in(0,R)$. When $\alpha = c_1(\overline{L})$ for some semipositive Hermitian line bundle $\overline{L}$ on $D_R$, we rather say that $\overline{L}$ has \emph{moderate growth}.
\end{defi}

The motivating example of a form of moderate growth is the following one.

\begin{ex}[Poincaré form]\label{poincareform}
Let
\begin{align*}
d\mu_R \defeq \frac{i}{2}\left(\frac{R}{R^2-|z|^2} \right)^2 dz \wedge d\overline{z}
\end{align*}
be the 2-form associated to the surface element of the \emph{Poincaré metric} $\frac{R}{R^2-|z|^2}|dz|$ on $D_R$. A direct computation shows that, for any $r\in(0,R)$,
  \begin{align*}
   T_{d\mu_R}(r) = \frac{\pi}{2}\log \frac{1}{1+\frac{r}{R}} + \frac{\pi}{2}\log \frac{1}{1-\frac{r}{R}}\text{.}
  \end{align*}
  Thus, the 2-form $d\mu_R$ on $D_R$ has moderate growth. 
\end{ex}

We have defined moderate growth for a form $\alpha$ as a growth condition on $T_{\alpha,0}(r)$ with respect to $r$. Our next result shows that a similar growth condition for $T_{\alpha,p}(r)$ holds uniformly for $p$ varying in a fixed compact subset.

\begin{lemma}\label{lemmagreenfunc}
  Let $0<r_0<r_1<R_1<R$ be real numbers. Then there exists a constant $C>0$ such that, for every $p\in D_{r_0}$, and every $r\in [R_1,R)$, we have
  \begin{align*}
    g_{D_r,p} \le g_{D_{r_1},p} + Cg_{D_r,0}\text{.}
  \end{align*}
\end{lemma}

\begin{proof}
  Let $r \in [R_1, R)$ and $p\in D_{r_0}$. We set
  \begin{align*}
   C_{r,p} \defeq \max_{z\in \partial D_{r_1}}\frac{g_{D_r,p}(z)}{g_{D_r,0}(z)}
  \end{align*}
  and
  \begin{align*}
   f_{r,p} \defeq g_{D_{r},p} - g_{D_{r_1},p} - C_{r,p}g_{D_r,0}\text{.}
  \end{align*}
  By definition of $C_{r,p}$, we see that $f_{r,p}\le 0$ over $\partial D_{r_1}$. Moreover, $f_{r,p}$ vanishes identically on $\partial D_r$. Since $f_{r,p}$ is subharmonic over the domains $D_{r_1}$ and $D_r \minus \overline{D}_{r_1}$, by the Maximum Principle, we conclude that $f_{r,p}\le 0$ everywhere.

  To finish the proof, it is sufficient to remark that $C_{r,p}$ is uniformly bounded for $r\in [R_1,R)$ and $p\in D_{r_0}$. Indeed, for $z \in \partial D_{r_1}$, we have
  \begin{align*}
    \frac{g_{D_r,p}(z)}{g_{D_r,0}(z)} = \frac{\log\left|\frac{r^2-\overline{p}z}{r(z-p)} \right|}{\log\left|\frac{r}{z} \right|} \le \frac{\log \frac{R^2 + r_0r_1}{R_1(r_1-r_0)}}{\log \frac{R_1}{r_1}}\text{.}
  \end{align*}
\end{proof}

\begin{prop}\label{lemmacompact}
Let $R>0$ be a real number, $K\subset D_R$ be a compact subset, and $\alpha$ be a locally bounded semipositive (1,1)-form on $D_R$. Fix $R_1\in(0,R)$ such that $K\subset D_{R_1}$. If $\alpha$ has moderate growth, then there exist real numbers $a,b>0$ such that
  \begin{align*}
{\sup}_{p\in K}T_{\alpha,p}(r) \le a + b \log \frac{1}{1-\frac{r}{R}}
  \end{align*}
  for every $r\in [R_1,R)$.
\end{prop}

\begin{proof}
  Let $0<r_0<r_1$ be real numbers such that $K\subset D_{r_0}$ and $r_1<R_1$. By Lemma \ref{lemmagreenfunc}, there exists a constant $C>0$ such that
  \begin{align*}
   g_{D_r,p} \le g_{D_{r_1},p} + C g_{D_r,0} 
  \end{align*}
  for every $r\in [R_1,R)$ and every $p\in K$, so that
  \begin{align*}
   T_{\alpha,p}(r) \le T_{\alpha,p}(r_1) + C T_{\alpha,0}(r)\text{.}
  \end{align*}
  Since $\alpha$ has moderate growth, to conclude it is sufficient to remark that the function $p \mapsto T_{\alpha,p}(r_1)$ is continuous, thus bounded on the compact $K$.
\end{proof}

\subsection{Jets and characteristic functions}

Let $r>0$ be a real number and $p \in D_r$. We define a probability measure $\pi_{r,p}$ supported on $\partial D_r$ by
  \begin{align*}
   \int \psi \pi_{r,p} = \frac{1}{2\pi}\int_{0}^{2\pi}\psi(\sigma_{r,p}^{-1}(re^{i\theta}))d\theta\text{,}
  \end{align*}
 where $\sigma_{r,p}$ is the function defined in Remark \ref{intparts}. For the next proposition, we shall need the following classical result.

\begin{lemma}\label{greenfunc}
As an equality of distributions on $\CC$, we have
  \begin{align*}
   -2dd^cg_{D_r,p} = \delta_p - \pi_{r,p}  \text{.}
  \end{align*}
\end{lemma}

\begin{proof}
Apply Remark \ref{intparts} and Stokes' Theorem (see also \ref{dc}).
\end{proof}

Let $U$ be an open subset of $\CC$ and $\overline{L} =(L,\| \ \|)$ be a Hermitian line bundle over $U$. If $s\in \Gamma(U,L)$ and $z\in U$, the  \emph{$m^{th}$ jet of $s$ at $z$} is denoted by $j^m_zs$. When $s$ has vanishing order at least $m$ at $z$, $j^m_zs$ is simply an element of the fiber of $L\tensor (\Omega^1_{U})^{\tensor m}$ at $z$. In this case, if $r$ is a real number strictly greater than $|z|$, we denote by 
$$
\|j^m_zs\|_r
$$ the norm of $j_z^ms$ with respect to the metric $\| \ \|$ on $L$ and the norm on $\Omega^1_{D_r,z}$ given by the dual of the Poincaré metric $\frac{r}{r^2-|z|^2}|dz|$ on $D_r$.

The following result, relating jets of sections with characteristic functions, is a basic tool in Nevanlinna Theory (see, for instance, \cite{NW04} Section 2.3); variants of it were used in the context of Diophantine Approximation in \cite{bost01} Proposition 4.14, \cite{bost04} Section 3, and \cite{gasbarri10} Theorem 5.13.

\begin{prop}\label{propmain}
  Let $R>0$ be a real number, $\overline{L} = (L, \| \ \|)$ be a semipositive Hermitian line bundle on $D_R$, and $p\in D_R$. For every real number $r\in (|p|,R)$ and every global section $s \in \Gamma(D_R,L)\minus \{0\}$, if $m \defeq \ord_{p}s$ denotes the vanishing order of $s$ at $p$, we have
 \begin{align}\label{mainformula}
   \log \|j_p^ms\|_r =T_{\overline{L},p}(r) + \int\log \|s\| \pi_{r,p} - \int g_{D_r,p}\delta_{\div(s)-m[p]}\text{.} 
 \end{align}
\end{prop}

We start with a lemma that follows immediately from the explicit formula for the Green's functions on disks (cf. Paragraph \ref{charfunc}).

\begin{lemma}\label{lemmamainprop}
  With the above notation, if $E\subset D_R$ denotes the support of the divisor $\div(s)$, then the distribution $\log \|s\| + m g_{D_r,p}$ on $D_R$ defines a $C^{\infty}$ function over $(D_r \minus E)\cup \{p\}$ and a continuous function over $\partial D_r \minus E$.  Moreover,
    \begin{align*}
     \lim_{z\rightarrow p}(\log \|s(z)\| + m g_{D_r,p}(z)) = \log \|j_p^ms\|_r\text{.}
    \end{align*}\hfill$\blacksquare$    
\end{lemma}

Observe now that, for a fixed $s$, both sides in formula (\ref{mainformula}) are continuous with respect to $r$. Since $E$ is discrete, we may thus assume that $\partial D_r \cap  E = \emptyset$.

\begin{proof}[Proof of Proposition \ref{propmain}]
 The Poincaré-Lelong formula yields the identity of currents
  \begin{align*}
   dd^c \log \|s\|^2 = \delta_{\div(s)} -c_1(\overline{L})\text{.}
  \end{align*}
  Thus, by Lemma \ref{greenfunc},
  \begin{align*}
 c_1(\overline{L})=  -2dd^c(\log \|s\| + mg_{D_r,p})+\delta_{\div(s)-m[p]}  + m\pi_{r,p}\text{.}
  \end{align*}
  Since $\pi_{r,p}$ is supported on $\partial D_r$, and $g_{D_r,p}$ vanishes identically on $\partial D_r$, we obtain
  \begin{align*}
     T_{\overline{L},p}(r) = \int g_{D_r,p}\cdot(-2dd^c(\log \|s\| + mg_{D_r,p})) + \int g_{D_r,p}\delta_{\div(s)-m[p]}\text{.}
  \end{align*}
  Note that, by our choice of $r$ and by Lemma \ref{lemmamainprop}, the above products of distributions are well-defined.

  By another application of Lemma \ref{greenfunc},
  \begin{align*}
     T_{\overline{L},p}(r) &= \int (\log \|s\| + mg_{D_r,p})(\delta_p -\pi_{r,p}) + \int g_{D_r,p}\delta_{\div(s)-m[p]}\\
                            & =\int (\log \|s\| + mg_{D_r,p})\delta_p - \int \log\|s\| \pi_{r,p} + \int g_{D_r,p}\delta_{\div(s)-m[p]}
  \end{align*}
  where in the second equality we used once again that $g_{D_r,p}$ vanishes identically on $\partial D_r$. To conclude, we apply once more Lemma \ref{lemmamainprop}, which ensures that the function $\log \|s\| + mg_{D_r,p}$ tends to $\log \|j^m_ps\|_r$ at $p$.
\end{proof}

\begin{coro}
Let us keep the notation of Proposition \ref{propmain}. Then,
\begin{equation}\label{jet1}
    \log \| j_p^m s \|_r \le T_{\overline{L},p}(r) + \log \| s \|_{L^{\infty}(\partial D_r)} \text{.}
  \end{equation}
 If, moreover, $p'$ is another point of $D_r$, and $m'$ denotes the vanishing order of $s$ at $p'$, then
  \begin{equation}\label{jet2}
   \log \| j_p^m s \|_r \le T_{\overline{L},p}(r) + \log \| s \|_{L^{\infty}(\partial D_r)}  -m'g_{D_r,p}(p')\text{.}
  \end{equation}
\end{coro}

\begin{proof}
  Since $\pi_{r,p}$ is a probability measure over $\partial D_r$, we have $\int \log \|s\| \pi_{r,p}\le \log \|s\|_{L^{\infty}(\partial D_r)}$. Thus, the estimate (\ref{jet1}) (resp. (\ref{jet2})) follows immediately from the non-negativity both of the function $g_{D_r,p}$ and of the distribution $\delta_{\div(s)-m[p]}$ (resp. $\delta_{\div(s)-m[p]-m'[p']}$). 
\end{proof}


\subsection{A first application of moderate growth}

We shall need the following elementary inequality.

\begin{lemma}\label{trick}
  Let $A$, $B$, and $R$ be positive real numbers. Set
  \begin{align*}
    r\defeq R \frac{B}{A+B}\text{.}
  \end{align*}
  If $\log (B/A)\ge 2$, then
  \begin{align*}
  A \log \left(\frac{1}{1-\frac{r}{R}}\right) - B \log r \le 2A \log \left(\frac{B}{A}\right) - B \log R\text{.}
  \end{align*}
\end{lemma}

\begin{proof}
  By homogeneity, we may assume that $A=R=1$, so that $r = B/(1+B)$ and our statement is equivalent to:
  \begin{align*}
   \log (1+B) - B \log (B/(1+B)) \le 2 \log B
  \end{align*}
  when $\log B \ge 2$. By subtracting $\log B$ from both sides, we see that this is yet equivalent to:
  \begin{align*}
(1+B)\log (1 + 1/B) \le \log B
  \end{align*}
  when $\log B \ge 2$. Now, this last inequality follows trivially from the fact that $\log(1+1/B)\le 1/B$.
\end{proof}

  

Note that $r$ as above is the minimum of the real function $t \mapsto  A \log \left(\frac{1}{1-\frac{t}{R}}\right) - B \log t$ defined on the open interval $(0,R)$. 

\begin{prop}\label{prop1mg}
  Let $R>0$ be a real number and $\overline{L}=(L, \| \ \|)$ be a semipositive Hermitian line bundle on $D_R$. If $\overline{L}$ has moderate growth, then there exist constants $\kappa_1,\kappa_2>0$ such that for every integer $d\ge 1$ and every bounded global section $s \in \Gamma(D_R,L^{\tensor d})\minus \{0\}$, if we denote $m \defeq \ord_0s$, then
  \begin{align*}
  \log \|j_0^ms\|_R \le \kappa_1 d + \kappa_2 d\log^+ m + \log \|s\|_{L^{\infty}(D_R)}\text{.}
  \end{align*}
\end{prop}

\begin{proof}
Since $\overline{L}$ has moderate growth, there exist real numbers $a,b>0$ such that
\begin{align*}
T_{\overline{L}}(r) \le a + b \log\left( \frac{1}{1-\frac{r}{R}}\right)
\end{align*}
for any $ r \in(0,R)$. We may assume that $b>1$. By the jet estimate (\ref{jet1}), for any $r\in(0,R)$, we have
\begin{align*}
 \log \|j_0^ms\|_r\le d\cdot T_{\overline{L}}(r) + \log \|s\|_{L^{\infty}(\partial D_r)} \le d\cdot T_{\overline{L}}(r) + \log \|s\|_{L^{\infty}(D_R)}\text{.}
\end{align*}
As
\begin{align*}
 \|j_0^ms\|_r = \|j_0^ms\|_R\left(\frac{r}{R} \right)^m\text{,}
\end{align*}
we obtain
\begin{align}\label{bound}
 \log \|j_0^ms\|_R \le ad + bd\log \left(\frac{1}{1-\frac{r}{R}} \right) - m \log r + m\log R + \log \|s\|_{L^{\infty}(D_R)}\text{.}
\end{align}

The result being trivial for $m=0$, we may assume that $m>0$. We now consider two cases. If $\log (m/bd)< 2$, then we may take $r \defeq R/2$ in (\ref{bound}) to obtain
\begin{align*}
\log \|j_0^ms\|_R \le (a + (1+e^2)b\log 2)d + \log \|s\|_{L^{\infty}(D_R)}\text{.}
\end{align*}
If $\log (m/bd)\ge 2$, we apply Lemma \ref{trick} for $A = bd$ and $B = m$:
\begin{align*}
  \log \|j_0^ms\|_R \le ad + 2bd\log \left(\frac{m}{bd} \right) + \log \|s\|_{L^{\infty}(D_R)}\le ad + 2bd\log m + \log \|s\|_{L^{\infty}(D_R)}\text{.}
\end{align*}
\end{proof}

\subsection{Bounding jets via Taylor coefficients at another point}

In this paragraph, we compare Taylor coefficients at different points. We start with a general result, and next we explain how moderate growth improves the estimate.

\begin{prop}\label{proptaylor}
Let $R>0$ be a real number, $\overline{L} = (L,\| \ \|)$ be a semipositive Hermitian line bundle on $D_R$, and $q \in D_R\minus \{0\}$. Fix a real number $R_0$ satisfying $|q|<R_0<R$, and a global holomorphic section $s_0 \in \Gamma(D_R,L)$ such that $s_0(q)\neq 0$. Then there exists a real number $\kappa >1$ such that, for every integer $J\ge 1$, every integer $d\ge 1$, and every global section $s \in \Gamma(D_R,L^{\tensor d})\minus\{0\}$, if $f$ denotes the germ of holomorphic function at $q$ such that $s=fs^{\tensor d}_0$ in a neighborhood of $q$, and if $m \defeq \ord_0s$, then
\begin{align*}
   \log \|j_0^ms\|_R \le \log \left(\left(\frac{R_0}{|q|} \right)^{-J}\left(\frac{R}{R_0}\right)^m\|s\|_{L^{\infty}(\partial D_{R_0})} + \kappa^{d+m+J}\max_{0\le j < J}\frac{|f^{(j)}(q)|}{j!} \right) + d\cdot T_{\overline{L}}(R_0)
\end{align*}
\end{prop}

Let us first remark that if such a constant $\kappa >0$ exists for $s_0 \in \Gamma(D_R,L)$ trivializing $L$ at $q$, then an analogous constant $\bar{\kappa}>0$ will exist for \emph{any other} trivialization $\bar{s}_0$ of $L$ in a neighborhood of $q$ --- we \emph{do not} require $\bar{s}_0$ to be a global section. Indeed, if we write $s = \bar{f}\bar{s}^{\tensor d}_0$ and $\bar{s}_0=us_0$ in a neighborhood of $q$, then $f = \bar{f}u^d$ and 
  \begin{align*}
  \max_{0\le j<J}\frac{|f^{(j)}(q)|}{j!}\le \max_{0\le j < J}\sum_{k+ l=j}\frac{|\bar{f}^{(k)}(q)|}{k!}\frac{|(u^d)^{(l)}(q)|}{l!}\le \left(J \max_{0\le j < J}\frac{|(u^d)^{(j)}(q)|}{j!}\right) \max_{0\le j < J}\frac{|\bar{f}^{(j)}(q)|}{j!}\text{.}
  \end{align*}
  We conclude by the Cauchy inequalities, which ensure that $\max_{0\le j < J}\frac{|(u^{d})^{(j)}(q)|}{j!}$ grows at most exponentially in $d+J$.

\begin{proof}
By the above remark, up to replacing $s_0$ by $z^{-\ord_0(s_0)}s_0$, we can assume that $s_0$ trivializes $L$ both at $q$ and at $0$. Let $a_j \in \CC$ be defined by the expansion
  \begin{align*}
    f(z) = z^m \sum_{j=0}^{\infty}a_j(z-q)^j
  \end{align*}
  in a neighborhood of $q$, and set
  \begin{align*}
   g(z)\defeq z^m\sum_{j=0}^{J-1}a_j(z-q)^j\text{.}
  \end{align*}
  Note that $g$ extends uniquely to a holomorphic function on $D_R$. Let $s_1,s_2 \in \Gamma(D_R,L^{\tensor d})$ be given by $s_1 \defeq gs^{\tensor d}_0$ and $s_2\defeq s-s_1$.  Observe that both $s_1$ and $s_2$ have vanishing order at least $m$ at $0$.

  Next, we estimate $\|j^m_0s_i\|_{R_0}$, $i=1,2$; for this, we shall first assume that each $j_0^ms_i\neq 0$. By the jet estimate (\ref{jet1}) for $p=0$, we have
  \begin{align*}
   \log \|j_0^ms_1\|_{R_0} \le d\cdot T_{\overline{L}}(R_0) + \log \|s_1\|_{L^{\infty}(\partial D_{R_0})}\text{.}
  \end{align*}
  Since $\ord_qs_2\ge J$, by the jet estimate (\ref{jet2}) for $p=0$ and $p'=q$, we have
  \begin{align*}
   \log \|j_0^ms_2\|_{R_0} \le d\cdot T_{\overline{L}}(R_0) + \log \|s_2\|_{L^{\infty}(\partial D_{R_0})}  -\log \left(\frac{R_0}{|q|} \right)J\text{.}
  \end{align*}
  Thus
  \begin{align*}
    \|j_0^ms\|_R  &= \|j_0^ms\|_{R_0}\left(\frac{R}{R_0}\right)^m \le (\|j_0^ms_1\|_{R_0} + \|j_0^ms_2\|_{R_0})\left(\frac{R}{R_0}\right)^m\\
                        &\le \left (\left(\frac{R}{R_0}\right)^m\|s_1\|_{L^{\infty}(\partial D_{R_0})} +  \left(\frac{R_0}{|q|} \right)^{-J}\left(\frac{R}{R_0}\right)^m\|s_2\|_{L^{\infty}(\partial D_{R_0})}\right)\exp(d\cdot T_{\overline{L}}(R_0))\text{.}
  \end{align*}
  Using that $\|s_2\|_{L^{\infty}(\partial D_{R_0})} \le  \|s_1\|_{L^{\infty}(\partial D_{R_0})} + \|s\|_{L^{\infty}(\partial D_{R_0})}$, we get
  \begin{align*}
    \|j_0^ms\|_R \le \left(\left(\frac{R_0}{|q|} \right)^{-J}\left(\frac{R}{R_0}\right)^m\|s\|_{L^{\infty}(\partial D_{R_0})} +  \left(1 + \left(\frac{R_0}{|q|} \right)^{-J}\right)\left(\frac{R}{R_0}\right)^m\|s_1\|_{L^{\infty}(\partial D_{R_0})}\right)\exp(d\cdot T_{\overline{L}}(R_0))\text{.}
  \end{align*}
  It should be clear at this point that the same estimate holds if $j^m_0s_1=0$ or $j^m_0s_2=0$.

  We now estimate $\|s_1\|_{L^{\infty}(\partial D_{R_0})}$. For any $z \in \partial D_{R_0}$, we have
 \begin{align*}
   |g(z)| &= R_0^m \left|\sum_{j=0}^{J-1}a_j(z-q)^j \right|\le R_0^m\left(\sum_{j=0}^{J-1}(2R_0)^j\right)\max_{0\le j < J}|a_j| \le JR_0^m\max\{1,(2R_0)^J\}\max_{0\le j < J}|a_j|\text{,}
 \end{align*}
so that
 \begin{align*}
   \|s_1\|_{L^{\infty}(\partial D_{R_0})} = \sup_{z \in \partial D_{R_0}}|g(z)|\|s_0(z)\|^d\le \|s_0\|^d_{L^{\infty}(\partial D_{R_0})}JR_0^m\max\{1,(2R_0)^J\}\max_{0\le j < J}|a_j| .
 \end{align*}

 To finish, we must bound the coefficients $a_j$. By definition, for any $j\in \NN$, 
  \begin{align*}
   a_j = \frac{1}{j!}\left.\frac{d^j}{dz^j}\right|_{z=q}\left(\frac{f(z)}{z^m}\right) = \sum_{k=0}^j\left(\frac{(-1)^k}{q^{m+k}}\binom{k+m-1}{k}\frac{f^{(j-k)}(q)}{(j-k)!} \right)\text{.}
  \end{align*}
  If $j< J$, then, for any $0\le k \le j$, we have the (crude but sufficient) estimate
  \begin{align*}
    \binom{k+m-1}{k} < \binom{J+m -1}{J} < 2^{m+J-1} < 2^{m+J}\text{,}
  \end{align*}
  so that
  \begin{align*}
    |a_j| \le \left(\sum_{k=0}^j\frac{1}{|q|^{m+k}}\right)2^{m+J}\max_{0\le k \le j}\frac{|f^{(k)}(q)|}{k!}\text{.}
  \end{align*}
  Thus,
  \begin{align*}
    \max_{0\le j < J}|a_j| \le J (2\max\{1,|q|^{-1}\})^{m+J}\max_{0\le j <J}\frac{|f^{(j)}(q)|}{j!}\text{.}
  \end{align*}
\end{proof}

\begin{prop}\label{boundcompact}
  Let $R>0$ be a real number, $\overline{L} = (L,\| \ \|)$ be a semipositive Hermitian line bundle over $D_R$, and let $K\subset D_R$ be a compact subset. If $\overline{L}$ has moderate growth, then there exist real numbers $\kappa_0,\kappa_1>1$ and an integer $d_0\ge 1$ such that, for any integer $d\ge d_0$ and any bounded section $s \in \Gamma(D_R,L^{\tensor d})\minus \{0\}$ for which $m \defeq \ord_0s$ satisfies $m\ge \kappa_0 d$, we have
  \begin{align*}
   \|s(z)\| \le m^{\kappa_1d}\left(\frac{|z|}{R} \right)^m\|s\|_{L^{\infty}(D_R)} 
  \end{align*}
  for every $z \in K\minus \{0\}$.
\end{prop}

\begin{proof}
 Fix any $R_1\in(0,R)$ such that $K\subset D_{R_1}$. By Proposition \ref{lemmacompact}, there exist real numbers $a,b>0$ such that
  \begin{align*}
   T_{\overline{L}, z}(r) \le a + b \log \frac{1}{1-\frac{r}{R}}
  \end{align*}
  for any $z \in K$ and every $r \in [R_1,R)$.

  Let $s \in \Gamma(D_R,L^{\tensor d})\minus\{0\}$ be a bounded section, and $z \in K\minus \{0\}$. We may assume that $\ord_zs = 0$. By the jet estimate (\ref{jet2}) for $p=z$ and $p'=0$, we have, for every $r\in [R_1,R)$,
  \begin{align*}
    \log \|s(z)\| &\le d\cdot T_{\overline{L},z}(r) + \log \|s\|_{L^{\infty}(\partial D_r)} -m\log \frac{r}{|z|}\\
                             &  \le ad + bd \log \frac{1}{1-\frac{r}{R}} - m \log r + m \log |z| + \log \|s\|_{L^{\infty}(D_{R})}\text{.}
  \end{align*}
 Assume that $m\ge e^2bd$ (i.e. $\log(m/bd)\ge 2$). It follows from Lemma \ref{trick} for $A=bd$ and $B = m$ that, if
\begin{align*}
  r\defeq R \frac{m}{bd+m}\text{,}
\end{align*}
then
\begin{align*}
bd \log \frac{1}{1-\frac{r}{R}} -m\log r \le 2bd\log \frac{m}{bd} - m\log R\text{.}
\end{align*}
If we also require that $m\ge b\frac{R_1}{R-R_1}d$, then $r\ge R_1$, so that
\begin{align*}
  \log \|s(z)\| \le 2bd\log m + (ad -2bd\log bd) + m\log \frac{|z|}{R}+ \log \|s\|_{L^{\infty}(D_R)}\text{.}  
\end{align*}
Now, for every integer $d\ge e^{\frac{a}{b}}/b$, we have $ad - 2bd\log bd \le 0$, and we get
\begin{align*}
\log \|s(z)\| \le 2bd\log m +m\log \frac{|z|}{R} + \log \|s\|_{L^{\infty}(D_R)}\text{.}
\end{align*}
We may thus take $\kappa_0 \defeq b \max\left\{e^2, \frac{R_1}{R-R_1}\right\}$, $\kappa_1 \defeq 2b$, and $d_0 \defeq \lceil e^{\frac{a}{b}}/b\rceil$. 
\end{proof}

The following result is a combination of Proposition \ref{proptaylor} together with the existence of a non-zero global section of $L$ (see the remark following the statement; actually, $L$ is holomorphically trivial on $D_R$) and Proposition \ref{boundcompact}.

\begin{coro}\label{corotaylor}
Let $R>0$ be a real number, $\overline{L} = (L,\| \ \|)$ be a semipositive Hermitian line bundle on $D_R$, $q \in D_R\minus \{0\}$, and $s_0$ be a holomorphic trivialization of $L$ in a neighborhood of $q$. Assume moreover that $\overline{L}$ has moderate growth. Then there exist real numbers $\kappa_i>1$, $i=0,\ldots,4$, and an integer $d_0\ge 1$, such that, for any integer $J\ge 1$, any integer $d\ge d_0$, and every bounded section $s \in \Gamma(D_R,L^{\tensor d})\minus \{0\}$ for which $m\defeq \ord_0 s$ satisfies $m\ge \kappa_0d$, if $f$ denotes the germ of holomorphic function at $q$ such that  $s = fs_0^{\tensor d}$ in a neighborhood of $q$, we have
  \begin{align*}
    \log \|j_0^ms\|_R \le \log \left(\frac{m^{\kappa_1d}}{\kappa_2^J}\|s\|_{L^{\infty}(D_R)} + \kappa_3^{d+m+J} \max_{0\le j < J}\frac{|f^{(j)}(q)|}{j!}\right) + \kappa_4d\text{.}
  \end{align*}
\end{coro}

\begin{proof}
Fix any real number $R_0 \in (|q|,R)$.  We take $\kappa_2 \defeq R_0/|q|$, $\kappa_3 \defeq \kappa$ given by Proposition \ref{proptaylor}, and $\kappa_4 \defeq T_{\overline{L}}(R_0)$. Since $\overline{L}$ has moderate growth, we may apply Proposition \ref{boundcompact} to the compact $K= \partial D_{R_0}$ to obtain real numbers $\kappa_0,\kappa_1>0$ and an integer $d_0\ge 1$ such that
\begin{align*}
  \|s\|_{L^{\infty}(\partial D_{R_0})} \le m^{\kappa_1 d}\left(\frac{R_0}{R} \right)^m \|s\|_{L^{\infty}(D_R)}
\end{align*}
  for any integer $d\ge d_0$ and any bounded section $s \in \Gamma(D_R,L^{\tensor d})\minus\{0\}$ such that $m\defeq \ord_0s\ge \kappa_0d$. We conclude by combining this bound with the estimate given by Proposition \ref{proptaylor}.
\end{proof}

In practice, we shall be concerned with the following particular situation.

\begin{coro}\label{corotaylor2}
Let $R>0$ be a real number, $\overline{L} = (L,\| \ \|)$ be a semipositive Hermitian line bundle of moderate growth on $D_R$, and $q \in D_R\minus \{0\}$. Fix a holomorphic trivialization $s_0$ of $L$ in a neighborhood of $q$, real constants $c_0,c_1,c_2 >0$, with $c_0<c_1$, and an integer $n\ge 2$. For any real number $C>0$, there exist real numbers $\gamma_0,\gamma_1>0$ such that, for any sufficiently large integer $d$, and any bounded section $s\in \Gamma(D_R,L^{\tensor d})$ satisfying
  \begin{align*}
   c_0d^n \le m\defeq \ord_0s \le c_1d^n\text{, }\ \ \ \log \|s\|_{L^{\infty}(D_R)}\le c_2d\log d\text{,}
  \end{align*}
 and
  \begin{align*}
   \max_{0\le j< \lceil\gamma_0d\log d\rceil}\log \frac{|f^{(j)}(q)|}{j!} \le -\gamma_1d^n\text{,}
  \end{align*}
  where $s = fs_0^{\tensor d}$ on a neighborhood of $q$, we have
  \begin{align*}
  \log \|j_0^ms\|_R \le -C d\log d \text{.}
  \end{align*}
\end{coro}

\begin{proof}
  Let $\kappa_i>1$, $i=0,\ldots,4$, be the constants given by Corollary \ref{corotaylor}. We claim that it suffices to take $\gamma_0>(\log \kappa_2)^{-1}(n\kappa_1 + c_2 +C)$ and $\gamma_1>c_1\log \kappa_3$.

  Indeed, let $s \in \Gamma(D_R,L^{\tensor d})$ be as in the statement. Since $n\ge 2$ and $m\ge c_0d^n$, if $d$ is sufficiently large, we have $m\ge \kappa_0d$, so that the conclusion of Corollary \ref{corotaylor} for $J\defeq \lceil \gamma_0d\log d \rceil$ applies:
  \begin{align}\label{ct1-1}
    \log \|j_0^ms\|_R \le \log \left(\frac{m^{\kappa_1d}}{\kappa_2^J}\|s\|_{L^{\infty}(D_R)} + \kappa_3^{d+m+J} \max_{0\le j < J}\frac{|f^{(j)}(q)|}{j!}\right) + \kappa_4d\text{.}
  \end{align}
  
  Since $m \le c_1d^n$, $\log \|s\|_{L^{\infty}(D_R)}\le c_2d\log d$, and $J \ge \gamma_0\d\log d$, we obtain
  \begin{align*}
   \log \left(\frac{m^{\kappa_1d}}{\kappa_2^J}\|s\|_{L^{\infty}(D_R)}\right)\le   (n\kappa_1+c_2 - (\log\kappa_2)\gamma_0)d\log d + \kappa_1(\log c_1) d\text{.}
  \end{align*}
  Thus, by our choice of $\gamma_0$, if $d$ is sufficiently large, we get
  \begin{align}\label{ct1-2}
   \log \left(\frac{m^{\kappa_1d}}{\kappa_2^J}\|s\|_{L^{\infty}(D_R)}\right) \le -(C+\varepsilon_1)d\log d\text{,}
  \end{align}
  for some $\varepsilon_1>0$.

  Since $m\le c_1d^n$, $J\le \gamma_0d\log d + 1$, and $\max_{0\le j< J}\log \frac{|f^{(j)}(q)|}{j!} \le -\gamma_1d^n$, we have
  \begin{align*}
 \log \left(\kappa_3^{d+m+J} \max_{0\le j < J}\frac{|f^{(j)}(q)|}{j!}\right) \le (c_1\log \kappa_3 - \gamma_1)d^n + \gamma_0(\log\kappa_3)d\log d + (\log \kappa_3)(d+1)\text{.}
  \end{align*}
  Thus, as $n\ge 2$, and by our choice of $\gamma_1$, if $d$ is sufficiently large, we obtain
  \begin{align}\label{ct1-3}
 \log \left(\kappa_3^{d+m+J} \max_{0\le j < J}\frac{|f^{(j)}(q)|}{j!}\right) \le -\varepsilon_2 d^n
  \end{align}
  for some $\varepsilon_2 >0$.

  We conclude by applying (\ref{ct1-2}) and (\ref{ct1-3}) in (\ref{ct1-1}), and by taking $d$ to be sufficiently large.
\end{proof}

\section{Analytic curves of moderate growth in quasi-projective varieties}\label{mgqp}

This section contains mostly well-known techniques and results in Nevanlinna Theory. These are nevertheless written in the literature in a form not suitable for our purposes. Although our proofs may vary, many of the theory concerning growth of entire analytic maps (``parabolic case'') easily translate into our hyperbolic situation; we refer the reader to the recent monograph \cite{NW04} for a thorough exposition of the general parabolic theory (in several variables). 

\subsection{Analytic curves of moderate growth in compact complex manifolds}

Let $R>0$ be a real number, $M$ be a compact complex manifold, and $\varphi: D_R \to M$ be an analytic map. Fix any Hermitian metric $h$ on $M$, and let $\omega \defeq -\Im h$ be the positive $(1,1)$-form associated to $h$; in other words, if $h = \sum_{k,l=1}^n h_{kl}dz_k\tensor d\bar{z}_l$ in a local chart $(z_1,\ldots,z_n)$ of $M$, then $\omega = \frac{i}{2}\sum_{k,l=1}^n h_{kl}dz_k\wedge d\bar{z}_l$.  

\begin{defi}\label{defmodgrowthcurve}
We say that $\varphi:D_R \to M$ has \emph{moderate growth} if the semipositive (1,1)-form $\varphi^*\omega$ on $D_R$ has moderate growth (see Definition \ref{defmodgrowth}).
\end{defi}

This notion does not depend on the choice of the Hermitian metric $h$. Indeed, since $M$ is compact, any two Hermitian metrics on $M$ are ``comparable'': if $h_0$ is another Hermitian metric on $M$, then there exist real numbers $\alpha,\beta >0$ such that $\alpha \| \ \|_{h_0} \le \| \ \|_h \le \beta \| \ \|_{h_0}$.

\begin{obs}
It follows from Remark \ref{intparts} that $T_{\varphi^*\omega}(r) = \int_0^r\left(\int_{D_t}\varphi^*\omega \right)d\log t$ can be thought of as a logarithmic integral of the areas of the disks $\varphi(D_t)$ in $M$ for $0 <t<r$.
\end{obs}

We next consider a simple example of curves of moderate growth.

\begin{ex}[Bounded derivative]\label{bd}
Let $\varphi: D_R \to M$ be an analytic map, and $h$ be a Hermitian metric on $M$. Then we can write
 \begin{align*}
\varphi^*\omega = \|\varphi'(z)\|_{R,h}^2 d\mu_R\text{,}
\end{align*}
where $d\mu_R$ is the Poincaré form defined in Example \ref{poincareform}, and $\|\varphi'(z)\|_{R,h}$ denotes the norm of the tangent map $D_z\varphi:T_zD_R \to T_{\varphi(z)}M$ with respect to the Poincaré metric on $D_R$, and the Hermitian metric $h$ on $M$. Since $d\mu_R$ has moderate growth, the analytic curve $\varphi$ has moderate growth in $M$ whenever the function $z \mapsto \|\varphi'(z)\|_{R,h}$ is bounded on $D_R$ (e.g., $\varphi$ extends continuously to $\overline{D_R}\subset \CC$).
\end{ex}

\subsection{Nevanlinna's characteristic function}

Let $M$ be a complex manifold, $\overline{L}=(L,\| \ \|)$ be a semipositive Hermitian line bundle on $M$, and $s_0 \in \Gamma(M,L)\minus\{0\}$ be a non-zero global section.

Let $R>0$ be a real number and $\varphi:D_R \to M$ be an analytic map whose image is not contained in the support of $\div(s_0)$. We define, for every $r\in(0,R)$,
\begin{align*}
m_{\varphi,\overline{L},s_0}(r) \defeq \frac{1}{2\pi}\int_0^{2\pi} \log \frac{1}{\|s_0(\varphi(re^{i\theta}))\|}d\theta
\end{align*}
and
\begin{align*}
N_{\varphi,\overline{L},s_0}(r) \defeq (\ord_0\varphi^*s_0)\log r + \sum_{0<|z|< r} (\ord_z\varphi^*s_0) \log \frac{r}{|z|}\text{.}
\end{align*} 
Then we can form the \emph{Nevanlinna characteristic function} on the interval $(0,R)$
\begin{align*}
T_{\varphi,\overline{L},s_0} \defeq m_{\varphi,\overline{L},s_0} + N_{\varphi,\overline{L},s_0}\text{.}
\end{align*}

For the next proposition, we introduce a temporary notation. If $s$ is a global section of $\varphi^*L$ and $m=\ord_0s$, we denote by $\ell(s)$ the unique element of the fiber of $\varphi^*L$ at $0\in D_R$ such that $j_0^ms = \ell(s)\tensor dz^{\tensor m}$ (the ``leading coefficient'' of $s$).

The following classical identity (cf. \cite{NW04} Theorem 2.3.31) is an immediate corollary of Proposition \ref{propmain} applied to the section $s=\varphi^*s_0$ and the point $p=0$.

\begin{prop}[Nevanlinna's First Fundamental Theorem]\label{fft}
For every $r\in(0,R)$, we have
\begin{align*}
  T_{\varphi^*\overline{L}}(r) = T_{\varphi,\overline{L},s_0}(r) + \log \|\ell(\varphi^*s_0)\|\text{.}
\end{align*}\hfill $\blacksquare$
\end{prop}

As an application we show that, when $R=1$, polynomial growth of Taylor coefficients implies moderate growth.

\begin{ex}\label{polgrowth}
  Let $\varphi=(\varphi_1,\ldots,\varphi_n): D\to \CC^n$ be an analytic map with coordinates $\varphi_i (z) = \sum_{j=0}^{\infty}a_{ij}z^j$.  Assume that there exist a real number $C> 1$ and an integer $d\ge 1$ such that
  \begin{align*}
   |a_{ij}|\le C j^d
  \end{align*}
  for every $1\le i \le n$ and $j\ge 0$. Then, when identifying $\CC^n$ with the open affine subset \linebreak $U_0 = \{(p_0:\cdots:p_n) \in \PP^n(\CC) \mid p_0\neq 0\}$ of $\PP^n(\CC)$ via $(z_1,\ldots,z_n) \mapsto (1:z_1:\cdots :z_n)$, the analytic curve $\varphi: D \to \PP^n(\CC)$ has moderate growth.

  Indeed, let $\overline{\mathcal{O}(1)}$ denote the line bundle $\mathcal{O}(1)$ on $\PP^n(\CC)$ endowed with the Fubini-Study metric; that is,
\begin{align*}
 \|X_i(p)\| = \frac{|p_i|}{\sqrt{|p_0|^2 + \cdots + |p_n|^2}}
\end{align*}
for every $0\le i \le n$, and $p = (p_0:\cdots:p_n)\in \PP^n(\CC)$. Since $N_{\varphi,\overline{\mathcal{O}(1)},X_0}$ vanishes identically, by Proposition \ref{fft}, it is sufficient to prove that there exist $a,b>0$ such that
\begin{align*}
m_{\varphi,\overline{\mathcal{O}(1)},X_0}(r) \le a + b \log \frac{1}{1-r}
\end{align*}
for every $r\in(0,1)$.

For any real numbers $t_1,\ldots,t_m\ge 0$, we have $\log^+(\sum_{i=1}^mt_i) \le \sum_{i=1}^m\log^+t_i +  \log m$, so that
\begin{align*}
\log \frac{1}{\|X_0(\varphi(re^{i\theta}))\|} = \log \sqrt{1 + \sum_{i=1}^n|\varphi_i(re^{i\theta})|^2} \le \sum_{i=1}^n\log^+|\varphi_i(re^{i\theta})| + \log(\sqrt{1+n})\text{.}
\end{align*}
Since
\begin{align*}
|\varphi_i(re^{i\theta})| \le C\sum_{j=0}^{\infty}j^dr^j \le Cd!\left(\frac{1}{1-r} \right)^{d+1}\text{,}
\end{align*}
we may take $a = \log (\sqrt{1+n}) + n\log(Cd!)$ and $b = n(d+1)$. 
\end{ex}

\subsection{The field of moderate functions on a disk}

In this paragraph we study more closely the case $M=\PP^1(\CC)$. We refer to \cite{tsuji75} Chapters V-VII for a survey on the classical work on this subject. 

Let $R>0$ be a real number and $f$ be meromorphic function on $D_R$, i.e., an analytic map $f:D_R \to \PP^1(\CC)$ which is not constant equal to $\infty = (0:1)$.

\begin{defi}
We say that $f$ is a \emph{moderate function} on $D_R$ if the analytic map $f:D_R \to \PP^1(\CC)$ has moderated growth.
\end{defi}

If $\overline{\mathcal{O}(1)}$ denotes the line bundle $\mathcal{O}(1)$ on $\PP^1(\CC)$ endowed with the Fubini-Study metric (see Example \ref{polgrowth}), we denote
\begin{align*}
T_{f}\defeq T_{f^*\overline{\mathcal{O}(1)}}\text{.}
\end{align*}
By Proposition \ref{fft}, we have
\begin{align*}
  T_{f} =  m_{f,\overline{\mathcal{O}(1)},X_0} + N_{f,\overline{\mathcal{O}(1)},X_0} + O(1)\text{,}
\end{align*}
where $O(1)$ denotes a constant. To lighten the notation, we shall write $m_f = m_{f,\overline{\mathcal{O}(1)},X_0}$ (resp. $N_f = N_{f,\overline{\mathcal{O}(1)},X_0}$).

Let $\mathcal{K}_{D_R}$ denote the field of meromorphic functions on $D_R$. It is classical (and easy to prove) that characteristic functions are compatible with the algebraic structure of $\mathcal{K}_{D_R}$ in the following sense: for $f,g\in \mathcal{K}_{D_R}\minus\{0\}$ and $n \in \ZZ\minus\{0\}$, we have
\begin{align}\label{algop}
T_{f+g} \le T_f + T_g + O(1)\text{, }\ \  T_{fg} \le T_f + T_g + O(1)\text{, }\ \ T_{f^{n}} = |n|T_f +O(1)
\end{align}

It follows from the above relations that the subset $\mathcal{K}_{D_R}^{m}$ of $\mathcal{K}_{D_R}$ consisting of moderate meromorphic functions is a field.

\begin{prop}[cf. \cite{NW04} Lemma 2.5.15]\label{algcomp}
  Let $f,f_1,\ldots,f_n$ be meromorphic functions on $D_R$. If $f$ is algebraic over the field $\CC(f_1,\ldots,f_n)\subset \mathcal{K}_{D_R}$, then there exist real numbers $a,b>0$ such that
  \begin{align*}
   T_f \le a + b \sum_{i=1}^nT_{f_i}\text{.}
  \end{align*}
\end{prop}

\begin{proof}
  Let $d$ be the degree of $f$ over $\CC(f_1,\ldots,f_n)$. If $d=0$, then the result follows immediately from formulas (\ref{algop}). Assume that $d\ge 1$, and let $P = X^d - g_{d-1}X^{d-1} - \cdots - g_0 \in \CC(f_1,\ldots,f_n)[X]$ be the minimal polynomial of $f$. Since each $g_i \in \CC(f_1,\ldots,f_n)$, it suffices to prove that $T_f \le \sum_{i=0}^{d-1}T_{g_i} + O(1)$.

  By formulas (\ref{algop}), we have
  \begin{align*}
  T_{f^{d}} = T_{(g_{d-1}f^{d-2} + \cdots + g_1)f + g_0} \le T_{g_{d-1}f^{d-2} + \cdots + g_1} + T_f + T_{g_0} + O(1)\text{.}
  \end{align*}
  By descending induction, we get
  \begin{align*}
  T_{f^d} \le (d-1)T_f + \sum_{i=0}^{d-1}T_{g_i} + O(1)\text{.}
  \end{align*}
  As $T_{f^d} = d\cdot T_f + O(1)$, we obtain
  \begin{align*}
   T_f \le \sum_{i=0}^{d-1}T_{g_i}+ O(1)\text{.}
  \end{align*}
\end{proof}

\begin{coro}
The field of moderate functions $\mathcal{K}_{D_R}^m$ is algebraically closed in $\mathcal{K}_{D_R}$. \hfill $\blacksquare$
\end{coro}

In particular, since the inclusion $D_R \to \PP^1(\CC)$ is easily seen to be an analytic map of moderate growth (see Example \ref{bd}), the field $\mathcal{K}_{D_R}^m$ contains the field of (univalued) algebraic meromorphic functions on $D_R$.

\subsection{Birational invariance and moderate growth in quasi-projective varieties}\label{birinv}

In this paragraph, we establish the birational invariance of moderate growth under a non-degeneracy hypothesis. Our arguments follow closely those of \cite{NW04} 2.5; we claim no originality here.

In what follows, if $f$ is a meromorphic function on $D_R$, we denote the divisor of zeros (resp. poles) of $f$ by $\div_0(f)$ (resp. $\div_{\infty}(f)$), so that $\div(f) = \div_0(f) - \div_{\infty}(f)$.

\begin{lemma}[cf. \cite{NW04} Theorem 2.5.7]\label{lemmacomp1}
  Let $M$ be a compact complex manifold endowed with a semipositive Hermitian line bundle $\overline{L} = (L,\|\  \|)$. Fix global sections $s_0,s_1 \in \Gamma(M,L)$, with $s_0\neq 0$. Then, for any analytic map $\varphi : D_R \to M$ whose image is not contained in the support of $\div(s_0)$, if we denote by $f$ the unique meromorphic function on $D_R$ such that $f\varphi^*s_0 = \varphi^*s_1$, we have
  \begin{align*}
   T_{f} \le T_{\varphi, \overline{L}, s_0} + O(1)\text{.}
  \end{align*}
\end{lemma}

\begin{proof}
  Let $H$ be the support of $\div(s_0)$. For $p\in M \minus H$, we have
  \begin{align*}
    \log \sqrt{1 + \frac{\|s_1(p)\|^2}{\|s_0(p)\|^2}} = \log \frac{1}{\|s_0(p)\|} + \log \sqrt{\|s_0(p)\|^2 + \|s_1(p)\|^2}\text{.}
  \end{align*}
  Since $M$ is compact, the functions $\|s_i\|$ on $M$ are bounded, so that
\begin{align*}
  \log \sqrt{1 + \frac{\|s_1\|^2}{\|s_0\|^2}} \le \log \frac{1}{\|s_0\|} + O(1)
\end{align*}
over $M \minus H$. In particular, we get
\begin{align*}
 m_{f} \le m_{\varphi, \overline{L},s_0} +O(1)\text{.}
\end{align*}
Since $\div_{\infty}(f) \le \div(\varphi^*s_0)$, the bound
\begin{align*}
N_{f} \le N_{\varphi,\overline{L},s_0}
\end{align*}  
is trivial.
\end{proof}

For the next lemma, we endow the line bundle $\mathcal{O}(1)$ over $\PP^n(\CC)$ with the Fubini-Study metric as in Example \ref{polgrowth}. Moreover, if $E = \sum_{z \in D_R}n_z[z]$ is a divisor in $D_R$, we denote $i(z,E) \defeq n_z$.  

\begin{lemma}\label{lemmacomp2}
  Let $\varphi: D_R \to \PP^n(\CC)$ be an analytic map whose image is not contained in the support of $\div(X_0)$. For $1\le j \le n$, let us denote by $f_j$ the unique meromorphic function on $D_R$ such that $f_j\varphi^*X_0=\varphi^*X_j$. Then
  \begin{align*}
    T_{\varphi, \overline{\mathcal{O}(1)},X_0} \le \sum_{j=1}^n T_{f_j} + O(1)\text{.}
  \end{align*}
\end{lemma}

\begin{proof}
  We first prove that $m_{\varphi,\overline{\mathcal{O}(1)},X_0} \le \sum_{j=1}^n m_{f_j}$. 
For any real numbers $t_1,\ldots,t_n\ge 0$, we have
  \begin{align*}
  \log \left(1+ \sum_{j=1}^nt_j \right) \le \sum_{j=1}^n \log(1+t_j)\text{.}
  \end{align*}
  Thus, for any $r\in(0,R)$,
  \begin{align*}
    m_{\varphi,\overline{\mathcal{O}(1)},X_0}(r) &= \frac{1}{2\pi} \int_0^{2\pi}\log \sqrt{1 + \sum_{j=1}^{n}|f_j(re^{i\theta})|^2}d\theta \\
                                                 &\le \sum_{j=1}^n \frac{1}{2\pi}\int_0^{2\pi}\log \sqrt{1 + |f_j(re^{i\theta})|^2}d\theta = \sum_{j=1}^nm_{f_j}(r)
  \end{align*}

Next, observe that to prove that $N_{\varphi,\overline{\mathcal{O}(1)},X_0}\le \sum_{j=1}^n N_{f_j}$ it suffices to show that
\begin{align*}
  \div(\varphi^*X_0) \le \sum_{j=1}^n{\div}_{\infty}(f_j)\text{.}
\end{align*}
Since each $\div_{\infty}(f_j)$ is an effective divisor, it is sufficient to prove that, for every $z\in D_R$, there exists $1\le j\le n$ such that $i(z,\div (\varphi^*X_0)) \le i(z, \div_{\infty}(f_j))$. Now, for any $1\le j \le n$, since $f_j \varphi^*X_0 = \varphi^*X_j$, we  may write
\begin{align*}
{\div}_{\infty}(f_j) = \div (\varphi^*X_0) + {\div}_0(f_j) - \div(\varphi^*X_j)\text{.}
\end{align*}
Finally, we simply remark that for any $z\in D_R$ for which $X_0(\varphi(z))=0$ (i.e., $\ord_{z}\varphi^*X_0 >0$ or, equivalently, $i(z,\div(\varphi^*X_0))>0$), there exists $1\le j \le n$ such that $X_j(\varphi(z))\neq 0$ (i.e., $\ord_{z}\varphi^*X_j = 0$), so that $i(z,\div_{\infty}(f_j)) = i(z,\div(\varphi^*X_0)) + i(z,\div_0(f_j)) \ge i(z,\div(\varphi^*X_0))$.
\end{proof}

Let $f,g : I \to \RR$ be real functions defined on some interval $I\subset \RR$. We say that $f$ and $g$ are \emph{comparable} if there exist real numbers $a,b,c,d>0$  such that
\begin{align*}
af - b \le g \le cf+d
\end{align*}
everywhere on $I$.

\begin{theorem}[cf. \cite{NW04} Theorem 2.5.18]\label{funcfield}
Let $R>0$ be a real number, $X$ be a smooth projective variety of dimension $n$ over $\CC$, and $\varphi: D_R \to X^{\an}$ be an analytic map whose image is Zariski-dense in $X$.  Then, for any positive $(1,1)$-form $\omega$ on $X$, and any transcendence basis $(f_1,\ldots,f_n)$ of the function field $\CC(X)$ of $X$, the real functions $T_{\varphi^*\omega}$ and $\sum_{j=1}^nT_{f_j\circ \varphi}$ on $(0,R)$ are comparable. In particular, $\varphi$ has moderate growth in $X$ if and only if $f_j\circ \varphi$ are moderate functions on $D_R$ for every $1\le j \le n$. 
\end{theorem}

Observe that the Zariski-density hypothesis above ensures that, for any rational function $f$ on $X$, the image of $\varphi$ is not contained in the indeterminacy locus of $f$, so that $f\circ \varphi$ is a well-defined meromorphic function on $D_R$.

\begin{proof}
Let $i:X \to \PP^{N}_{\CC} = \Proj \CC[X_0,\ldots,X_N]$ be a closed immersion such that $i\circ \varphi(D_R)$ is not contained in the support of $\div(X_0)$, and consider the rational functions $g_j \in \CC(X)$, $1\le j \le N$, given by restriction of $X_j/X_0$ to $X$. 

It follows from the compactness of $X^{\an}$ (cf. remark following Definition \ref{defmodgrowthcurve}), and from Proposition \ref{fft}, that the functions $T_{\varphi^*\omega}$ and $T_{i\circ \varphi,\overline{\mathcal{O}(1)},X_0}$ are comparable. By Lemmas \ref{lemmacomp1} and \ref{lemmacomp2}, the functions $T_{i\circ \varphi,\overline{\mathcal{O}(1)},X_0}$ and $\sum_{j=1}^N T_{g_j\circ\varphi}$ are comparable. Furthermore, as $\CC(X) = \CC(g_1,\ldots,g_N)$ is an algebraic extension of $\CC(f_1,\ldots,f_n)$, we deduce from formulas (\ref{algop}) and from Proposition \ref{algcomp} that $\sum_{j=1}^NT_{g_j\circ \varphi}$ and $\sum_{j=1}^nT_{f_j\circ \varphi}$ are comparable. Our statement follows by transitivity of comparability.
\end{proof}

In particular, moderate growth in projective varieties is a birational invariant.

\begin{coro}\label{corobirinv}
Let $f:X \to Y$ be a birational morphism between smooth projective varieties over $\CC$. If $R>0$ is a real number, then an analytic map $\varphi:D_R \to X^{\an}$ with Zariski-dense image has moderated growth if and only if $f\circ \varphi :D_R \to Y^{\an}$ has moderate growth. \hfill $\blacksquare$
\end{coro}

Combining the standard argument in the proof of Corollary \ref{corostandard} with a resolution of singularities yields the following.

\begin{coro}
 Let $X$ be a smooth quasi-projective variety over $\CC$, and let $j_i : X \hookrightarrow \overline{X}_i$, $i=1,2$, be smooth projective compactifications of $X$. If $R>0$ is a real number and $\varphi: D_R \to X^{\an}$ is an analytic map with Zariski-dense image, then $j_1\circ \varphi$ has moderate growth if and only if $j_2\circ \varphi$ has moderate growth. \hfill $\blacksquare$
\end{coro}

We may thus define an unambiguous notion of moderate growth for Zariski-dense analytic curves in smooth quasi-projective varieties.

\begin{defi}
Let $X$ be a smooth quasi-projective variety, $R>0$ be a real number, and $\varphi:D_R \to X^{\an}$ be an analytic map with Zariski-dense image. We say that $\varphi$ has \emph{moderate growth} if there exists a smooth projective compactification $j: X \hookrightarrow \overline{X}$ of $X$ such that $j\circ\varphi:D_R \to \overline{X}^{\an}$ has moderate growth.
\end{defi}

\section{Construction of auxiliary sections}\label{secauxsec}

We prove in this section Theorem \ref{auxsec} below, generalizing the construction of auxiliary polynomials in Nesterenko's method. Our approach, based on Bost's method of slopes, differs from the classical combinatorial one. However, the backbone of the argument remains the same: Minkowski's theorem on minima of lattices (see Proposition \ref{minkthm} below).

\subsection{Notation and statement}\label{notstat}

Let $K$ be a number field and $\mathcal{O}_K$ be its ring of integers. Recall that, if $\mathcal{X}$ is an \emph{arithmetic scheme} over $\mathcal{O}_K$ (i.e., an integral scheme $\mathcal{X}$ with a separated and flat morphism of finite type $\mathcal{X} \to \Spec \mathcal{O}_K$) with smooth generic fiber $\mathcal{X}_K$, a Hermitian line bundle $\overline{L}=(L,(\| \ \|_{\sigma})_{\sigma: K \into \CC})$ over $\mathcal{X}$ is the data of a line bundle $L$ on $\mathcal{X}$ and a family of $C^{\infty}$ Hermitian metrics $\| \ \|_{\sigma}$ on the holomorphic line bundles $L_{\sigma}$ over $\mathcal{X}_{\sigma}^{\an}$ deduced from $L$ by the field embeddings $\sigma: K \into \CC$ that is invariant under complex conjugation. 

If $d\ge 1$ is an integer, and $s \in \Gamma(\mathcal{X},L^{\tensor d})$ is a global section, we denote
\begin{align*}
 \|s\|_{\mathcal{X}} \defeq \max_{\sigma}\|s\|_{\sigma,L^{\infty}(\mathcal{X}_{\sigma}^{\an})}\text{.}
\end{align*}

This section is devoted to the proof of the following theorem.

\begin{theorem}\label{auxsec}
  Let $\mathcal{X}$ be a projective arithmetic scheme of relative dimension $n\ge 1$ over $\mathcal{O}_K$ with smooth generic fiber $\mathcal{X}_K$, and fix any Hermitian line bundle $\overline{L} = (L,(\| \ \|_{\sigma})_{\sigma: K \into \CC})$ on $\mathcal{X}$ such that $L_K$ is ample on $\mathcal{X}_K$. Let $\hat{\varphi} : \Spf \mathcal{O}_K[\![q]\!] \to \mathcal{X}$ be a morphism of formal $\mathcal{O}_K$-schemes such that, for every field embedding $\sigma: K \hookrightarrow \CC$, the formal curve $\hat{\varphi}_{\sigma} : \Spf \CC[\![q]\!] \to \mathcal{X}_{\sigma}$ lifts to an analytic map $\varphi_{\sigma}: D_{R_{\sigma}} \to \mathcal{X}_{\sigma}^{\an}$ defined on some complex disk of radius $R_{\sigma}>0$, and assume that
  $$
  \prod_{\sigma:K\into \CC}R_{\sigma}=1\text{.}
  $$
  If, moreover,
  \begin{enumerate}
   \item the image of $\hat{\varphi}_K : \Spf K[\![q]\!] \to \mathcal{X}_K$ is Zariski-dense, and
   \item for every field embedding $\sigma:K\into \CC$, $\varphi_{\sigma}$ has moderate growth in $\mathcal{X}_{\sigma}^{\an}$,
  \end{enumerate}
then, there are constants $c_1,c_2,c_3>0$ such that, for every large enough positive integer $d$, there exists $s \in \Gamma(\mathcal{X},L^{\tensor d})\setminus\{0\}$ such that
\begin{align*}
m\defeq \ord_0 \hat{\varphi}^*s > c_1d^n
\end{align*}
 and
  \begin{align*}
    \log \|s\|_{\mathcal{X}} \le c_2d + c_3d\log m\text{.}
  \end{align*}
\end{theorem}

\begin{obs}
By Paragraph \ref{birinv} one could also assume that $\mathcal{X}$ is only quasi-projective over $\Spec \mathcal{O}_K$, and then construct ``auxiliary sections'' on any projective compactification of $\mathcal{X}$ having smooth generic fiber.
\end{obs}

If we require the stronger condition of ZL-density of $\hat{\varphi}_K$ instead of Zariski-density, we obtain the following.

\begin{coro}\label{coroauxsec}
With hypotheses and notation as in Theorem \ref{auxsec}, if moreover $\hat{\varphi}_K: \Spf K[\![q]\!] \to \mathcal{X}_K$ is ZL-dense, then there exist constants $c_1,c_2>0$ such that, for every large enough positive integer $d$, there exists $s \in \Gamma(\mathcal{X},L^{\tensor d})\minus \{0\}$ such that
\begin{align*}
\ord_0 \hat{\varphi}^*s > c_1d^n
\end{align*}
and
\begin{align*}
\log \|s\|_{\mathcal{X}}\le c_2d\log d\text{.}
\end{align*}
\end{coro}

\subsection{Recollections of Arakelov theory; the slope inequality}

For the convenience of the reader, we recollect in this paragraph some fundamental notions and results concerning Hermitian vector bundles over rings of algebraic integers. Proofs and further developments can be found in \cite{bost95} Appendix A, \cite{bost01} Paragraphs 4.1-4.2, and \cite{BK10} Paragraphs 3.1-3.3.

Let $K$ be number field, $\mathcal{O}_K$ be its ring of integers, and set $S\defeq \Spec \mathcal{O}_K$. Recall that a \emph{Hermitian vector bundle} over $S$ is a couple $\overline{E} = (E,(\| \ \|_{\sigma})_{\sigma: K \into \CC})$, where $E$ is a projective $\mathcal{O}_K$-module of finite type, and $(\| \ \|_{\sigma})_{\sigma:K\into \CC}$ is a family of Hermitian norms over $E_{\sigma}\defeq E\tensor_{\sigma:\mathcal{O}_K \into \CC}\CC$, invariant under complex conjugation. If $\rk E = 1$, we say that $\overline{E}$ is a \emph{Hermitian line bundle} over $S$.

The multilinear constructions in the category of projective modules over $\mathcal{O}_K$ (e.g., tensor products, quotients, Hom) make sense in the category of Hermitian vector bundles over $S$.

\begin{defi}
  Let $\overline{E} = (E,(\| \ \|_{\sigma})_{\sigma:K \into \CC})$ be a Hermitian vector bundle over $S$, and fix $s \in \det E \smallsetminus \{0\}$, where $\det E \defeq \bigwedge^{\rk E} E$. We define the \emph{Arakelov degree} of $\overline{E}$ by
  \begin{align*}
    \widehat{\deg}(\overline{E}) \defeq \log |(\det E)/\mathcal{O}_Ks| - \sum_{\sigma: K \into \CC}\log \|s\|_{\sigma} \in \RR\text{.}
  \end{align*}
 This is easily seen not to depend on the choice of $s$. We define moreover the \emph{normalized Arakelov degree} of $\overline{E}$ by
  \begin{align*}
\widehat{\deg}_n(\overline{E}) \defeq \frac{1}{[K:\QQ]} \widehat{\deg}(\overline{E})
  \end{align*}
  and the \emph{slope} of $\overline{E}$ by
  \begin{align*}
  \hat{\mu} (\overline{E}) \defeq \frac{1}{\rk E}\widehat{\deg}_n(\overline{E})
  \end{align*}
  when $\rk E>0$, and $\hat{\mu}(\overline{E})\defeq-\infty$ when $\rk E = 0$.
\end{defi}

\begin{prop}[cf. \cite{bost01} 4.1.1]\label{genprop}
  The following properties hold:
  \begin{enumerate}
  \item If $\overline{L}$ and $\overline{M}$ are Hermitian line bundles over $S$, we have
    \begin{align*}
     \adeg(\overline{L}\tensor \overline{M}) = \adeg(\overline{L}) + \adeg(\overline{M})\text{.}
    \end{align*}
\item Let $\overline{E}$ be a Hermitian vector bundle over $S$ and
  \begin{align*}
E = E^0 \supset E^1 \supset \cdots \supset E^N \supset \{0\}
  \end{align*}
  be a filtration of $E$ by saturated $\mathcal{O}_K$-submodules. Then
  \begin{align*}
   \adeg(\overline{E}) = \adeg (\overline{E^N}) + \sum_{i=0}^{N-1}\adeg(\overline{E^i/E^{i+1}})\text{,}
  \end{align*}
  where $\overline{E^N}$ (resp. $\overline{E^i/E^{i+1}}$) denotes the Hermitian vector bundle with underlying module $E^N$ (resp. $E^i/E^{i+1}$) and Hermitian structure induced by $\overline{E}$.
   \item For every Hermitian vector bundle $\overline{E}$ over $S$, and every Hermitian line bundle $\overline{L}$ over $S$, we have
  \begin{align*}
   \hat{\mu}(\overline{E}\tensor \overline{L}) = \hat{\mu}(\overline{E}) + {\adeg}_n (\overline{L})\text{.}
  \end{align*}
\end{enumerate}
\end{prop}


Let $\overline{E}$ and $\overline{F}$ be Hermitian vector bundles over $S$. For every maximal ideal $\pp$ of $\mathcal{O}_K$, we denote by $K_{\pp}$ the completion of $K$ at $\pp$ and by $\| \ \|_{\pp}$ the non-archimedean norm over $\Hom_{K_{\pp}}(E_{K_{\pp}},F_{K_{\pp}})$ associated to the $\mathcal{O}_{K_{\pp}}$-lattice $\Hom_{\mathcal{O}_{K_{\pp}}}(E_{\mathcal{O}_{K_{\pp}}},F_{\mathcal{O}_{K_{\pp}}})$; explicitly, if $\varphi \in \Hom_{K_{\pp}}(E_{K_{\pp}},F_{K_{\pp}}) \minus \{0\}$, then $\|\varphi\|_{\pp} \defeq |\mathcal{O}_K/\pp|^{-v_{\pp}(\varphi)}$, where $v_{\pp}(\varphi) = \max\{n \in \ZZ  \mid \pi_{\pp}^{-n}\varphi \in \Hom_{\mathcal{O}_{K_{\pp}}}(E_{\mathcal{O}_{K_{\pp}}},F_{\mathcal{O}_{K_{\pp}}})\}$ and $\pi_{\pp}$ denotes some uniformizer of $\mathcal{O}_{K_{\pp}}$. For a field embedding $\sigma : K \into \CC$, we consider the operator norm on $\Hom_{\CC}(E_{\sigma},F_{\sigma})$:
\begin{align*}
\|\varphi \|_{\sigma} = \max_{v \in E_{\sigma}\minus \{0\}} \frac{\|\varphi(v)\|_{\sigma}}{\|v\|_{\sigma}}\text{.}
\end{align*}
Then, the \emph{height} of a non-zero $K$-linear map $\varphi : E_K \to F_K$ is defined by
  \begin{align*}
 h_{\overline{E},\overline{F}}(\varphi) = \frac{1}{[K:\QQ]}\left(\sum_{\pp} \log \|\varphi\|_{\pp} + \sum_{\sigma} \log \|\varphi\|_{\sigma} \right)\text{.}
  \end{align*}
  If $\varphi=0$, our convention is that $h_{\overline{E},\overline{F}}(\varphi) \defeq -\infty$.

\begin{prop}[Slope inequality; \cite{bost01} Proposition 4.5]\label{slopeineq}
  With the above notation, if \mbox{$\varphi : E_K \to F_K$} is injective, then
  \begin{align*}
   \hat{\mu}(\overline{E}) \le \hat{\mu}_{\max}(\overline{F}) + h_{\overline{E},\overline{F}}(\varphi)\text{,}
  \end{align*}
  where $\hat{\mu}_{\max}(\overline{F}) \defeq \sup \{\hat{\mu}(\overline{F'}) \mid F'\text{ is an } \mathcal{O}_K\text{-submodule of } F\}$.
\end{prop}

Let us point out that $\hat{\mu}_{\max}(\overline{F})$ is attained by a saturated submodule of $F$ (cf. \cite{bost95} A.3). In particular, if $\rk F=1$, then $\hat{\mu}_{\max}(\overline{F}) = \hat{\mu}(\overline{F})$.

\subsection{Short vectors in filtered Hermitian vector bundles}

Let $K$ be a field, $\mathcal{O}_K$ be its ring of integers, and $S=\Spec \mathcal{O}_K$. Let $\overline{E} = (E, (\| \ \|_{\sigma})_{\sigma: K \into \CC})$ be a non-zero Hermitian vector bundle over $S$; we denote its \emph{first successive minimum} by
\begin{align*}
\lambda_1(\overline{E}) \defeq \inf \left.\left\{\max_{\sigma }\|s\|_{\sigma} \right| s \in E\minus \{0\}\right\}\text{.}
\end{align*}
Since $s\mapsto s\tensor 1$ identifies $E$ with a lattice in the $\RR$-vector space $E\tensor_{\ZZ}\RR$, the first successive minimum is attained by some element $s \in E\minus \{0\}$.

\begin{prop}[Minkowski]\label{minkthm}
  Let $\overline{E}$ be a non-zero Hermitian vector bundle over $S$. Then
  \begin{align*}
\log \lambda_1(\overline{E}) \le -\hat{\mu}(\overline{E}) + \frac{1}{2}\log (\rk E) + \frac{\log |\Delta_K|}{2[K:\QQ]} + \frac{1}{2}\log [K:\QQ]\text{,}
  \end{align*}
  where $\Delta_K$ denotes the discriminant of $K$ over $\QQ$.
\end{prop}

This statement might be obtained from \cite{BK10} pp. 1027-1028 by considering the Hermitian vector bundle over $\Spec \ZZ$ given by the direct image of $\overline{E}$ via $S \to \Spec \ZZ$.

Let $(\overline{E_d})_{d\ge 1}$ be a family of Hermitian vector bundles over $S$ such that
\begin{align*}
r_d \defeq \rk E_d \to +\infty
\end{align*}
as $d \rightarrow +\infty$. Assume that, for each $d\ge 1$, we are given a \emph{separated} filtration
\begin{align*}
E_d^{0}= E_d \supset E_d^1 \supset E_d^2 \supset \cdots
\end{align*}
by saturated $\mathcal{O}_K$-submodules. We endow each $E_d^m$ with the Hermitian vector bundle structure induced from $\overline{E_d}$.

\begin{prop}\label{lemmaaux}
  With the above notation, assume that there exists an integer $k\ge 1$ and a double sequence $(a_{d,m})_{d\ge 1,m\ge 0}$ of positive real numbers, non-decreasing in $m$ for every $d\ge 1$, such that
  \begin{align}\label{est1}
    \rk (E^m_d/E^{m+1}_d)\le k
  \end{align}
 and
  \begin{align}\label{est2}
\hat{\mu}(\overline{E_d^m/E_d^{m+1}}) \le  a_{d,m}
  \end{align}
  for every $d\ge 1$ and $m\ge 0$. Then, for every $d\ge 1$ such that $r_d\neq 0$, there exists $m\ge \left\lfloor\frac{r_d}{2k}\right\rfloor$ and $s \in E_d^m\minus E_d^{m+1}$ satisfying
  \begin{align*}
    \max_{\sigma} \log \|s\|_{\sigma} \le \max\{0,-2\hat{\mu}(\overline{E}_d)\} + a_{d,m} + \frac{1}{2}\log r_d  + \frac{\log |\Delta_K|}{2[K:\QQ]} + \frac{1}{2}\log [K:\QQ]\text{.}
  \end{align*}
\end{prop}

\begin{proof}
  Let $d\ge 1$ such that $r_d\neq 0$ and set
  \begin{align*}
   m' \defeq \left\lfloor\frac{r_d}{2k}\right\rfloor\text{,}
  \end{align*}
  so that, by (\ref{est1}),
  \begin{align*}
    r_d^{m'}\defeq \rk E_d^{m'} \ge r_d - km' \ge \frac{1}{2}r_d>0\text{.}
  \end{align*}

  By Proposition \ref{genprop} (2), we have
  \begin{align*}
  \hat{\mu}(\overline{E_d}) = \frac{r_d^{m'}}{r_d} \hat{\mu}(\overline{E_d^{m'}}) + \frac{1}{r_d}\sum_{0\le i < m'} \rk (E_d^i/E_d^{i+1})\hat{\mu}(\overline{E_d^i/E_d^{i+1}})\text{.} 
  \end{align*}
  (When $E^{i}_d/E_d^{i+1}=0$, $\rk (E_d^i/E_d^{i+1}) = 0$ and $\hat{\mu}(\overline{E_d^i/E_d^{i+1}}) = -\infty$, so that $\rk (E_d^i/E_d^{i+1})\hat{\mu}(\overline{E_d^i/E_d^{i+1}}) = 0$ by convention.) Using hypotheses (\ref{est1}) and (\ref{est2}), and that $(a_{d,m})$ is non-decreasing in $m$ for every $d$, we obtain
  \begin{align*}
   \hat{\mu}(\overline{E_d}) \le \frac{r_d^{m'}}{r_d}\hat{\mu}(\overline{E_d^{m'}}) + \frac{m'ka_{d,m'}}{r_d}\text{,}
  \end{align*}
  or, equivalently,
  \begin{align*}
   -\hat{\mu}(\overline{E_d^{m'}})\le - \frac{r_d}{r_d^{m'}}\hat{\mu}(\overline{E_d}) + \frac{m'k}{r_d^{m'}}a_{d,m'}\text{.}
  \end{align*}
  Since $r_d\le 2r_{d}^{m'}$ and $m' \le \frac{r_d}{2k}$, we conclude that
  \begin{align*}
  -\hat{\mu}(\overline{E_d^{m'}}) \le \max\{0,-2\hat{\mu}(\overline{E_d})\} + a_{d,m'}\text{.} 
  \end{align*}
  
  Let $s \in E_d^{m'}\minus\{0\}$ be such that $\max_{\sigma }\|s\|_{\sigma} = \lambda_1(\overline{E_d^{m'}})$. Then Proposition \ref{minkthm} yields
  \begin{align*}
    \max_{\sigma }\log \|s\|_{\sigma} &\le -\hat{\mu}(\overline{E_d^{m'}}) + \frac{1}{2}\log r_d^{m'} +\frac{\log |\Delta_K|}{2[K:\QQ]} + \frac{1}{2}\log [K:\QQ]\\
    &\le \max\{0,-2\hat{\mu}(\overline{E_d})\} + a_{d,m'}+\frac{1}{2}\log r_d^{m'}+\frac{\log |\Delta_K|}{2[K:\QQ]} + \frac{1}{2}\log [K:\QQ]\text{.}
  \end{align*}
 Thus $m \defeq \max\{i \in \NN \mid s \in E_d^i\} \ge m'$ satisfies the conclusion of our statement.
\end{proof}

\subsection{Proof of Theorem \ref{auxsec}} \label{proofauxsec}

Consider the notation and hypotheses of Theorem \ref{auxsec}. Let us first observe that if Theorem \ref{auxsec} holds for some particular choice of Hermitian metric $(\| \ \|_{\sigma})_{\sigma : K \into \CC}$ on $L$, then a similar statement holds for any other choice of metric, up to modifying the constant $c_2$. We may thus assume that each $(L_{\sigma}, \| \ \|_{\sigma})$ is a positive Hermitian line bundle on $\mathcal{X}_{\sigma}^{\an}$.

For every integer $d\ge 1$,
\begin{align*}
E_d \defeq \Gamma(\mathcal{X},L^{\tensor d})
\end{align*}
is a projective $\mathcal{O}_K$-module of finite type. For each field embedding $\sigma:K \into \CC$, we may consider the uniform norm $\| \ \|_{L^{\infty}(\mathcal{X}_{\sigma}^{\an})}$ on $E_{d,\sigma}$ induced by the Hermitian metric $\| \ \|_{\sigma}$ on $L_{\sigma}$.

Note that the norm $\| \ \|_{L^{\infty}(\mathcal{X}_{\sigma}^{\an})}$ is not Hermitian in general. We denote by $\| \ \|_{d,\sigma}$ the \emph{John norm} on $E_{d,\sigma}$ attached to $\| \ \|_{L^{\infty}(\mathcal{X}_{\sigma}^{\an})}$ (cf. \cite{bost15} Appendix F); this is a Hermitian norm on $E_{d,\sigma}$ satisfying
\begin{align}\label{johncomp}
\| \ \|_{L^{\infty}(\mathcal{X}_{\sigma}^{\an})} \le \| \ \|_{d,\sigma} \le (2\rk E_d)^{1/2}\| \ \|_{L^{\infty}(\mathcal{X}_{\sigma}^{\an})}\text{.}
\end{align}
We may thus consider the Hermitian vector bundle over $\Spec \mathcal{O}_K$
\begin{align*}
 \overline{E_d} \defeq (E_d, (\| \ \|_{d,\sigma})_{\sigma: K \into \CC})\text{.} 
\end{align*}

We define a decreasing filtration $(E_d^m)_{m\ge 0}$ by saturated submodules on $E_d$ via
\begin{align*}
E_d^m \defeq \{s \in E_d \mid \ord_0 \hat{\varphi}^*s \ge m\}\text{.}
\end{align*}
Since the image of $\hat{\varphi}_K: \Spf K[\![q]\!]\to \mathcal{X}_K$ is Zariski-dense, $(E_d^m)_{m\ge 0}$ is a separated filtration. The subquotients $E_d^m/E_{d}^{m+1}$ bear Hermitian vector bundle structures $\overline{E_d^m/E_{d}^{m+1}}$ induced by $\overline{E_d}$.

Let us denote by $\Omega$ the fiber of coherent sheaf $\Omega^1_{\Spf \mathcal{O}_K[\![q]\!]/\mathcal{O}_K}$ at the point of $\Spf \mathcal{O}_K[\![q]\!]$ given by the ideal $(q)\subset \mathcal{O}_K[\![q]\!]$. This is a trivial $\mathcal{O}_K$-module generated by $dq$. In what follows, we endow $\Omega$ with a structure of Hermitian line bundle $\overline{\Omega}$, defined by
\begin{align*}
\| \alpha dq\|_{\sigma}=|\sigma(\alpha)|
\end{align*}
for any $\alpha \in \mathcal{O}_K$ and any embedding $\sigma: K \into \CC$. Observe that $\overline{\Omega}$ is isomorphic to the trivial Hermitian line bundle, and therefore $\adeg \overline{\Omega} = 0$.
 
 Let $\hat{\varphi}(0) : \Spec \mathcal{O}_K \to \mathcal{X}$ denote the reduction of $\hat{\varphi}$ modulo $q$, i.e., the composition of $\hat{\varphi}$ with the closed immersion $\Spec \mathcal{O}_K \to \Spf \mathcal{O}_K[\![q]\!]$ associated to the ideal $(q)$. The Hermitian structure on $L$ endows $\hat{\varphi}(0)^*L$ with the structure of a Hermitian line bundle over $\Spec \mathcal{O}_K$.

 For every integers $d\ge 1$ and $m\ge 0$, we have an injective $\mathcal{O}_K$-linear map
 \begin{align*}
   \gamma^m_d : E_d^m/E_d^{m+1} \to \hat{\varphi}(0)^*L^{\tensor d} \tensor_{\mathcal{O}_K} \Omega^{\tensor m}
 \end{align*}
 defined by mapping the class $[s] \in E_d^m/E_d^{m+1}$ of $s \in E_d^m$  to $j^m_0 \hat{\varphi}^*s$ (the jet of order $m$ at $q=0$ of $\hat{\varphi}^*s$).
 
 \begin{lemma}\label{l2}
 There exist constants $\kappa_1,\kappa_2>0$ such that for every integers $d\ge 1$ and $m\ge 0$ we have
 \begin{align*}
h(\gamma^m_{d,K}) \le \kappa_1d + \kappa_2 d \log^+ m\text{,}
 \end{align*}
where $h$ denotes the height of $\gamma^m_{d,K}$ with respect to the Hermitian vector bundles $\overline{E_d^m/E_{d}^{m+1}}$ and $\hat{\varphi}(0)^*\overline{L}^{\tensor d} \tensor_{\mathcal{O}_K} \overline{\Omega}^{\tensor m}$.
 \end{lemma}

\begin{proof}
  Let $\sigma : K \into \CC$ be a field embedding. Since $\varphi_{\sigma} : D_{R_{\sigma}} \to \mathcal{X}_{\sigma}^{\an}$ has moderate growth, it follows from Proposition \ref{prop1mg} that there exist constants $\kappa_{1,\sigma},\kappa_{2,\sigma}>0$ such that, for any integers $d\ge 1$ and $m\ge 0$ for which and $E^m_d/E^{m+1}_d\neq 0$, and any $s \in E_{d,\sigma}^m\minus E_{d,\sigma}^{m+1}$, we have
  \begin{align}\label{l2e}
   \log \|j^m_0 \varphi_{\sigma}^*s\|_{R_{\sigma}} - \log \|s\|_{L^{\infty}(\mathcal{X}_{\sigma}^{\an})} \le \kappa_{1,\sigma}d + \kappa_{2,\sigma}d\log^+ m\text{.}
  \end{align} 
  Here, the norm $\| \ \|_{R_{\sigma}}$ is the norm $\| \ \|_r$, introduced before Proposition \ref{propmain}, in the special case where $r = R_{\sigma}$ and $\overline{L}$ is $\varphi_{\sigma}^*L_{\sigma}$ equipped with the pullback of $\| \ \|_{\sigma}$.

  Note that
  \begin{align*}
    \|j_0^m\hat{\varphi}^*s\|_{\sigma} = R_{\sigma}^{-m}\|j_0^m\varphi_{\sigma}^*s\|_{R_{\sigma}}\text{,}
  \end{align*}
  where $\|j_0^m\hat{\varphi}^*s\|_{\sigma}$ denotes the norm of $j_0^m\hat{\varphi}^*s = \gamma_{d}^m([s])$ with respect to the Hermitian structure of $\hat{\varphi}(0)^*\overline{L}^{\tensor d}\tensor \overline{\Omega}^{\tensor m}$. The estimate (\ref{l2e}), together with (\ref{johncomp}),  shows that
  \begin{align*}
   \log \|\gamma_{d}^m\|_{\sigma} + m \log R_{\sigma} \le \kappa_{1,\sigma}d + \kappa_{2,\sigma}d\log^+m\text{.}  
  \end{align*}
  
  Since $\prod_{\sigma: K \into \CC}R_{\sigma}=1$, we obtain
  \begin{align*}
   \sum_{\sigma} \log \|\gamma^m_d\|_{\sigma} \le \left(\sum_{\sigma}\kappa_{1,\sigma} \right)d + \left(\sum_{\sigma}\kappa_{2,\sigma}\right)d\log^+ m\text{.}
  \end{align*}
Since $\gamma_d^m$ is defined over $\mathcal{O}_K$, we have $\|\gamma_{d,K}^m\|_{\pp} \le 1$ for every maximal ideal $\pp$ of $\mathcal{O}_K$, so that
\begin{align*}
h(\gamma_{d,K}^m) \le \frac{1}{[K:\QQ]} \sum_{\sigma}\log \|\gamma_d^m\|_{\sigma} \le \frac{1}{[K:\QQ]} \left(\sum_{\sigma}\kappa_{1,\sigma} \right)d+ \frac{1}{[K:\QQ]}\left(\sum_{\sigma}\kappa_{2,\sigma} \right)d\log^+ m\text{.}
\end{align*}
\end{proof}

\begin{proof}[End of proof of Theorem \ref{auxsec}]
  Let us first remark that, as $L_K$ is ample, we have
  \begin{align*}
   r_d = \rk E_d = \dim \Gamma(\mathcal{X}_K,L_K^{\tensor d}) \sim_{d\rightarrow +\infty} \frac{\deg_{L_K} \mathcal{X}_K}{n!}d^n\text{.}  
  \end{align*}
  In particular, $\log r_d = O(d)$ as $d\rightarrow +\infty$.   We shall apply Proposition \ref{lemmaaux} for $(E_{d}^m)_{d\ge 1,m\ge 0}$ defined as above. This suffices by the estimates (\ref{johncomp}).

  Note that condition (\ref{est1}) is trivially verified for $k=1$. Moreover, by the same argument of \cite{bost01} Proposition 4.4 (cf. \cite{bost01} Lemma 4.1) and by the estimates (\ref{johncomp}), there exists a constant $c>0$ such that $-\hat{\mu}(\overline{E_d})\le cd$ for every $d\ge 1$.  Thus, to finish our proof, it is sufficient to find constants $a,b>0$ such that
  \begin{align*}
  \hat{\mu}(\overline{E_d^m/E_{d}^{m+1}})\le a_{d,m}\defeq ad + b d\log^+ m
  \end{align*}
  for every $d\ge 1$ and $m\ge 0$ (condition (\ref{est2})).

  By Lemma \ref{l2}, there exist constants $\kappa_1,\kappa_2>0$ such that, for every $d\ge 1$ and $m\ge 0$ such that $E_d^m/E_d^{m+1}\neq 0$, we have
  \begin{align*}
   h(\gamma_{d,K}^m)\le \kappa_1d + \kappa_2d\log^+m\text{.}
  \end{align*}
  Thus, since $\gamma_{d,K}^m$ is injective, we may apply the Slope Inequality (Proposition \ref{slopeineq}) to obtain
  \begin{align*}
   \hat{\mu}(\overline{E_d^m/E_d^{m+1}}) \le\hat{\mu}(\hat{\varphi}(0)^*\overline{L}^{\tensor d} \tensor_{\mathcal{O}_K} \overline{\Omega}^{\tensor m}) + h(\gamma_{d,K}^m)\le  (\kappa_1 + \hat{\mu}(\hat{\varphi}(0)^*\overline{L}))d + \kappa_2d\log^+m\text{.}
  \end{align*}
\end{proof}

\section{Derivatives of sections of line bundles along vector fields}\label{derivesections}

A crucial step in Nesterenko's method involves applying a certain differential operator (deduced from the Ramanujan equations) to auxiliary polynomials. It is also important to understand how this differential operator affects the degree and the norm $\| \ \|_{\infty}$ of a polynomial. 

Our generalization of Nesterenko's proof replaces polynomials of degree $d$ by global sections of the $d$th tensor power of some ample line bundle. In this section we explain how to differentiate global sections of tensor powers of a line bundle $L$ along a vector field $v$. Under a projectivity hypothesis, we also explain how $L^{\infty}$ norms with respect to some Hermitian metric on $L$ are affected by a differential operator deduced from $v$.

\subsection{The basic definition}

Let $M$ be a compact connected complex manifold, and $L$ be a line bundle over $M$ endowed with a global holomorphic section $s_0\in \Gamma(M,L)\minus \{0\}$. To $L$ is associated the graded ring $R = \bigoplus_{d\ge 0}R_d$, where $R_d \defeq \Gamma(M,L^{\tensor d})$.

Let $v$ be a meromorphic vector field on $M$, and assume that $v$ is holomorphic on the open subset $M_{s_0} \defeq \{p\in M \mid s_0(p)\neq 0\}$. Then there is a smallest integer $k\ge 0$, the ``order of pole of $v$ at $\div(s_0)$'', such that $v\tensor s_0^{\tensor k}$ defines a global holomorphic section of $TM \tensor L^{\tensor k}$.

The vector field $v$ induces a $\CC$-derivation of degree $k+1$ of the graded ring $R$
$$
\partial_v : R \to R
$$
given as follows. By definition, $\partial_v$ is the zero map on $R_0 \cong \CC$. Let $d\ge 1$ be an integer, $s \in R_d$, and $f: M_{s_0}\to \CC$ be the holomorphic function for which $s = fs_0^{\tensor d}$ over $M_{s_0}$. Then $\partial_vs\in R_{d+k+1}$ is defined as the unique global section of $L^{\tensor d + k + 1}$ such that $\partial_vs =v(f)s_0^{\tensor d+k+1}$ over $M_{s_0}$. The next lemma guarantees that this is well defined.

\begin{lemma}
  With the above notations, $v(f)s_0^{\tensor d+k+1}$ extends to a global holomorphic section of $L^{\tensor d + k +1}$.
\end{lemma}

\begin{proof}
  The couple $(L,s_0)$ corresponds canonically to an effective analytic Cartier divisor $E$ on $M$; let $(g_i,U_i)_{i\in I}$ be a family of local equations of $E$ for some open covering $M=\bigcup_{i\in I}U_i$. By hypothesis, for every $i\in I$, $g_i^kv$ extends to a holomorphic vector field on $U_i$.

  An element $s \in R_d$ might be identified with a meromorphic function $f = s/ s_0^{\tensor d}$ on $M$ having pole of order at most $d$ on $E$, i.e., such that $g_i^df$ defines a holomorphic function on $U_i$ for every $i\in I$. Under this identification, our statement is equivalent to the assertion that $g_i^{d+k+1}v(f)$ defines a holomorphic function on $U_i$ for every $i\in I$.

  Now, for $i\in I$, we have
  \begin{align*}
g_i^{k+1}v(g_i^df) = g_i^{k+1}(d\cdot g_i^{d-1}v(g_i)f + g_i^dv(f)) = d\cdot g_i^kv(g_i)g_i^df + g_i^{d+k+1}v(f)\text{,} 
\end{align*}
so that $g_i^{d+k+1}v(f)$ defines a holomorphic function on $U_i$.
\end{proof}

Finally, it is easy to see that the $\CC$-linear map $\partial_v:R \to R$ satisfies Leibniz's rule: if $s \in R_d$ and $t \in R_e$, then
  \begin{align*}
   \partial_{v}(s\tensor t)  = \partial_vs \tensor t + s \tensor \partial_vt
  \end{align*}
  in $R_{d+e+k+1}$.

\subsection{Estimates of uniform norms}

Let us keep the notation of the last paragraph and fix once and for all some $t \in R_{k+1}$ (recall that $k$ denotes the ``order of pole of $v$ at $\div(s_0)$'').

For any integer $j\ge 1$ and $d\ge 1$, we define a differential operator of degree $j(k+1)$
\begin{align*}
  \partial_v^{[j]} : R \to R
  \end{align*}
  as the composition
  \begin{align*}
   \partial_v^{[j]} = \partial_v \circ (\partial_v - t) \circ \cdots \circ (\partial_v -(j-1)t)\text{.}
  \end{align*}

\begin{prop}\label{bounddersec}
With the above notation, assume moreover that $L = i^*\mathcal{O}(1)$ for some projectively normal embedding $i:M \into \PP^n(\CC)$. Let $\| \ \|$ be a Hermitian metric on $L$. Then there exists a constant $C>0$ such that, for any integer $j\ge 1$,  any sufficiently large positive integer $d$, and any $s \in R_d$, we have
  \begin{align*}
    \|\partial^{[j]}_vs\|_{L^{\infty}(M)} \le C^{j+d}(j+d)^{j} \|s\|_{L^{\infty}(M)}\text{,}
  \end{align*}
  where $\|\partial^{[j]}_v(s)\|_{L^{\infty}(M)}$ (resp. $\|s\|_{L^{\infty}(M)}$) denotes the uniform norm on $M$ with respect to the Hermitian metric on $L^{\tensor d + j(k+1)}$ (resp. $L^{\tensor d}$) induced by $\| \ \|$. 
\end{prop}

Our proof is a reduction to the case $M=\PP^n(\CC)$. Let $\| \ \|$ denote the Fubini-Study metric on the line bundle $\mathcal{O}(1)$ over  $\PP^n_{\CC} = \Proj \CC[X_0,\ldots,X_n]$ (cf. Example \ref{polgrowth}), and let us identify $\Gamma(\PP^n_{\CC},\mathcal{O}(d))$ with the $\CC$-vector space $\CC[X_0,\ldots,X_n]_d$ of homogeneous polynomials of degree $d$. If $P = \sum_{|I|=d}a_IX^I$, we consider the norms
  \begin{align*}
    \|P\|_{\infty} \defeq \max_{|I|=d}|a_I|\ \ \text{ and } \ \  \|P\|_1\defeq \sum_{|I|=d}|a_I|\text{.}  
  \end{align*}
  The uniform norm of $P$, seen as an element of $\Gamma(\PP^n(\CC),\mathcal{O}(d))$, with $\mathcal{O}(1)$ equipped with the Fubini-Study metric, is given by
  \begin{align*}
\|P\|_{L^{\infty}(\PP^n(\CC))} = \sup_{z \in \CC^{n+1}\minus\{0\}}\frac{|P(z)|}{\left(\sum_{i=0}^n|z_i|^2\right)^{\frac{d}{2}}}\text{.}
  \end{align*}
  
\begin{lemma}\label{compnorm}
  For any $P \in \Gamma(\PP^n_{\CC},\mathcal{O}(d))$, we have
  \begin{align*}
  (n+1)^{-\frac{d}{2}}\|P\|_{\infty} \le \|P\|_{L^{\infty}(\PP^n(\CC))}\le  \|P\|_{1} \le \binom{d+n}{n}\|P\|_{\infty}\text{.}
  \end{align*}
\end{lemma}  

\begin{proof}
  If we write $P = \sum_{|I|=d} a_IX^I$, then Cauchy's integral formula gives, for any multi-index $I$,
  \begin{align*}
   a_I =\frac{1}{(2\pi i)^{n+1}}\int_{(\partial D)^{n+1}}\frac{P(z)}{z^{I+\mathbf{1}}}dz_0\cdots dz_n\text{,}
  \end{align*}
  where $D$ denotes the unit disk in $\CC$ and $\mathbf{1}$ the multi-index of order $n+1$ having $1$ at each coordinate. Thus, if $[z]$ denotes the image in $\PP^n(\CC)$ of a point $z \in \CC^{n+1}\minus\{0\}$, 
  \begin{align*}
  |a_I| \le \sup_{z \in (\partial D)^{n+1}}|P(z)| = (n+1)^{\frac{d}{2}}\sup_{z \in (\partial D)^{n+1}}\|P([z])\|\text{.}
  \end{align*}
  This proves that $(n+1)^{-\frac{d}{2}}\|P\|_{\infty} \le \|P\|_{L^{\infty}(\PP^n(\CC))}$.

  For any $z \in \CC^{n+1}\minus\{0\}$, we have
  \begin{align*}
   \|P([z])\| = \frac{|P(z)|}{\left(\sum_{i=0}^n |z_i|^2\right)^{\frac{d}{2}}} \le \frac{\sum_{|I|=d}|a_I||z^I|}{\left(\sum_{i=0}^n |z_i|^2\right)^{\frac{d}{2}}} \le \frac{\max_{|I|=d}|z^I|}{\left(\sum_{i=0}^n |z_i|^2\right)^{\frac{d}{2}}}\|P\|_1\text{.}
  \end{align*}
Now, if $I=(i_0,\ldots,i_n)$ is a multi-index satisfying $|I|=d$, then it is clear that
  \begin{align*}
   |z^I|^2 = (|z_0|^2)^{i_0}\cdots (|z_n|^2)^{i_n} \le \left(\sum_{i=0}^n|z_i|^{2}\right)^d\text{.}
  \end{align*}
  We thus obtain $\|P\|_{L^{\infty}(\PP^n(\CC))} \le \|P\|_1$.

  The inequality $\|P\|_{1} \le \binom{n+d}{n} \|P\|_{\infty}$ is an immediate consequence of $\dim \Gamma(\PP^n(\CC),\mathcal{O}(d)) = \binom{n+d}{n}$.
\end{proof}

\begin{proof}[Proof of Proposition \ref{bounddersec}]
  Since $M$ is compact, if the conclusion of the statement holds for some Hermitian metric $\| \ \|$, then, up to replacing the constant $C$, it also holds for any other Hermitian metric on $L$. We may thus assume that $\| \ \|$ is induced by the Fubini-Study metric on $\mathcal{O}(1)$ via the embedding $i$.

  Let $(X_0,\ldots ,X_n)$ denote the projective coordinates of $\PP^n(\CC)$, seen as global sections of $\mathcal{O}(1)$, and let $t_j\in R_1$ be the restriction of $X_j$ to $M$ for every $0\le j \le n$. Since $i: M \into \PP^n(\CC)$ is projectively normal, for any integer $d\ge 1$, $R_d$ is generated as a $\CC$-vector space by the monomials of degree $d$ in $t_0,\ldots,t_n$.

  We lift $v$ to $\PP^n(\CC)$ as follows. For every $0\le j \le n$, let $P_j \in \Gamma(\PP^n(\CC),\mathcal{O}(k+2)) = \CC[X_0,\ldots,X_n]_{k+2}$ be a lifting of $\partial_vt_j \in R_{k+2}$. Then there exists a unique  $\CC$-derivation $\partial$ of $\bigoplus_{d\ge 0}\Gamma(\PP^n(\CC),\mathcal{O}(d))=\CC[X_0,\ldots,X_n]$, of degree $k+1$, such that $\partial X_j = P_j$ for every $0\le j \le n$. It is easy to see that, for every integer $d\ge 0$, the diagram
 $$
 \begin{tikzcd}
   \CC[X_0,\ldots,X_n]_d \arrow{r}{\partial} \arrow{d}[swap]{i^*} & \CC[X_0,\ldots,X_n]_{d+k+1} \arrow{d}{i^*}\\
   R_d \arrow{r}{\partial_v} & R_{d+k+1}
 \end{tikzcd}
 $$
 commutes. Moreover, if $Q \in \CC[X_0,\ldots,X_n]_{k+1}$ is any lifting of $t \in R_{k+1}$, then it is clear that, for any $j\ge 1$,
  \begin{align*}
   \partial^{[j]} \defeq \partial \circ (\partial - Q) \circ \cdots \circ (\partial - (j-1)Q)
  \end{align*}
  makes the diagram
$$
 \begin{tikzcd}
   \CC[X_0,\ldots,X_n]_d \arrow{r}{\partial^{[j]}} \arrow{d}[swap]{i^*} & \CC[X_0,\ldots,X_n]_{d+j(k+1)} \arrow{d}{i^*}\\
   R_d \arrow{r}{\partial^{[j]}_v} & R_{d+j(k+1)}
 \end{tikzcd}
 $$
 commute.

For every multi-index $I\in \NN^{n+1}$, we have
\begin{align*}
\| \partial X^I\|_{\infty} \le |I| \max_{0\le i \le n}\|P_i\|_{\infty}\text{.}
\end{align*}
This implies that, for any $d\ge 1$ and any homogeneous polynomial $P \in \CC[X_0,\ldots,X_n]_d$,
\begin{align*}
 \|\partial P\|_{\infty} \le  d\left(\max_{0\le i \le n}\|P_i\|_{\infty}\right) \|P\|_{\infty}\text{.}
\end{align*}
Thus, if $\kappa \defeq  \|Q\|_{\infty}+(k+1)\max_{0\le i \le n}\|P_i\|_{\infty} $ and $S\in \CC[X_0,\ldots,X_n]$ is a homogeneous polynomial of degree $d+m(k+1)$ for some $0\le m\le j-1$, we have
\begin{align*}
\|(\partial - mQ)S\|_{\infty} \le \left((d+m(k+1))\left(\max_{0\le i \le n}\|P_i\|_{\infty}\right) + m \|Q\|_{\infty}\right)\|S\|_{\infty}\le \kappa(d+j)\|S\|_{\infty}\text{.}
\end{align*}
By induction, we conclude that, for any $d\ge 1$ and any $P\in \CC[X_0,\ldots,X_n]_d$, we have
\begin{align}\label{diffbound}
  \|\partial^{[j]}P\|_{\infty} \le  \kappa^j (d+j)^j\|P\|_{\infty}\text{.}
\end{align}

To complete our proof, we apply a lifting argument. By \cite{bost04} Proposition 3.5, there exists a constant $C_0>0$ such that, for every sufficiently large integer $d$ and every $s \in \Gamma(M, L^{\tensor d})$, there exists a lifting $P \in \Gamma(\PP^n_{\CC},\mathcal{O}(d))$ of $s$ such that
\begin{align}\label{boundlift}
\|P\|_{L^{\infty}(\PP^n(\CC))} \le C^d_0\|s\|_{L^{\infty}(M)}\text{.}
\end{align}
Thus, for any $j\ge 1$,
\begin{align*}
  \|\partial^{[j]}_vs\|_{L^{\infty}(M)} &\le \|\partial^{[j]}P\|_{L^{\infty}(\PP^n(\CC))} \le  \|\partial^{[j]}P\|_{1} \le \binom{d+ j(k+1)}{j(k+1)}\|\partial^{[j]}P\|_{\infty} && \text{by Lemma \ref{compnorm}}\\
    &\le 2^{d+ j(k+1)}\|\partial^{[j]}P\|_{\infty}  \\
                                        &\le 2^{d+ j(k+1)} \kappa^j(d+j)^j\|P\|_{\infty} &&\text{by (\ref{diffbound})}\\
                                        &\le 2^{d+ j(k+1)} \kappa^j(d+j)^j(n+1)^{\frac{d}{2}}\|P\|_{L^{\infty}(\PP^n(\CC))} && \text{by Lemma \ref{compnorm}}\\
                 &\le 2^{d+ j(k+1)} \kappa^j(d+j)^j(n+1)^{\frac{d}{2}}C_0^d\|s\|_{L^{\infty}(M)} && \text{by (\ref{boundlift})}\text{.}
\end{align*}
\end{proof}

\subsection{The arithmetic case}

We shall actually need an arithmetic variant of the above constructions.

Consider the notation and terminology of Paragraph \ref{notstat}. Let $K$ be a number field, $\mathcal{X}$ be a projective arithmetic scheme over $S=\Spec \mathcal{O}_K$ with smooth generic fiber, and $L$ be a line bundle over $\mathcal{X}$ endowed with a global section $s_0 \in \Gamma(\mathcal{X}, L)\minus\{0\}$. Arguing as above, we see that a section $w \in \Gamma(\mathcal{X}_{s_0},\mathcal{D}er_{\mathcal{O}_S}(\mathcal{O}_{\mathcal{X}}))$ induces an $\mathcal{O}_K$-derivation $\partial_w$ of the ring $\bigoplus_{d\ge 0}\Gamma(\mathcal{X},L^{\tensor d})$.

Let us fix $t \in \Gamma(\mathcal{X}, L^{k+1})$, where $k\ge 0$ is the ``order of pole of $w$ at $\div (s_0)$'', and consider the differential operators $\partial_w^{[j]} = \partial_w \circ (\partial_w - t) \circ \cdots \circ (\partial_w -(j-1)t)$, for $j\ge 0$, as above.

By applying Proposition \ref{bounddersec} for each projective embedding $\sigma:K \into \CC$, we obtain the following corollary.

\begin{coro}\label{bounddersecar}
With the above notation, assume moreover that $L = i^*\mathcal{O}(1)$ for some closed immersion $i:\mathcal{X} \into \PP^n_{\mathcal{O}_K}$ over $S$ such that $i_K: \mathcal{X}_K \into \PP^n_K$ is projectively normal. Let $(\| \ \|_{\sigma})_{\sigma: K \into \CC}$ be a Hermitian structure on $L$. Then, there exists a constant $C>0$ such that, for any integer $j\ge 1$,  any sufficiently large positive integer $d$, and any $s \in \Gamma(\mathcal{X},L^{\tensor d})$, we have
  \begin{align*}
    \|\partial^{[j]}_ws\|_{\mathcal{X}} \le C^{j+d}(j+d)^{j} \|s\|_{\mathcal{X}}\text{.}
  \end{align*}
\end{coro}

\section{Proof of Theorem \ref{lethm}}\label{finalproof}

Recall the notation and hypotheses of Theorem \ref{lethm}: $\mathcal{X}$ is a quasi-projective arithmetic scheme over $\mathcal{O}_K$ of relative dimension $n\ge 2$ with smooth generic fiber, and $\hat{\varphi}: \Spf \mathcal{O}_K[\![q]\!] \to \mathcal{X}$ is a morphism of formal $\mathcal{O}_K$-schemes such that
\begin{enumerate}
   \item [(i)] the formal curve $\hat{\varphi}_K : \Spf K[\![q]\!] \to \mathcal{X}_K$ is ZL-dense in $\mathcal{X}_K$ and satisfies the differential equation
   \begin{align*}
    q\frac{d\hat{\varphi}_K}{dq} = v\circ \hat{\varphi}_K\text{;}
   \end{align*}
   \item [(ii)] for any field embedding $\sigma : K \into \CC$, the formal curve $\hat{\varphi}_{\sigma} : \Spf \CC[\![q]\!] \to \mathcal{X}_{\sigma}$ lifts to an analytic curve $\varphi_{\sigma}: D_{R_{\sigma}}\subset \CC \to \mathcal{X}_{\sigma}^{\an}$ of moderate growth. We also assume that $\prod_{\sigma: K \into \CC}R_{\sigma}=1$. 
\end{enumerate}

Let $\overline{\mathcal{X}}$ be some projective compactification with smooth generic fiber of the arithmetic variety $\mathcal{X}$ over $\mathcal{O}_K$. Fix a Hermitian line bundle $\overline{L} = (L, (\| \ \|_{\sigma})_{\sigma: K \into\CC})$ over $\overline{\mathcal{X}}$ such that $L_K$ is ample and $(L_{\sigma},\| \ \|_{\sigma})$ over $\mathcal{X}_{\sigma}^{\an}$ is positive for every $\sigma: K \into \CC$.

In view of Philippon's algebraic independence criterion (Theorem \ref{aic}), Theorem \ref{lethm} will be a direct consequence of the following.

\begin{theorem}\label{lethm'}
With the above notation, for any field embedding $\sigma:K \into \CC$, and any $z \in D_{R_{\sigma}}\minus \{0\}$, there exist real constants $c_0,c_1,c_2,c_3>0$ such that, for every sufficiently large positive integer $d$, there exists a positive integer $d'\le c_0 d\log d$, and $t \in \Gamma(\overline{\mathcal{X}},L^{\tensor d'})$ satisfying
  \begin{align*}
\log \|t\|_{\overline{\mathcal{X}}} \le c_1 d\log^2d
  \end{align*}
  and
  \begin{align*}
-c_2d^n \le \log \|t(\varphi_{\sigma}(z))\|_{\sigma}\le -c_3 d^n\text{.}
  \end{align*}
\end{theorem}

We shall prove this theorem in three steps corresponding to the next three lemmas.

\begin{lemma}[Auxiliary sections]\label{lemma1}
There exist constants $a,b,c>0$ such that, for every sufficiently large positive integer $d$, there is a global section $s \in \Gamma(\overline{\mathcal{X}},L^{\tensor d})$ such that
  \begin{align}\label{boundorderzero}
    ad^n < \ord_0 \hat{\varphi}^*s \le b d^n 
  \end{align}
  and
  \begin{align}\label{boundnorm}
    \log \|s\|_{\overline{\mathcal{X}}}\le  cd\log d \text{.}
  \end{align}
\end{lemma}

\begin{proof}
 Since $\hat{\varphi}_K$ is ZL-dense in $\mathcal{X}_K$, and $\varphi_{\tau}$ has moderate growth in $\mathcal{X}_{\tau}^{\an}$ for every embedding $\tau: K \into \CC$, our statement follows immediately from Corollary \ref{coroauxsec}.
\end{proof}

Fix a field embedding $\sigma :K \into \CC$ and $z \in D_{R_{\sigma}}\minus \{0\}$. By the projective Prime Avoidance Lemma, there is an integer $k\ge 1$, and a global section $s_0 \in \Gamma(\overline{\mathcal{X}},L^{\tensor k})$ such that $\overline{\mathcal{X}}_{s_0} \defeq (s_0\neq 0)\subset \mathcal{X}$ and $\varphi_{\sigma}(z)\in \overline{\mathcal{X}}_{s_0,\sigma}^{\an}$. Up to replacing $L$ by $L^{\tensor k}$, we may assume that $k=1$ (cf. Remark \ref{rkinv}).

\begin{lemma}\label{lemma2}
  There exist constants $\gamma_0,\gamma_1,\gamma_2>0$ such that, for every sufficiently large positive integer $d$, and every $s \in \Gamma(\overline{\mathcal{X}},L^{\tensor d})$ as in Lemma \ref{lemma1}, there exists $j\le \gamma_0d\log d$ such that, if we write $s = fs_0^{\tensor d}$ over $\overline{\mathcal{X}}_{s_0}$, then
  \begin{align*}
  -\gamma_1d^n\le \log |(\varphi_{\sigma}^*f)^{(j)}(z)|\le -\gamma_2d^n\text{.} 
  \end{align*}
\end{lemma}

\begin{proof}
  Let $d$ be a sufficiently large positive integer and $s \in \Gamma(\overline{\mathcal{X}},L^{\tensor d})$ be as in Lemma \ref{lemma1}. Set $m \defeq \ord_0\hat{\varphi}^*s$.

  According to Proposition \ref{prop1mg} and to the bounds (\ref{boundorderzero}) and (\ref{boundnorm}), for every embedding $\tau \neq \sigma$, there is a constant $\kappa_{\tau}$ (not depending on $d$ or $s$) such that
  \begin{align}\label{esttau}
   \log \|j^m_0\varphi_{\tau}^*s\|_{R_{\tau}} \le \kappa_{\tau}d\log d\text{.}
  \end{align}
  We recall that the norm $\| \ \|_{R_{\tau}}$ above is the norm $\| \ \|_r$, introduced before Proposition \ref{propmain}, in the special case where $r = R_{\tau}$ and $\overline{L}$ is $\varphi_{\tau}^*L_{\tau}$ equipped with the pullback of $\| \ \|_{\tau}$.
  
  Fix any constant $C> \sum_{\tau\neq \sigma}\kappa_{\tau}$. Then Corollary \ref{corotaylor2} shows that there exist real numbers $\gamma_0,\gamma_1>0$ such that, for sufficiently large $d$, if
  \begin{align}\label{boundder}
    \max_{0\le j \le \lfloor\gamma_0d\log d \rfloor}\log \frac{|(\varphi_{\sigma}^*f)^{(j)}(z)|}{j!} < -\gamma_1d^n
  \end{align}
  then
  \begin{align}\label{estsigma}
   \log \|j_0^m\varphi_{\sigma}^*s\|_{R_{\sigma}} \le -C d\log d\text{.} 
  \end{align}

  By contradiction, assume that (\ref{boundder}) holds. Observe that $j_0^m\hat{\varphi}^*s$ is an element of $\hat{\varphi}(0)^*L^{\tensor d}\tensor \Omega^{\tensor m}$, and the Hermitian structure on \mbox{$\hat{\varphi}(0)^*\overline{L}^{\tensor d}\tensor \overline{\Omega}^{\tensor m}$} allows us to consider its norms $(\|j_0^m\hat{\varphi}^*s\|_{\tau})_{\tau:K \into \CC}$. For every field embedding $\tau:K \into \CC$, we have
  \begin{align*}
    \|j_0^m\hat{\varphi}^*s\|_{\tau} = R_{\tau}^{-m}\|j_0^m\varphi_{\tau}^*s\|_{R_{\tau}}\text{.}
  \end{align*}
  Thus, since $\prod_{\tau: K \into \CC}R_{\tau}=1$, we obtain from (\ref{esttau}) and (\ref{estsigma})
  \begin{align*}
    \sum_{\tau:K \into \CC} \log \|j_0^m\hat{\varphi}^*s\|_{\tau} \le -\left(C-\sum_{\tau\neq \sigma}\kappa_{\tau}\right)d\log d\text{.}
  \end{align*}
  On the other hand, by definition of the Arakelov degree, we have 
  \begin{align*}
  \sum_{\tau:K \into \CC} \log \|j_0^m\hat{\varphi}^*s\|_{\tau} \ge -\adeg(\hat{\varphi}(0)^*\overline{L}^{\tensor d}\tensor\overline{\Omega}^{\tensor m}) = -\adeg(\hat{\varphi}(0)^*\overline{L})d\text{.} 
  \end{align*}
  This contradicts our choice of $C$ for $d \gg 0$. We conclude that, for sufficiently large $d$, (\ref{boundder}) cannot hold, so that there exists an integer $j\le \gamma_0d\log d$ for which
  \begin{align}\label{b1}
   \log |(\varphi_{\sigma}^*f)^{(j)}(z)|\ge \log \frac{|(\varphi_{\sigma}^*f)^{(j)}(z)|}{j!} \ge -\gamma_1d^n\text{.}
  \end{align}
  
  Next, we bound $\log |(\varphi_{\sigma}^*f)^{(j)}(z)|$ from above. Let $\Delta$ be a disk centered in $z$, of radius $\varepsilon > 0$ small enough so that $\overline{\Delta} \subset \varphi_{\sigma}^{-1}(\mathcal{X}_{s_0,\sigma}^{\an})$. It follows from Proposition \ref{boundcompact}, and bounds (\ref{boundorderzero}) and (\ref{boundnorm}), that there is a constant $c'>0$ such that
  \begin{align*}
  \log\max_{\zeta \in \partial \Delta} |\varphi^*_{\sigma}f(\zeta)| \le -c'd^n\text{.}
  \end{align*}
 By the Cauchy inequalities, we have
  \begin{align*}
   \frac{|(\varphi_{\sigma}^*f)^{(j)}(z)|}{j!} \le \frac{\max_{\zeta\in \partial \Delta}|\varphi_{\sigma}^*f(\zeta)|}{\varepsilon^j}\text{,}
  \end{align*}
  so that
  \begin{align*}
  \log |(\varphi_{\sigma}^*f)^{(j)}(z)| \le -c'd^n + \log j! - j\log \varepsilon\text{.}
  \end{align*}
  Since $j=O(d\log d)$, we have $\log j! = O(d\log^2d)$, and we conclude that there is a constant $\gamma_2>0$ (not depending on $d$ or $s$) such that
  \begin{align}\label{b2}
   \log |(\varphi_{\sigma}^*f)^{(j)}(z)| \le -\gamma_2d^n
  \end{align}
  for every sufficiently large $d$.
\end{proof}

Again by Remark \ref{rkinv}, up to replacing $L$ by a sufficiently large tensor power of itself, we may assume that there exists a closed immersion $i: \overline{\mathcal{X}} \into \PP^n_{\mathcal{O}_K}$ over $\mathcal{O}_K$ such that $i_K: \overline{\mathcal{X}}_K \into \PP^n_{K}$ is projectively normal and  $L=i^*\mathcal{O}(1)$.

Fix any $\alpha \in \mathcal{O}_K\minus\{0\}$ that ``clears the denominators of $v$'', i.e., such that $w\defeq \alpha v$ defines a non-zero global section of $\mathcal{D}er_{\mathcal{O}_K}(\mathcal{O}_{\mathcal{X}})$. Let $k\ge 0$ be the smallest integer for which $w\tensor s_0^{\tensor k}$ defines a global section of $\mathcal{D}er_{\mathcal{O}_K}(\mathcal{O}_{\overline{\mathcal{X}}})\tensor L^{\tensor k}$, and let $\partial_w$ be the $\mathcal{O}_K$-derivation of degree $k+1$ of the ring $\bigoplus_{d\ge 0}\Gamma(\overline{\mathcal{X}},L^{\tensor d})$ defined in Section \ref{derivesections}. For any integer $j\ge 1$, set
\begin{align*}
  \partial_w^{[j]} = \partial_w \circ (\partial_w -\alpha s_0^{\tensor k+1})\circ \cdots \circ (\partial_w - (j-1)\alpha s_0^{\tensor k+1})\text{.}
  \end{align*}

\begin{lemma}\label{lemma3}
  There exist constants $c_1,c_2,c_3>0$  such that, for every sufficiently large positive integer $d$, and every $s \in \Gamma(\overline{\mathcal{X}},L^{\tensor d})$ as in Lemma \ref{lemma1}, if $j$ denotes the integer constructed in Lemma \ref{lemma2}, then the section $t\defeq \partial_w^{[j]}(s) \in \Gamma(\overline{\mathcal{X}},L^{\tensor d+ j(k+1)})$ satisfies
  \begin{align*}
\log \|t\|_{\overline{\mathcal{X}}} \le c_1 d\log^2d
  \end{align*}
  and
  \begin{align*}
-c_2d^n \le \log \|t(\varphi_{\sigma}(z))\|_{\sigma}\le -c_3 d^n\text{.}
  \end{align*}
\end{lemma}

\begin{proof}
Since $j$ grows at the order of $d\log d$, by Corollary \ref{bounddersecar} and bound (\ref{boundnorm}), there exists $c_1>0$ such that
  \begin{align*}
   \log \|t\|_{\overline{\mathcal{X}}} \le c_1d\log^2 d\text{.}
  \end{align*}
  
 In order to bound $\|t(\varphi_{\sigma}(z))\|_{\sigma}$, we first remark that the formal identity of differential operators
  \begin{align*}
   q^j\frac{d^j}{dq^j} = q\frac{d}{dq}\left(q\frac{d}{dq} - 1 \right)\cdots \left(q\frac{d}{dq} - (j-1) \right)
  \end{align*}
  and the differential equation
  \begin{align*}
   D\hat{\varphi}_K\left(q\frac{d}{dq} \right)=\hat{\varphi}^*_Kv
  \end{align*}
  yield:
  \begin{align*}
    \hat{\varphi}^*t &= \hat{\varphi}^*(\partial_w(\partial_w-\alpha s_0^{\tensor k+1})\cdots (\partial_w - (j-1)\alpha s_0^{\tensor k+1})(s))\\
                       & = \alpha^j\left[q\frac{d}{dq}\left(q\frac{d}{dq} - 1 \right)\cdots \left(q\frac{d}{dq} - (j-1) \right)\hat{\varphi}^*(f)\right]\hat{\varphi}^*(s_0)^{\tensor d+j(k+1)}= (\alpha q)^j\frac{d^j\hat{\varphi}^*(f)}{dq^j}\hat{\varphi}^*(s_0)^{\tensor d+j(k+1)}\text{.}
  \end{align*}
  A similar formula holds for $\varphi_{\sigma}$. Thus
  \begin{align*}
   \log \|t(\varphi_{\sigma}(z))\|_{\sigma} = \log |(\varphi_{\sigma}^*f)^{(j)}(z)| + j\log |\alpha z| + (d+j(k+1))\log \|\varphi_{\sigma}^*s_0(z)\|_{\sigma}\text{.}
  \end{align*}
  Since $j$ grows at the order of $d\log d$, we conclude from (\ref{b1}) and (\ref{b2}) that there exist real constants $c_2>c_3>0$ such that
  \begin{align*}
  -c_2d^n \le  \log \|t(\varphi_{\sigma}(z))\|_{\sigma} \le - c_3d^n
  \end{align*}
  for sufficiently large $d$.  
\end{proof}

To finish the proof, we simply remark that, if $c_0$ is any real number satisfying $c_0> (k+1)\gamma_0$, then the degree $d'\defeq d+j(k+1)$ of $t$ constructed above satisfies $d' \le d + (k+1)\gamma_0d\log d \le c_0d\log d$, for $d$ sufficiently large. 

\begin{appendix}

 \section{Philippon's algebraic independence criterion for projective varieties}\label{genphil}

 
 Let $K$ be a number field, $\mathcal{X}$ be a projective arithmetic scheme over $\mathcal{O}_K$ of relative dimension $n\ge 2$ (cf. definition in Paragraph \ref{notstat}), and $\overline{L} = (L,(\| \ \|_{\sigma})_{\sigma: K \into \CC})$ be a Hermitian line bundle over $\mathcal{X}$ with $L$ relatively ample over $\Spec \mathcal{O}_K$. Recall that, if $s \in \Gamma(\mathcal{X}, L^{\tensor d})$ for some integer $d\ge 1$, then we denote $\|s\|_{\mathcal{X}} = \max_{\sigma} \|s\|_{\sigma,L^{\infty}(\mathcal{X}_{\sigma}^{\an})}$, where $\sigma$ runs through the set of field embeddings of $K$ in $\CC$.

 The proof of the main theorem of this article relies on the following generalized version of an algebraic independence criterion of Philippon (cf. \cite{philippon86} Théorème 2.11 and \cite{nesterenko96} Lemma 2.5).
 
\begin{theorem}\label{aic}
  Let $\sigma :K \into \CC$ be a field embedding and $p \in \mathcal{X}_{\sigma}(\CC)$. Suppose that there exist an integer $m\in [2,n]$, a non-decreasing sequence of positive real numbers $(\ell_d)_{d\ge 1}$ satisfying $\ell_d^{m-1}=o(d)$ as $d\rightarrow +\infty$, and real constants $a>b>0$ such that, for every sufficiently large positive integer $d$, there exists an integer $d' \le d\ell_d$ and a section $s \in \Gamma(\mathcal{X},L^ {\tensor d'})$ satisfying
  \begin{align*}
   \log \|s\|_{\mathcal{X}} \le d\ell_d 
  \end{align*}
  and
  \begin{align*}
 -ad^m \le \log \|s_{\sigma}(p)\|_{\sigma}\le -bd^m\text{.}
  \end{align*}
  Then the field of definition $K(p)$ of the complex point $p$ in $\mathcal{X}_K$ satisfies
  \begin{align*}
  \trdeg_{\QQ} K(p) \ge m-1\text{.}
  \end{align*}
\end{theorem}

\begin{obs}\label{rkinv}
  For any integer $k\ge 1$, the conditions in the above statement are verified for the Hermitian line bundle $\overline{L}$ if and only if similar conditions hold for the tensor power $\overline{L}^{\tensor k}$ of $\overline{L}$ (up to multiplying $\ell_d$, $a$, and $b$ by suitable constants).

 Moreover, since $\mathcal{X}$ is proper over $\Spec \mathcal{O}_K$, it is easy to see that if the above statement is true for a particular choice of Hermitian structure on $L$, then it also holds for any other Hermitian structure on $L$.
\end{obs}

In what follows, we explain how to deduce the above statement from Philippon's original result concerning $\mathcal{X} = \PP^n_{\mathcal{O}_K}$. The main technical tool is the following ``integral lifting lemma''.

\begin{lemma}\label{lift}
  Let $\mathcal{X}$ and $\mathcal{Y}$ be projective arithmetic schemes over $\mathcal{O}_K$, $\overline{L}$ be a Hermitian line bundle over $\mathcal{X}$, with $L$ relatively ample over $\Spec \mathcal{O}_K$, and $\mathcal{Y}\to \mathcal{X}$ be a closed immersion over $\Spec \mathcal{O}_K$. Endow $L|_{\mathcal{Y}}$ with the induced Hermitian structure. Then, there exists a real number $C>0$ such that, for every sufficiently large positive integer $d$, any section $s \in \Gamma(\mathcal{Y}, L|_{\mathcal{Y}}^{\tensor d})$ can be lifted to a section $\tilde{s}\in \Gamma(\mathcal{X},L^{\tensor d})$ satisfying
  \begin{align*}
   \|\tilde{s}\|_{\mathcal{X}} \le C^d\|s\|_{\mathcal{Y}} \text{.}
  \end{align*}
\end{lemma}

This type of result is well known in Arakelov Geometry and goes back to Zhang's work on arithmetic ampleness \cite{zhang92}. For lack of reference, we sketch a proof.

\begin{proof}[Sketch of the proof]
If $R$ is a ring, we denote $\Gamma(\mathcal{X},L^{\tensor d})_R \defeq \Gamma(\mathcal{X},L^{\tensor d})\tensor_{\ZZ}R$. 

Let $d$ be large enough so that the restriction map $\rho_d:\Gamma(\mathcal{X},L^{\tensor d}) \to \Gamma(\mathcal{Y},L|_{\mathcal{Y}}^{\tensor d})$ is surjective. Since $\bigoplus_{\sigma}\Gamma(\mathcal{X}_{\sigma},L_{\sigma}^{\tensor d}) \cong \Gamma(\mathcal{X},L^{\tensor d})_{\CC}$, we may apply \cite{bost04} Proposition 3.5 to obtain a constant $C_0>0$, not depending on $d$ or $s$, and an element $t \in  \Gamma(\mathcal{X},L^{\tensor d})_{\RR} \subset \Gamma(\mathcal{X},L^{\tensor d})_{\CC}$ lifting $s$ (i.e. $(\rho_d\tensor 1) (t) = s\tensor 1$) and satisfying
 $$
 \|t\|_{\mathcal{X}} \le C_0^d \|s\|_{\mathcal{Y}}\text{.}
 $$
 The idea now is to define $\tilde{s}$ as the element of $\rho_d^{-1}(s)$ minimizing the distance $\|t-\tilde{s}\|_{\mathcal{X}}$ in $\Gamma(\mathcal{X},L^{\tensor d})_{\RR}$.
 
 To finish the proof, we must show that the diameter of the fundamental domain of the lattice $\Gamma(\mathcal{X},\mathcal{I}_{\mathcal{Y}}\tensor L^{\tensor d}) = \ker \rho_d$ in $\Gamma(\mathcal{X},\mathcal{I}_{\mathcal{Y}}\tensor L^{\tensor d})_{\RR}$ grows at most exponentially in $d$. We mimic the argument in the proof of \cite{charles17} Proposition 2.5. Since $L$ is ample, there exists an integer $n\ge 1$ such that, for any sufficiently large integer $r$, and any positive integer $q$, the morphism
  \begin{align*}
   \Gamma(\mathcal{X}, \mathcal{I}_{\mathcal{Y}}\tensor L^{\tensor r}) \tensor_{\ZZ} \Gamma(\mathcal{X},L^{\tensor n})^{\tensor q}  \to  \Gamma(\mathcal{X}, \mathcal{I}_{\mathcal{Y}}\tensor L^{\tensor nq+r})
  \end{align*}
  is surjective. Choose sufficiently large integers $r_1,\ldots,r_n$ forming a complete residue system modulo $n$. Fixing bases of the finite free $\ZZ$-modules $\Gamma(\mathcal{X}, L^{\tensor n}), \Gamma(\mathcal{X},\mathcal{I}_{\mathcal{Y}}\tensor L^{\tensor r_1}), \ldots,\Gamma(\mathcal{X},\mathcal{I}_{\mathcal{Y}}\tensor L^{\tensor r_n})$, we see that there exists a constant $B>1$ such that any $\Gamma(\mathcal{X},\mathcal{I}_{\mathcal{Y}}\tensor L^{\tensor nq + r_i})$ admits a full rank submodule having a basis whose elements have norm bounded by $B^q$. By \cite{zhang92} Lemma 1.7, the $\ZZ$-module $\Gamma(\mathcal{X},\mathcal{I}_{\mathcal{Y}}\tensor L^{\tensor nq + r_i})$ admits a basis whose elements have norm bounded by $rB^q$, where $r$ denotes the rank of $\Gamma(\mathcal{X},\mathcal{I}_{\mathcal{Y}}\tensor L^{\tensor nq + r_i})$. Since $r$ grows polynomialy in $q$, and $r_1,\ldots,r_n$ form a complete residue system modulo $n$, we conclude that there exists a constant $\kappa >0$ such that, for any sufficiently large integer $d$, the $\ZZ$-module $\Gamma(\mathcal{X},\mathcal{I}_{\mathcal{Y}}\tensor L^{\tensor d})$ admits a basis consisting of elements with norm bounded by $\kappa^d$.
\end{proof}

To handle the case $\mathcal{X}=\PP^n_{\mathcal{O}_K}$, we compare the \emph{height} $\overline{\mathbf{h}}(P)$ of a homogeneous polynomial $P \in \mathcal{O}_K[X_0,\ldots,X_n]$ of degree $d$ used in \cite{philippon86} with the Fubini-Study norm $\|s\|_{\PP^{n}_{\mathcal{O}_K}}$ of the corresponding section $s \in \Gamma(\PP^n_{\mathcal{O}_K},\mathcal{O}(d))$. By definition,
\begin{align*}
  \overline{\textbf{h}}(P) = \frac{1}{[K:\QQ]}\sum_{\sigma:K \into \CC}\log^+M_{\sigma}(P)\text{,}
\end{align*}
where, for any field embedding $\sigma: K \into \CC$, we set
\begin{align*}
M_{\sigma}(P) \defeq \exp \left(\frac{1}{(2\pi)^{n+1}}\int_0^{2\pi} \cdots \int_0^{2\pi} \log |P^{\sigma}(e^{i\theta_0},\ldots,e^{i\theta_n})| d\theta_0\cdots d\theta_n\right)\text{.} 
\end{align*}

\begin{lemma}\label{lemmaheight}
  Let $(\| \ \|_{\sigma})_{\sigma:K \into \CC}$ denote the Fubini-Study Hermitian structure on the line bundle $\mathcal{O}(1)$ over the arithmetic scheme $\PP^n_{\mathcal{O}_K}$. For any integer $d\ge1$, and any section $s \in \Gamma(\PP^n_{\mathcal{O}_K},\mathcal{O}(d))$, if $P \in \mathcal{O}_K[X_0,\ldots,X_n]$ denotes the homogeneous polynomial of degree $d$ corresponding to $s$, then
  \begin{align*}
   \overline{\textbf{h}}(P) \le\log^+ \|s\|_{\PP^n_{\mathcal{O}_K}} + \frac{(n+1)}{2}d\text{.}
  \end{align*}
\end{lemma}

\begin{proof}
  For any field embedding $\sigma:K \into \CC$, and any $(\theta_0,\ldots,\theta_n) \in [0,2\pi]^{n+1}$, we have
\begin{align*}
  \|s_{\sigma}(e^{i\theta_0}:\cdots:e^{i\theta_n})\|_{\sigma} = \frac{| P^{\sigma}(e^{i\theta_0},\ldots,e^{i\theta_n})|}{(n+1)^{\frac{d}{2}}}\text{,}
\end{align*}
so that
\begin{align*}
\log M_{\sigma}(P) \le \log \|s\|_{\sigma,L^{\infty}(\PP^n(\CC))} + \frac{(n+1)}{2}d \le \log \|s\|_{\PP^n_{\mathcal{O}_K}} + \frac{(n+1)}{2}d\text{.}
\end{align*}
Clearly, a similar inequality holds with $\log^+$ in place of $\log$. The result follows by taking the arithmetic mean over all $\sigma:K \into \CC$.
\end{proof}

\begin{proof}[Proof of Theorem \ref{aic}]
  The case where $\mathcal{X}=\PP^n_{\mathcal{O}_K}$ and $\overline{L}$ is given by $\mathcal{O}(1)$ endowed with the Fubini-Study metric follows from Lemma \ref{lemmaheight} and \cite{philippon86} Théorème 2.11 (cf. \cite{nesterenko96} Lemma 2.5).

  The general case follows from this one by considering a closed immersion $i:\mathcal{X} \to \PP^n_{\mathcal{O}_K}$ over $\mathcal{O}_K$ satisfying $i^*\mathcal{O}(1) = L^{\tensor k}$ for some $k\ge 1$, and by applying Lemma \ref{lift} and Remark \ref{rkinv}.
\end{proof}


 
 \section{$D$-property and ZL-density in quasi-projective varieties}\label{appzl}

 Let $k$ be a field, $X$ be a smooth quasi-projective variety over $k$, and $\mathcal{F}$ be an $\mathcal{O}_X$-submodule of rank one of the tangent bundle $T_{X/k}$ such that the quotient $T_{X/k}/\mathcal{F}$ is torsion-free, i.e., a one dimensional (possibly singular) \emph{foliation} on $X$.

 Let $p \in X(k)$ be a $k$-point of $X$. We say that a formal curve $\hat{\varphi}: \Spf k[\![q]\!] \to X$ is an \emph{integral curve} of $\mathcal{F}$ at $p$ if $\hat{\varphi}(0)=p$ and if the image of the tangent map
 \begin{align*}
D\hat{\varphi} : T_{\Spf k[\![q]\!]/k} \to \hat{\varphi}^*T_{X/k}
 \end{align*}
 factors through the subbundle $\hat{\varphi}^*\mathcal{F}$ of $\hat{\varphi}^*T_{X/k}$. Moreover, if $\mathcal{F}(p)\defeq \Gamma(\Spec k, p^*\mathcal{F})$ denotes the fiber of $\mathcal{F}$ at $p$, we say that $\hat{\varphi}$ is \emph{smooth} if $\hat{\varphi}'(0) \defeq D_0\hat{\varphi}(\frac{d}{dq}) \in \mathcal{F}(p)$ is non-zero.

 From now on, we assume that $k$ has \emph{characteristic 0}. By a formal version of the Frobenius Theorem, for every $p\in X(k)$ such that $\mathcal{F}(p)\neq 0$, there exists a unique smooth integral curve $\hat{\varphi}$ of $\mathcal{F}$ at $p$, up to composition by an automorphism of $\Spf k[\![q]\!]$.
 
We say that a closed subscheme $Y$ of $X$ is \emph{$\mathcal{F}$-invariant} if the ideal of $Y$ in $\mathcal{O}_X$ is stable under the derivations of $\mathcal{F}\subset T_{X/k}= \mathcal{D}er_k(\mathcal{O}_X)$.

\begin{defi}\label{dprop}
  Let $X$ be a smooth quasi-projective variety over the field $k$, let $\mathcal{F}$ be a one dimensional foliation on $X$, and let $\hat{\varphi} : \Spf k[\![q]\!] \to X$ be a formal integral curve of $\mathcal{F}$. We say that $\hat{\varphi}$ satisfies the \emph{$D$-property} for $\mathcal{F}$ if there exists a constant $C>0$ such that, for every $\mathcal{F}$-invariant closed subvariety $Y$ of $X$, there exists a Cartier divisor $D$ whose support contains $Y$ satisfying
  \begin{align*}
   \mult_{\hat{\varphi}}D\le C\text{.}
  \end{align*}
 \end{defi}

 Observe that, if $\hat{\varphi}$ satisfies the $D$-property, then its image is Zariski-dense in $X$. Indeed, the Zariski-closure of the image of an integral curve of $\mathcal{F}$ is $\mathcal{F}$-invariant.

 \begin{theorem}[Nesterenko-Binyamini]\label{nesbin}
  Let $X$ be a smooth quasi-projective variety over an algebraically closed field $k$ of characteristic 0, $v\in \Gamma(X,T_{X/k})\minus\{0\}$ be a vector field on $X$, and $\hat{\varphi}: \Spf k[\![q]\!] \to X$ be a smooth formal curve satisfying the differential equation
   \begin{align*}
    q\frac{d\hat{\varphi}}{dq} = v \circ \hat{\varphi}\text{.}
   \end{align*}
 If $\hat{\varphi}$ satisfies the $D$-property for the foliation generated by $v$, then $\hat{\varphi}$ is ZL-dense in $X$.
 \end{theorem}

 Note that $\hat{\varphi}(0)$ is a singular point of $v$. In the non-singular case, i.e., $\hat{\varphi}$ satisfies the differential equation $\frac{d\hat{\varphi}}{dq} = v\circ \hat{\varphi}$, stronger statements are true (cf. \cite{binyamini14} Theorem 2), but an analogous of the above result may be also obtained by virtually the same proof. 
 
 Binyamini's original result (\cite{binyamini14} Corollary 3) builds on ideas of Nesterenko and concerns the case of an analytic integral curve of a polynomial vector field on some affine space over $\CC$.  In what follows, we merely indicate how a slight modification of the geometric methods of Binyamini may be used to prove the Theorem \ref{nesbin} above.

 We start by recasting the $D$-property into a more workable form.
 
 \begin{prop}
Let $\overline{X}$ be any projective compactification of $X$ and $L$ be an ample line bundle on $\overline{X}$. Then, a formal curve $\hat{\varphi} : \Spf k[\![q]\!] \to X$ satisfies the $D$-property for a one dimensional foliation $\mathcal{F}$ on $X$ if and only if there exists a constant $C>0$ such that, for every $\mathcal{F}$-invariant closed subvariety $Y$ of $X$, there exists an integer $d\ge 1$, and a global section $s \in \Gamma(\overline{X},L^{\tensor d})$ vanishing identically on $Y$ such that $\ord_0 \hat{\varphi}^*s \le C$.
\end{prop}

\begin{proof}
  The sufficiency is clear: consider the divisors $\div(s)$.

  Conversely, suppose that $\hat{\varphi}$ satisfies the $D$-property for $\mathcal{F}$ with constant $C>0$ and let $Y$ be an $\mathcal{F}$-invariant closed subvariety of $X$. Since $L$ is ample, we may assume that $Y$ contains $\hat{\varphi}(0)$; otherwise, there exists an integer $d\ge 1$ and a section $s \in \Gamma(\overline{X},L^{\tensor d})$ vanishing on $Y$ such that $s(\hat{\varphi}(0))\neq 0$, so that $\ord_{0}\hat{\varphi}^*s=0\le C$.

  Let $D$ be a divisor whose support contains $Y$ such that $\mult_{\hat{\varphi}}D \le C$, and let $f$ be a local equation for $D$ on some open neighborhood $U$ of $\hat{\varphi}(0)$. Since $L$ is ample, there exists an integer $m\ge 1$ and a section $s_0 \in \Gamma(\overline{X},L^{\tensor m})$ such that $\hat{\varphi}(0) \in \overline{X}_{s_0}$ and $\overline{X}_{s_0}\subset U$. Now, there exists an integer $n\ge 1$, and a global section $s \in \Gamma(\overline{X},L^{\tensor mn})$ such that $s=f s_0^{\tensor mn}$ over $\overline{X}_{s_0}$. It is clear that $s$ vanishes identically on $Y$ and satisfies $\ord_{0}\hat{\varphi}^*s = \ord_0 \hat{\varphi}^*f = \mult_{\hat{\varphi}}D \le C$.
\end{proof}

Consider the hypotheses and notation of Theorem \ref{nesbin}. Fix a projective compactification $\overline{X}$ of $X$, and an ample line bundle $L$ on $\overline{X}$ endowed with a global section $s_0 \in \Gamma(\overline{X},L)$ satisfying $\overline{X}_{s_0}\subset X$. Recall from Section \ref{derivesections} that $v$ defines a $k$-derivation $\partial_v$ on the ring $\bigoplus_{d\ge 0}\Gamma(\overline{X}, L^{\tensor d})$.

Let $p = \hat{\varphi}(0)$. By a \emph{formal cycle} of $X$ at $p$, we mean a cycle in the scheme $\Spec \hat{\mathcal{O}}_{X,p}$, where $\hat{\mathcal{O}}_{X,p}$ denotes the completion of the local ring $\mathcal{O}_{X,p}$ with respect to its maximal ideal. Note that every (global) cycle of $X$ induces, by localization and formal completion, a formal cycle of $X$ at $p$.

Let $Y$ be a prime formal cycle of $X$ at $p$ corresponding to the prime ideal $\pp$ of $\hat{\mathcal{O}}_{X,p}$ and denote by $I_{\hat{\varphi}}$ the ideal of $\im \hat{\varphi}$ in $\hat{\mathcal{O}}_{X,p}$. Assume that $I_{\hat{\varphi}}$ does not contain $\pp$ (i.e., $Y$ does not contain the image of $\hat{\varphi}$). Since $\hat{\varphi}: \Spf k[\![q ]\!] \to X$ is smooth, the image of $I_{\hat{\varphi}}$ in the local ring $\hat{\mathcal{O}}_{X,p}/\pp$ contains some power of the maximal ideal. We may thus consider the \emph{Samuel multiplicity}
\begin{align*}
\mult_{\hat{\varphi}}Y \defeq e_{I_{\hat{\varphi}}/I_{\hat{\varphi}}\cap \pp}(\hat{\mathcal{O}}_{X,p}/\pp)\text{.}
\end{align*}
By additivity, we may extend this definition to every formal cycle of $X$ at $p$ whose components do not contain the image of $\hat{\varphi}$. By abuse of notation, if $Z$ is a (global) cycle of $X$, we denote by $\mult_{\hat{\varphi}}Z$ the multiplicity  $\mult_{\hat{\varphi}}\hat{Z}$ of its completion at $p$.

\begin{prop}
  The multiplicity function constructed above satisfies the following properties:
  \begin{enumerate}
  \item If $Z=\div(f)$, for some $f \in \hat{\mathcal{O}}_{X,p}$, then $\mult_{\hat{\varphi}}(Z) = \ord_0\hat{\varphi}^*f$.
\item If $Z = p$, then $\mult_{\hat{\varphi}}(Z)=1$.
 \item For any closed subvariety $Y$ of $X$, any integer $d\ge 1$, and any $s\in \Gamma(\overline{X},L^{\tensor d})\minus\{0\}$ vanishing identically on $Y$, we have $\mult_{\hat{\varphi}}(Y) \le \ord_0\hat{\varphi}^*s\cdot \mult_p(Y)$.
\item For any closed subvariety $Y$ of $X$, any integer $d\ge 1$, and any $s \in \Gamma(\overline{X},L^{\tensor d})\minus\{0\}$ vanishing identically on $Y$ for which $\partial_vs$ does not vanish identically on $Y$, we have $\mult_{\hat{\varphi}}(Y) \le \mult_{\hat{\varphi}}(Y\cdot \div(\partial_vs))$.
\item There is an integer $n_0\ge 0$ such that, for every closed subvariety $Y$ of $X$ not contained in a $v$-invariant subvariety of $X$, if $d\ge 1$ is the smallest integer for which there is $s \in \Gamma(\overline{X},L^{\tensor d})\minus\{0\}$ vanishing identically on $Y$, then $\min \{n \mid \partial_v^ns\text{ does not vanish identically on }Y\} \le n_0$.
  \end{enumerate}
\end{prop}

Properties (1) and (2) are easy. For properties (3) and (4), see \cite{binyamini14} Lemma 8 and Proposition 9. Finally, property (5) follows by an adaptation of the arguments in \cite{binyamini14} Section 3. 

Once this is established, the proof Theorem \ref{nesbin} becomes completely analogous to the proof of \cite{binyamini14} Theorem 3. 
\end{appendix}

\end{document}

%% file: header-en.tex
\usepackage{amssymb,amsmath,amsthm}
\usepackage{enumerate}
\usepackage{mathtools}

\usepackage[utf8]{inputenc}

\usepackage{color}
\usepackage[colorlinks,linkcolor=black,citecolor=black]{hyperref}

\usepackage{tikz-cd}
\tikzcdset{
arrow style=tikz,
diagrams={>={Straight Barb[scale=0.8]}}
}

\def\AA{\mathord{\mathbf{A}}}
\def\CC{\mathord{\mathbf{C}}}
\def\RR{\mathord{\mathbf{R}}}
\def\QQ{\mathord{\mathbf{Q}}}

\def\ZZ{\mathord{\mathbf{Z}}}

\def\NN{\mathord{\mathbf{N}}}
\def\PP{\mathord{\mathbf{P}}}

\def\pp{\mathord{\mathfrak{p}}}

\def\d{{\rm d}}

\def\ord{{\rm ord}}

\def\minus{\setminus}

\def\Hom{\mathop{\rm Hom}}

\def\im{\mathop{\rm im}}

\def\Im{\mathop{\rm Im}}

\def\Spec{\mathop{\rm Spec}}

\def\Spf{\mathop{\rm Spf}}
\def\Proj{\mathop{\rm Proj}}

\def\div{\mathop{\rm div}}
\def\rk{\mathop{\rm rk}}
\def\adeg{\mathop{\rm \widehat{\deg}}}

\def\trdeg{\text{\rm trdeg}}

\def\an{\text{\rm an}}

\def\mult{\text{\rm mult}}

\def\defeq{\coloneqq}

\def\tensor{\otimes}
\def\into{\hookrightarrow}

\def\to{\longrightarrow}
\def\mapsto{\longmapsto}

\newtheorem{theorem}{Theorem}[section]
\newtheorem{prop}[theorem]{Proposition}

\newtheorem{lemma}[theorem]{Lemma}

\newtheorem{coro}[theorem]{Corollary}

\theoremstyle{definition}
\newtheorem{ex}[theorem]{Example}
\newtheorem{defi}[theorem]{Definition}

\newtheorem{obs}[theorem]{Remark}

\numberwithin{equation}{section}